\documentclass[reqno, 12pt]{amsart}
\usepackage{latexsym,amssymb,amsmath,amsfonts,amsthm,url}
\usepackage{enumitem}
\usepackage[english]{babel}
\usepackage{graphicx}
\graphicspath{%
    {converted_graphics/}
    {/}
}
\usepackage{bbm}
\usepackage[dvipsnames]{xcolor}
\usepackage{hyperref} 
\usepackage[margin = 1 in]{geometry}
\usepackage{amssymb}
\usepackage{mathtools}
\usepackage{mathrsfs}
\usepackage{comment}
\usepackage[numbers,sort&compress,square]{natbib}
\newtheorem{theorem}{Theorem}
\newtheorem{lemma}{Lemma}

\newtheorem{prop}{Proposition}

\newtheorem{corollary}{Corollary}
\theoremstyle{remark}

\newtheorem{remark}{Remark}
\numberwithin{equation}{section}
\numberwithin{theorem}{section}
\numberwithin{lemma}{section}
\numberwithin{prop}{section}
\numberwithin{corollary}{section}

\newcommand{\abs}[1]{\left|   #1   \right|}
\newcommand{\norm}[1]{\left\|   #1   \right\|}
\def\beq{ \begin{equation} }
    \def\eeq{ \end{equation} }

\def\tilde{\widetilde}
\def\hat{\widehat}

\def\ep{\varepsilon}
\renewcommand{\epsilon}{\varepsilon}

\def\square{\vcenter{\vbox{\hrule height .4pt
            \hbox{\vrule width .4pt height 5pt \kern 5pt
                \vrule width .4pt} \hrule height .4pt}}}

\def\sqz{\kern -0.2em}

\def\d{\mathbf{d}}

\def\E{\mathbb{E}}
\def\N{\mathbb{N}}
\def\1{\mathbbm{1}}
\def\P{\mathbb{P}}

\def\S{S}
\def\X{X}
\def\I{I}

\def\ER{Erd\H{o}s-R\'enyi}
\newcommand{\CM}{\mathbb{CM}}
\renewcommand{\bar}{\overline }
\def\cd{\xrightarrow{\mathcal{D}}}

\DeclareMathOperator {\var}{Var}

\title{Susceptible-Infected Epidemics on Evolving Graphs at Critical Infection Rate}

\author[Chen]{Wenze Chen}
\address{School of Mathematics and Statistics, Donghua University, China.}
\email{chenwenze@mail.dhu.edu.cn}

\author[Hou]{Haojie Hou}
\address{School of Mathematics and Statistics, Beijing Institute of Technology, China.}
\email{houhaojie@bit.edu.cn}

\author[Ma]{Ruibo Ma}
\address{School of Mathematics and Statistics, Beijing Jiaotong University, China.}
\email{9908@bjtu.edu.cn}

\author[Yao]{Dong Yao}
\address{School of Mathematics and Statistics/RIMS, Jiangsu Normal University, China. }
\email{dongyao@jsnu.edu.cn}

\begin{document}

\begin{abstract}
    Consider an SI  process on a graph $G$ where each S--I connection becomes I--I at rate $\lambda$. Here S and I stand for ``susceptible'' and ``infected'' respectively.    The evoSI model is a modification of the SI model in which S--I edges are  broken at rate $\rho$ and the ``S'' connects to a randomly chosen vertex. It is proven in Durrett and Yao [2022, Electron. J. Probab.]  that, for the supercritical evoSI process on the configuration model, there exists a quantity $\Delta$ depending on the first three moments of the degree distribution such that the sign of $\Delta$ governs the continuity of the phase transition of the final epidemic size  near the critical infection rate $\lambda_c$.  

    In this paper, we consider the critical evoSI model on the configuration model, i.e., $\lambda=\lambda_c$.  We show that,  if $\Delta>0$, then the probability of a major outbreak starting from a single infected individual is $Cn^{-1/3}(1+o(1))$ for some explicit constant $C>0$, where $n$ is the size of the graph. On the contrary, if $\Delta<0$, then   this probability is $o(n^{-1/3})$. The case $\Delta<0$ is  reminiscent of the critical {\ER}  graphs, where the probability for the size of the largest component to be of order $n$ decays exponentially in $n$.  
\end{abstract}

\maketitle
\section{Introduction}

Susceptible-Infected (SI) and  Susceptible-Infected-Recovered (SIR) models are valuable tools for understanding processes such as information diffusion and diseases spreading in  social networks. 
Mathematically, they correspond to  \emph{interacting particle systems on graphs}, where the set of vertices  represents individuals, and the set of edges reflects the connections between them.  In the SI model, each S--I edge transitions to an I--I edge at a rate $\lambda$, i.e., after an independent exponential time with parameter $\lambda$. Once a vertex enters the infected state I, it remains infected indefinitely. In the SIR model, individuals who become infected transition to the recovered state R at a rate of $\gamma$ and subsequently remain in the recovered state forever. 
Over the past two decades, significant progress  has been made in the rigorous analysis of epidemic models on random graphs. See \cite{Durrett2010} for an overview. 

In recent years, there has been a growing interest in studying epidemic models on \emph{evolving} graphs.  See, e.g., \cite{MR4474536, MR4888144, MR4901638, MR4861781, MR4474532} and the references therein. One type of evolving mechanism, proposed in \cite{jiang2019},  goes as follows.
In addition to the epidemic dynamics, the graph structure itself is  allowed to evolve as follows: S--I edges are broken at a rate $\rho$ and the susceptible individuals affected by the broken edges establish new connections with a uniformly randomly selected individual in the graph. The SI and SIR models on such evolving graphs are  termed \textit{evoSI} and \textit{evoSIR}, respectively.

Of great interest in the mathematical analysis of epidemics is the \emph{final epidemic size} (denoted by $\Lambda_n$ if the underlying graph $G_n$ has size $n$), by which we mean the number of vertices that are eventually infected (in SI epidemics) or recovered (in SIR epidemics). Since the real-world networks are usually huge, we are mostly interested in the large $n$ behavior of $\Lambda_n$. This leads to the concept of a \emph{major outbreak},  by which we mean $\Lambda_n/n>\ep$ for some $\ep>0$ independent of $n$. The critical infection rate $\lambda_c$ is then defined as the smallest $\lambda$ such that a major outbreak occurs with positive probability, 
\begin{equation*}
    \lambda_c:=\inf \left\{ \lambda>0 : \lim_{\ep\to 0}\limsup_{n\to\infty}\P_1(\Lambda_n /n>\ep)>0 \right\},
\end{equation*} 
where the subscript $1$ in $\P_1$ indicates that initially a uniformly chosen vertex is infected.

An interesting feature of epidemic models on evolving graphs is that they may exhibit discontinuous phase transitions in the fraction of the final epidemic size $\Lambda_n$, i.e., as $\lambda \searrow\lambda_c$, conditioned on a major outbreak, $\Lambda_n/n$ does not converge to $0$ in probability. For the evoSIR model on {\ER} random graphs $G(n,\mu/n)$ (a random graph where every pair of vertices is connected with probability $\mu/n$), Jiang et al.\@ \cite{jiang2019} showed that its critical value equals  $ (\gamma +\rho)/(\mu-1)$. Moreover, it was conjectured in \cite{jiang2019} (via simulations) that there exists a discontinuous phase transition in $\Lambda_n$ when $\mu=5$ and $\rho=4$. This conjecture was later proven by Ball and Britton \cite{MR4456028}, which provides a sufficient condition for a discontinuous phase transition for a slightly generalized version of evoSIR on {\ER} graphs. The condition in \cite{MR4456028} was later shown to be also necessary in the work \cite{MR4876392} by Hou and the first and fourth authors of the present paper.

Durrett and Yao \cite{MR4474532} turned to the case of SI dynamics and analyzed the evoSI model on the Newman–Strogatz–Watts (NSW) version of the configuration model $\CM(n,\mathbf{D}_n)$ (see Section \ref{subsec:random_graphs} for its definition). They showed that  there exists an explicit quantity  such that the continuity of the phase transition of $\Lambda_n$ is governed by its sign.

The current paper deals with the evoSI model in the configuration model with the critical infection rate $\lambda=\lambda_c$.  Our main result Theorem \ref{surviprob} below shows that, for the quantity $\Delta$ given in \eqref{Delta},  the probability of a major outbreak for the evoSI model on the configuration model is of order $n^{-1/3}$ if $\Delta>0$; while the outbreak probability is $o(n^{-1/3})$ if $\Delta<0$.  The proof of Theorem \ref{surviprob}  demonstrates that the critical behavior of the evoSI model can be characterized by three phases: a near-critical random walk phase, a diffusion phase, and a deterministic growth phase. We expect the analysis in this paper to be useful for studying other stochastic processes on random graphs, especially when one can couple the process and the graph structure together.

\emph{Notational convention}:
throughout the article, we write  $a_n=O(b_n)$ if there is a constant $C>0$, s.t$.$ $\abs{a_n}\leq C \abs{b_n}$ for all large $n$, and $a_n=o(b_n)$ if $\lim_{n\to\infty} \abs{a_n/b_n}=0$. We say that $a_n=\Theta(b_n)$ if $a_n=O(b_n)$ and $b_n=O(a_n)$.

\subsection{Statement of the main result}

In this paper, we consider the evoSI process in the configuration model. Given a non-negative integer $n$ and a sequence of (deterministic) degrees $d_{1,n},\ldots, d_{n,n}$, we attach $d_{1,n},\ldots,d_{n,n}$ half-edges to $n$ vertices labeled $1, 2,\ldots, n$, respectively, before pairing these half-edges uniformly at random. Two paired half-edges form an edge in $G_n$. Letting $\mathbf{d}_n$ denote the degree sequence,  we call the generated graph the configuration model on $n$ vertices with the prescribed degree sequence $\mathbf{d}_n$ and denote it by $\CM(n,{\mathbf{d}_n})$. This model is also called the Molloy–Reed configuration graph. For more details on the configuration model, we refer the reader to  \cite[Chapter 7]{MR3617364}.
 
We will impose the following regularity conditions on $\mathbf{d}_n$. Define
\begin{equation*}
    p_{k,n}:= n^{-1}\sum_{i=1}^{n}\1[d_{i,n}=k], \quad k\in \mathbb{N}:=\{0,1,2,\ldots\}.
\end{equation*}
Our standard Assumption {\bf (H)}  is as follows:
\begin{enumerate}
\item[{\bf (H1)}]    (Finite exponential moment) There exist two positive constants $\eta_{\ref{exp1}} $ and $C_{\ref{exp1}}$ such that   for all sufficiently large $n$, 
    \begin{equation}\label{exp1}
        \frac{1}{n} \sum_{i=1}^n e^{ \eta_{\ref{exp1}}d_{i,n}}\leq C_{\ref{exp1}}.
    \end{equation}

    \noindent 
\item[{\bf(H2)}]   (Convergence of the degree sequence) There exists a probability distribution $\{p^*_{k},k\in \mathbb{N}\}$ such that
    \begin{equation*}
        \lim_{n\to\infty}p_{k,n}=p_k^*,\quad \forall k\in \mathbb{N}.
    \end{equation*}  
 \item[{\bf(H3)}]  (Concentration of the degree sequence) There exists a constant  $C_{\ref{concern1}}\in (0, 2/3)$ such that  for  all sufficiently large  $n$,
        \begin{equation}\label{concern1}
        \sum_{k\in \mathbb{N}}(k+1)^4\abs{np_{k,n}-np^*_{k}}\leq n^{C_{\ref{concern1}}}.
    \end{equation}
\end{enumerate}

\begin{remark}[Bounds for the tail and maximum degree]\label{rem:bdpk}
    Assumption {\bf (H1)}   implies that for sufficiently large $n$, 
    \begin{equation}\label{bdpkn}
        p_{k,n} \leq C_{\ref{exp1}} \exp(-\eta_{\ref{exp1}}k), \quad \forall k\in \mathbb{N}.
    \end{equation}
    Combining this with {\bf (H2)}, we see that 
    \begin{equation}\label{bdpk}
        p^*_{k}\leq C_{\ref{exp1}}\exp(-\eta_{\ref{exp1}}k), \quad \forall k \in \mathbb{N}.
    \end{equation}
    Also, from {\bf (H1)}, the maximal degree
    $
        D_{n,\mathrm{max}}:=\max _{1\leq i\leq n}   d_{i,n}
    $
    must satisfy
    \begin{equation}\label{dmaxbd}
        D_{n,\mathrm{max}}\leq \frac{\log(nC_{\ref{exp1}})}{\eta_{\ref{exp1}}}.
    \end{equation}
\end{remark}

\begin{remark}[On the verification of {\bf (H3)}]\label{rem:2}
    Suppose that {\bf (H1)} and {\bf (H2)} hold. Then it follows from \eqref{bdpkn} and \eqref{bdpk} that for all sufficiently large $n$, 
    \begin{align*}
        & n \sum_{k > \log n/ \eta_{\ref{exp1}} }(k+1)^4 \abs{p_{k,n}-p^*_{k}} \leq  n \sum_{k> \log n/   \eta_{\ref{exp1}} }(k+1)^4 \max \{p_{k,n},p^*_{k}\}  \\
        \leq &  C_{\ref{exp1}} n \sum_{k > \log n/ \eta_{\ref{exp1}} }(k+1)^4 \exp(-\eta_{\ref{exp1}}k) = \frac{C_{\ref{exp1}}}{\eta_{\ref{exp1}}^4}  O\left( \log^4 n\right),
    \end{align*}
    where the implied constant in $O(\log^4n)$ is universal. Therefore, given {\bf (H1)} and {\bf(H2)}, we can find a constant $C_{\ref{equiva_cond}} = C_{\ref{equiva_cond}}(\eta_{\ref{exp1}})> 0$ such that a sufficient condition for {\bf (H3)} is
    \begin{equation}\label{equiva_cond}
        \abs{np_{k,n}-np^*_{k}}\leq C_{\ref{equiva_cond}} n^{C_{\ref{concern1}}}/\log^5 n
        \quad    \mbox{ for all } 0\leq k\leq \log n/ \eta_{\ref{exp1}}.
    \end{equation}

\end{remark}

For $r\geq 1$, let $m_r:= \sum_{k\in \mathbb{N}} k^r p^*_k$ be the $r$-th moment of the distribution $p^*_{\cdot}$. Using the same approach as in \cite{MR4474532}, under {\bf (H1)} and {\bf(H2)}, the critical infection rate of the evoSI model on the configuration model $\CM(n,\d_n)$ is given by
\begin{equation}\label{eq:crivalue}
    \lambda_c:=\frac{\rho m_1}{m_2-2m_1}.
\end{equation}
Our main result is stated in the following theorem. 
    
\begin{theorem}\label{surviprob} 
    Let $\{\mathbf{d}_n,n\geq 1\}$ be degree sequences that satisfy {\bf(H)}. Define 
    \begin{equation}\label{Delta}
        \Delta:= -\frac{m_3 -3m_2 + 2m_1}{m_1} + 3 (m_2-2m_1).
    \end{equation}
    Consider evoSI on the random graph $\CM(n,\d_n)$ with $\lambda=\lambda_c$. 
    \begin{enumerate}
        \item If $\Delta>0$, then there exist two constants  $C_{\ref{eq:del>0case}},\ep_0>0$ such that for any fixed $0<\ep<\ep_0$,  
        \begin{equation}\label{eq:del>0case}
            \lim_{n\to\infty} n^{1/3}\P_1(\Lambda_n/n>\ep)=C_{\ref{eq:del>0case}}.
        \end{equation}
        
        \item If $\Delta<0$, then for any $\epsilon>0$,
        \begin{equation}\label{eq:del<0case}
            \lim_{n\to\infty} n^{1/3}\P_1(\Lambda_n/n> \epsilon)=0.  
        \end{equation}
    \end{enumerate}
\end{theorem}

\begin{remark}
    The constant $C_{\ref{eq:del>0case}}$  is given by \eqref{qto00}. Under $\mathbb{P}_1$, the initial infection is assigned to a randomly chosen vertex in the graph $\CM(n,\mathbf{d}_n)$---and crucially, $\mathbb{P}_1$ explicitly incorporates two sources of randomness: the inherent randomness of the graph $\CM(n,\mathbf{d}_n)$ itself, and the randomness of the initially selected infected vertex. 
\end{remark}

\subsection{Results for  random graphs with a random degree sequence}\label{subsec:random_graphs}
We note that
the convergence rate in Theorem \ref{surviprob} is actually \emph{uniform} over the degree sequence. In other words, given $\{p_k^*\}$ (with $\Delta>0$), 
for any $\delta>0$, there exists $N$ (which  depends only on $\{p_k^*\}$, the constants $C_{\ref{exp1}}, C_{\ref{concern1}}$ and the error parameter $\delta$), such that for all $n\geq N$  and $\d_n$ satisfying {\bf(H)}, we have
$$
\abs{n^{1/3}\P_1(\Lambda_n/n>\epsilon)-C_{\ref{eq:del>0case}}}\leq \delta.
$$
Similarly, if $\Delta<0$, then we have 
$$
n^{1/3}\P_1(\Lambda_n/n>\epsilon)\leq \delta.
$$
    In this section we give two examples of random graphs with random degree sequences   satisfying {\bf(H)} with high probability. Using the uniform version of  Theorem  \ref{surviprob}  just mentioned, we see that the limits \eqref{eq:del>0case} and \eqref{eq:del<0case} apply to them as well.

\emph{Example 1: NSW random graph}.
     Let $D$ be a non-negative integer-valued random variable satisfying         $\E [\exp(cD)] <\infty$ for some constant $c>0$.  Let $\mathbf{D}_n:=\{D_{1,n}, 
   \dotsc, D_{n,n}\}$ be i.i.d.\@ sampled from the distribution of $D$. Increase $D_{1,n}$ by 1 if the sum $\sum_{i=1}^n D_{i,n}$ is odd.  The NSW model is then constructed as the configuration model $\CM(n, \mathbf{D}_n)$ with the random degree sequence $\mathbf{D}_n$. 
Now for each $k$, the (random) empirical degree distribution
$p_{k,n}$  converges to the $p^*_k:=\P(D=k)$.
 Set $   \eta _ {\ref{exp1}} :=c/2$ and $C_{\ref{exp1}} :=2(\E[\exp(cD)])^ {1/2}$, then by Chebyshev's inequality, Assumption {\bf (H1)} holds with  probability at least $1-1/n$. Moreover, combining Lemma \ref{chernoff}  and the union bound,  for any fixed  $C_{\ref{concern1}}>1/2$, 
\begin{align*}
    & \P\left(\exists \, 0\leq k\leq \frac{\log n}{   \eta _ {\ref{exp1}}}  \text{, such that } \abs{np_{k,n}-np^*_{k}}>\frac{ C_{\ref{equiva_cond}} n^{C_{\ref{concern1}}}}{\log^5 n} \right) \\
    \leq & 2 \left(   \frac{\log n}{   \eta _ {\ref{exp1}}}+1 \right) \exp\left(-\frac{C_{\ref{equiva_cond}}^ 2    n^{2C_{\ref{concern1}}-1}}{ 10\log^{10} n}\right)=o(n^{-1}),
\end{align*}
where the constant $C_{\ref{equiva_cond}}$ is given in \eqref{equiva_cond}  in Remark \ref{rem:2}.
Thus, by Remark \ref{rem:2},  {\bf(H)} holds for $\d_n$ with probability $ \geq 1 - 2/n$. Thus  the complement of this event has a probability of $o(n^{-1/3})$, and   Theorem \ref{surviprob} is applicable to  $\CM(n, \mathbf{D}_n)$.

\emph{Example 2: {\ER} graph $G(n,\mu/n)$}\footnote{We can construct an {\ER} graph $G(n,\mu/n)$ as follows. First sample its degree sequence $\d_n=\{d_{i,n}, 1\leq i\leq n\}$ by using $n(n-1)/2$ i.i.d.\@ Bernoulli random variables with mean $\mu/n$. Then we construct the configuration model $\CM(n, \d_n)$ and condition on it to be simple. The resulting graph has the same law as $G(n,\mu/n)$. By adapting the argument in \cite[Theorem 2.7]{MR3275704}, conditioning on simple graphs does not change the conclusion of Theorem \ref{surviprob}.}.
 In {\ER} graph $G(n,\mu/n)$, each pair of vertices is connected independently  with probability $\mu/n$ for some fixed $\mu>1$.  Now the empirical degree distribution converges to $\textnormal{Poisson}   (\mu)$, i.e., $p^*_k=e^{-\mu} \cdot \mu^k / ({k!})   $ for all $ k \geq 0$.  There exist two constants $C _ {\text{\ref{poicase}a}}$ and $  C _ {\text{\ref{poicase}b}}$ such that  
    \begin{equation}\label{poicase}
            \E\left(         \frac{1}{n}    \sum_{i=1}^n \exp(d_{i,n} ) \right)\leq C _ {\text{\ref{poicase}a}} \textnormal{ and } \var\left(         \frac{1}{n}    \sum_{i=1}^n \exp(d_{i,n} )  \right)\leq \frac{ C _ {\text{\ref{poicase}b}} }{n}.
    \end{equation}
For the same reason as in the previous example, {\bf(H1)} holds with probability at least $(1-1/n)$ for  $\eta _ {\ref{exp1}}=\frac{1}{2}$ and some large constant $C_{\ref{exp1}}$. Let
$\tilde{p}_{\cdot,n}$ denote the probability mass function  of $B   (n-1, \mu/n)$.  Using   the rate of  convergence of the binomial distribution to the Poisson distribution (see \cite[Theorem 2]{MR56861}), there exists a constant $C _{\ref{biopoi}}$ such that
   \begin{equation}\label{biopoi}
        \sum_{k=0}^{\infty}\abs{\tilde{p}_{k,n}-p^*_k}\leq \frac{ C _{\ref{biopoi}} }{n}
   \end{equation}
for all $n >0$. By writing $np_{k,n}$ as a sum of indicators, there exists a constant $C_{\ref{4pex}} >0$ such that
\begin{equation}
    \label{4pex}
    \E [(np_{k,n}-n\tilde{p}_{k,n})^4 ]\leq C_{\ref{4pex}}n^2 
\end{equation}
for all $n$ and $k$. Consequently, for some constant 
   $C_{\ref{p-ptilde}}>0$, we have
    \begin{equation}
        \label{p-ptilde}
        \P(\abs{np_{k,n}-n\tilde{p}_{k,n} }\geq  C_{\ref{equiva_cond}} n^{0.6}/\log^5 n  ) \leq  C_{\ref{4pex}}   n^{-0.39}
    \end{equation}
    for large $n$.  Combining \eqref{biopoi}, \eqref{p-ptilde} and Remark \ref{rem:2}, we see that, with the choice of $C_{\ref{concern1}}=0.6$, {\bf(H)} holds true with probability at least $(1 - n^{ -0.38 })$ for $n$ large. A direct computation gives $\Delta :=2\mu^2-3\mu$. By Theorem \ref{surviprob},  \eqref{eq:del>0case} holds if $\mu>3/2$, and \eqref{eq:del<0case} is valid if $\mu<3/2$.

\subsection{Related work}

\subsubsection{The size of components of {\ER} graphs in the critical window}\label{sec:er}
Consider the {\ER} graph $G_n=G(n, n^{-1}(1+u n^{-1/3}))$, where $u$ is a real parameter. Let $B_t$ be a standard Brownian motion, and define the process 
\begin{equation*}
    W(s;u):=B_s+us-\frac{s^2}{2} \textnormal{, }s\geq 0.  
\end{equation*}
It is shown by Aldous in \cite{MR1434128} that the exploration process, or the breadth-first walk, on $G_n$ converges to $W(s;u)$ after appropriate scalings in  space and time. Let $W^R(s;u)$ be the reflection of $W(s;u)$ at $0$, i.e., 
\begin{equation*}
    W^R(s;u):=W(s;u)-\min_{0\leq s'\leq s} W(s';u),
\end{equation*}
and let $\mathcal{C}_j$ be the size of the $ j $-th largest component of $G_n$. Aldous also proved that the  sequence $\{n^{-2/3} \mathcal{C} _ j\}^\infty _{j = 1}  $ converges in distribution to the ordered excursions of $W^R(\cdot, u)$ (called the  multiplicative coalescent) in the space\footnote{The space $    l^2_{\searrow}$ is the  set
    \begin{equation*}
     \{(x_1,x_2,\ldots): x_1\geq x_2\geq \cdots,   \textnormal{ all } x_n \geq 0,\sum x_i^2<\infty \}
    \end{equation*}
     endowed with the metric $ d ((x_1,x_2,\dotsc ), (y_1,y_2,\dotsc)) =\sqrt{\sum_{i=1}^\infty (x_i-y_i)^2}$. }  $l^2_{\searrow}$.    A comparison between the results of \cite{MR1434128} and ours is provided in Remark \ref{comaldous} below.

    \subsubsection{evoSI model in the supercritical regime} 
 Let $D$ be a  non-negative integer-valued random variable  and  consider the NSW random graphs $\CM(n, \mathbf{D}_n)$ sampled  from $D$, mentioned in the second example in Section \ref{subsec:random_graphs}. 
 Durrett and Yao \cite{MR4474532}  analyzed  the evoSI model   in the supercritical regime. They identified the critical rate $\lambda_c$ in \eqref{eq:crivalue}.   
Recall the definition of $\Delta$ in \eqref{Delta} with $p_k^*=\P(D=k)$.  Then \cite{MR4474532} proved that, if  $\Delta>0$, then there exists a discontinuous phase transition for  $\Lambda_n$. For some $\ep_0>0$ and some $\delta_0>0$, 
        \begin{equation*}
            \lim_{n\to\infty}\P_{1}(\Lambda_n/n>\ep_0\mid\textnormal {a major  outbreak occurs})=1 \text{ for all }\lambda_c<\lambda< \lambda_c + \delta_0 \text{.}
        \end{equation*}    
On the other hand,   if $\Delta <0$, then the phase transition is always continuous. For any $\ep>0$,
       \begin{equation*}
            \lim_{\lambda \searrow \lambda_c}\limsup_{n\to\infty}\P_{1}(\Lambda_n/n>\ep )=0.
        \end{equation*}  

    

\subsubsection{SI-$\rho$ and  SIR-$\rho$ model}

Ball and Britton \cite{MR4456028} considered a generalized version of the evoSI model, called 
 SI-$\rho$ model, with an additional parameter  $\alpha\in [0,1]$. In this model,  susceptible vertices  break their connections with their infected neighbors at a rate $\rho$. Then, with probability $\alpha$,  a susceptible vertex rewires the old edge to a randomly chosen vertex; with probability $(1-\alpha)$, the edge is simply removed from the graph.  They proved that the phase transition in the SI-$\rho$ model on an {\ER} graph $G(n,\mu/n)$  is discontinuous if and only if $\alpha>1/3$ and $\mu>3\alpha/(3\alpha-1)$. In fact, \cite{MR4456028} showed that $\Lambda_n$ converges to a solution to some explicit equation. For  $\alpha=1$ (where the model reduces to the evoSI model), the aforementioned condition is $\mu>3/2$, which coincides with \cite{MR4474532}  (with $D$ distributed as $ \text{Poisson}(\mu)$).

If we also allow infected vertices to recover at  a rate $\gamma$ in evoSI or SI-$\rho$ model, then we get the so-called  
\emph{evoSIR} or \emph{SIR-$\rho$} model. The evoSIR model on  $G(n,\mu/n)$ is discovered to (potentially) have a discontinuous transition in \cite{jiang2019}, where the authors also identified the critical value to be $(\rho+\gamma)/(\mu-1)$. Later, \cite{MR4456028} showed that, the condition
    \begin{equation}\label{cond_discont}
        \rho(2\alpha-1)>\gamma \quad \text{and}\quad \mu>\frac{2\rho \alpha }{\rho(2\alpha-1)-\gamma},
    \end{equation} 
    is sufficient for a discontinuous
     phase transition of $\Lambda_n$. Chen, Hou and Yao proved in   \cite{MR4876392}  that \eqref{cond_discont} is actually also necessary for the discontinuity.


We conjecture that, similarly to Theorem \ref{surviprob},   the probability that the SIR-$\rho$ model  on {\ER} graphs at its critical infection rate has a major outbreak is  $Cn^{-1/3}(1+o(1))$ when \eqref{cond_discont} holds, and  $o(n^{-1/3})$ otherwise. We leave this problem for future investigations. 


\subsection{Proof of Theorem \ref{surviprob}: an outline}

In this section, we will explain the main ideas and key steps in the proof of Theorem \ref{surviprob}.

\subsubsection{Bounding evoSI by avoSI and AB-avoSI models}\label{sssec:evoavo}

Inspired by the approach of \cite{MR4474532}, we bound the number of infected vertices in evoSI between two other models, the avoSI model and the AB-avoSI model. These two models are Markov processes and can be defined by coupling the graph exploration process with the epidemic. For $1 \leq i \leq n$, vertex $i$ has $d_{i,n}$ half-edges and they are all unpaired at time $0$. This resembles the construction of the configuration model. These half-edges, however, remain unpaired until a later random time. We say that a half-edge is infected if it is attached to an infected vertex. In the avoSI model, the following two types of random events may happen. 
\begin{itemize}
    \item At rate $\lambda$, each infected half-edge $h$ pairs with another randomly chosen half-edge. If the vertex $y$ connected to the other half-edge is susceptible, then it becomes infected. The two paired half-edges form an edge. Note that if vertex $y$ changes from state S to I, then all half-edges connected to $y$ become infected half-edges.
    
    \item Each infected  half-edge is removed from its vertex at rate $\rho$ and immediately attaches to a randomly chosen vertex. Note that the state of the half-edge changes to the state of the newly attached vertex.
\end{itemize}

As for the AB-avoSI model, we need two indices for each half-edge $h$: the infection index $A(h,t)$ and rewiring index $B(h,t)$. 
\begin{itemize}
    \item The infection index  $A(h,t)=0$ if $h$ has not been infected by time $t$. If $h$ first becomes  infected  at time $s$, we set $A(h,t)=s$ for  $t\geq s$.
    
    \item The rewiring index $B(h,t)=0$ if the half-edge $h$ has not been rewired by time $t$. If $h$ attaches to a random vertex at time $s$, we update $B(h,t)$ to $s$. If $\tau_m(h)$ denotes the time when $h$ reattaches for the $m$-th time and $\tau_0(h)=0$, then $B(h,t)=\sup \{ \tau_m(h) \mid \tau_m(h)\leq t\}$.  
\end{itemize}
    
The evolution rule of AB-avoSI is different from that of avoSI. In particular, if infected half-edge $h$ pairs with another half-edge $h'$ at time $t$, $h$ may infect $h'$ only when $B(h',t)<A(h,t)$. Otherwise, $h'$ and its vertex $x$ remain in their state prior to the pairing. See \cite[Sections 3.1 and 4.1]{MR4474532} for more details on the definition of avoSI and AB-avoSI as well as their comparisons to evoSI.
    
Lemmas 3.2 and 4.1 in \cite{MR4474532} showed that the final epidemic size of avoSI stochastically dominates that of evoSI, which in turn dominates that of AB-avoSI. Moreover, Lemma 1.2 in \cite{MR4474532} showed\footnote{The paper \cite{MR4474532} considered the NSW version of the configuration model, but Lemmas 1.2, 3.2 and 4.1 there  hold true for the $\CM(n,\mathbf{d}_n)$ model here \emph{mutatis mutandis}.} that the three processes have the same critical infection rate $\lambda_c$. Hence, it suffices to prove the following two theorems to establish Theorem \ref{surviprob}. Recall  that $\Delta$ is defined in \eqref{Delta}.
\begin{theorem}\label{thm:AB}
   Consider the avoSI model with $\lambda=\lambda_c$ on the graph $\CM(n,\d_n)$. Suppose that $\d_n$ satisfies {\bf(H)}.  Let $\widehat{\Lambda}_n$ be the final epidemic size in the avoSI model.  
  \begin{enumerate}
      \item If $\Delta>0$, then there exists an $\ep_0>0$
      such that for $0 <\ep<\ep_0$,       
        \begin{equation}\label{eq:t1.2}
            \lim_{n\to\infty} n^{1/3}\P_1(\widehat{\Lambda}_n/n>\ep)=C_{\ref{eq:del>0case}},  
        \end{equation}
        where $C_{\ref{eq:del>0case}}$ is the same constant as the one in \eqref{eq:del>0case}.
        \item If $\Delta<0$, then for any $\ep>0$,
      \begin{equation*}
            \lim_{n\to\infty} n^{1/3}\P_1(\widehat{\Lambda}_n/n>\ep)=0
            \textnormal{.}  
        \end{equation*}
  \end{enumerate}
\end{theorem}

\begin{theorem}\label{thm:avo}
    Theorem \ref{thm:AB} also holds true for the  AB-avoSI model at the critical rate $\lambda_c$.
\end{theorem}

\subsubsection{The three-stage analysis of the avoSI model}\label{sssec:threestage}

 We will focus on the avoSI model and  give a detailed proof of Theorem \ref{thm:AB}. The analysis of the AB-avoSI model is similar and will be sketched in Section \ref{sec:ABavo}. Let $X_t$ be the number of half-edges in the model at time $t$, and let $X_{I,t}$ be the number of infected half-edges at time $t$. As in \cite{MR4474532} and \cite{MR3275704}, we apply a time-change by multiplying all transition rates by $( X_t -1 ) / (\lambda X_{I,t})$. This time change simplifies the analysis while keeping the final epidemic size $\widehat{\Lambda}_n$ intact. 

From now on, we focus on the avoSI model after the time-change. We will slightly abuse the notation by 
using the same symbols for the avoSI process before and after the time-change.  Since the avoSI process stops after there are no infected half-edges, i.e., $X_{I,t}=0$,  it is natural to investigate the evolution of $\X_{I,t}$. In particular, consider the time
\begin{equation*}
    \gamma_n := \inf\{t\geq 0 \mid X_{I,t}=0\}.
\end{equation*}

Let $\I_t$ and $S_t$ be the number of infected vertices and susceptible vertices at time $t$. Define $X_{S,t}$ and $S_{t,k}$ as the number of susceptible half-edges and susceptible vertices with $k$ half-edges at time $t$, respectively. Note that $n$, the size of the graph $\CM(n,\mathbf{d}_n)$, is omitted from these expressions. 

We will see in \eqref{-1xit0} that the evolution of $X_{I,t}$ is given by
    \begin{equation}\label{eq:00x}
        \mathrm{d}X_{I,t} = \left( -2(X_t - 1) + \sum_{k=0}^\infty k^2 S_{t,k} - \frac{\rho}{\lambda} \cdot \frac{X_t - 1}{n} S_t \right) \mathrm{d}t + \mathrm{d}M_{I,t},
    \end{equation}
where $M_{I,t}$ is a martingale.   Our analysis will show that the term in the large bracket on the right-hand side of \eqref{eq:00x}, called the \emph{drift term}, is roughly  $m_1\Delta n \mathrm{d} t$, while the order of $\mathrm{d}M_{I, t}$ is $\sqrt{Cn \mathrm{d} t}$ for some constant $C>0$. Therefore, we opt to analyze the model in the following three stages.
\begin{enumerate}[label = (\roman*)]
    \item The time $0 \leq t\leq qn^{-1/3}$ for some small constant $q >0$.  During this stage,
    the martingale term is dominant on the right-hand side of \eqref{eq:00x}.
    \item The time $qn^{-1/3}\leq  t\leq Qn^{-1/3}$ for some large constant $Q$. During this stage, the drift term and the martingale term have the same order of influence as $n\to  \infty$.
    \item The time $Qn^{-1/3}\leq t\leq c$ for another small constant $c>0$. During this stage, the drift term is dominant on the right-hand side.
\end{enumerate}

Define
\begin{equation}\label{defN_q}
    \sigma^2 := \frac{\sum_{k \geq 0} k(k-2)^2 p^*_k}{(1 + \rho/\lambda)m_1} + \frac{\rho}{\rho + \lambda_c}.
\end{equation}
In the first stage, we study the tail probability of $\P_1(\gamma_n>qn^{-1/3})$,  stated as follows.

\begin{theorem} \label{YZLimit}
Fix any $0<q<1$. There exist two positive constants $c_{q,Y}, c_{q,Z}$ such that
        \begin{equation}\label{eq:limcqyz}
          c_{q,Z}\leq \liminf_{n\to\infty}n^{1/3}\P_1(\gamma_n>qn^{-1/3}) \leq\limsup_{n\to\infty}n^{1/3}\P_1(\gamma_n>qn^{-1/3})\leq 
          c_{q,Y} 
        \end{equation}
    and that 
    \begin{equation*}
      \lim_{q\to 0} c_{q,Y} q^{1/2}= \lim_{q\to 0} c_{q,Z}q^{1/2}= \left( \frac{\pi \sigma^2}{2m_1} \left(1+\frac{\rho}{\lambda_c}\right) \right)^{-1/2}. 
  \end{equation*}
\end{theorem}

Let $B_1^+$ be the Brownian meander at time $1$ with density function $x\exp(-x^2/2) 1_{[0, \infty)}(x)$. 
Define
\begin{equation}\label{eq:realnq}
    N_q:= \left( 1+ \frac{\rho} {\lambda} \right)m_1qn^{2/3}.
\end{equation}
Let $\cd$ denote `convergence in distribution'.
Theorem \ref{YZLimit} leads to the following corollary.
 \begin{corollary}\label{phase1}
 There exist two random variables $\chi_q$ and $\zeta_q$  such that $\chi_q, \zeta_q \cd  B^+_1$ as $q\to 0$, and that  for any bounded, continuous and monotonically increasing function $f: \mathbb {R}\to [0, \infty)$, 
\begin{equation}\label{ineq1}
\begin{split}
     \frac{c_{q,Z}}{c_{q,Y}}  \E(f(\zeta_q)) &    \leq   \liminf_{n\to\infty}\E_1\left( f\left(\frac{X_{I,qn^{-1/3}}}{\sigma \sqrt{N_q}  } \right) | \gamma_n>qn^{-1/3} \right)\\ &\leq    \limsup_{n\to\infty}\E_1\left( f\left(\frac{X_{I,qn^{-1/3}}}{\sigma \sqrt{N_q} }\right) | \gamma_n>qn^{-1/3} \right) \leq  \frac{c_{q,Y}}{c_{q,Z}}\E(f(\chi_q)).
   \end{split}
\end{equation}
 \end{corollary}

 The appearance of the Brownian meander should not be surprising, since  it  is the limit of a zero-mean random walk  conditioned to stay positive over a large time interval (see, e.g., \cite{MR415702} and \cite{MR362499}).

For the second stage, we show that the behavior of $\X_{I,t}/n^{1/3}$ can be approximated by a diffusion process.  
  Define a function
\begin{equation}\label{deff1xq}
F_1(x,q):=\P\left(\sigma x   \sqrt{(1+\rho/\lambda)m_1q}+\frac{m_1\Delta}{2}(s^2-q^2)+C_{\ref{deff1xq}}B_{s-q}> 0, \forall s\geq q\right),
\end{equation}
where
\begin{equation*}
C_{\ref{deff1xq}}:= \sqrt{m_3-4m_2+4m_1+\rho m_1/\lambda_c}=\sqrt{m_3-3m_2+2m_1},
\end{equation*}
and $\{B_s, s\geq 0\}$ is a standard Brownian motion. 
Note that if $\Delta<0$, then $F_1(x,q) \equiv  0$.

\begin{theorem}\label{phase2}
        For any $0<q<1$,
we have
\begin{align*}
    \frac{c_{q,Z}}{c_{q,Y}} \E(F_1(\zeta_q,q)) &  \leq \liminf_{Q\to \infty}\liminf_{n\to\infty}
    \P_1\left(X_{I,Qn^{-1/3}}\geq m_1 \Delta Q^2 n^{1/3}/4 \middle| \gamma_n>qn^{-1/3} \right)\\ & \leq \limsup_{Q\to\infty}\limsup_{n\to\infty}
    \P_1(X_{I,Qn^{-1/3}}\geq m_1 \Delta Q^2 n^{1/3}/4|\gamma_n>qn^{-1/3})
    \\ 
    & \leq  \frac{c_{q,Y}}{c_{q,Z}} \E(F_1(\chi_q,q)) \text{.}
    \end{align*} 
  In particular, for the case $\Delta<0$, we have $F_1(x,q)\equiv 0$, and thus
  \begin{equation}
      \lim_{Q\to\infty}\limsup_{n\to\infty}
    \P_1(X_{I,Qn^{-1/3}}\geq m_1 \Delta Q^2 n^{1/3}/4 \mid \gamma_n>qn^{-1/3})=0.
  \end{equation}
    \end{theorem}
    Once the process enters the third stage $\gamma_n>Qn^{-1/3}$, which is only asymptotically (as $Q\to \infty$) possible if $\Delta>0$,  we show that there will be a major outbreak with high probability.
    \begin{theorem}\label{survivalcond}
    Consider the case $\Delta>0$. Then there exists a constant $C_{\ref{const3}}>0$ such that
        \begin{equation}\label{const3}
\lim_{Q\to\infty}        \liminf_{n\to\infty}\P_1\left(\gamma_n>C_{\ref{const3}} \mid  X_{I,Qn^{-1/3}}> \frac{m_1\Delta Q^2 n^{1/3}}{4} \right)=1. 
    \end{equation}
    \end{theorem}
    
Theorem \ref{thm:AB} follows by combining Theorem \ref{YZLimit}, Corollary \ref{phase1}, Theorem \ref{phase2} and Theorem \ref{survivalcond}, as proven in Section \ref{sec:pfthm1}. Theorem \ref{thm:avo} can be proven in the same way. Theorem \ref{surviprob} then follows from Theorems \ref{thm:AB} and \ref{thm:avo}.

\begin{remark}\label{comaldous}
    As we discussed in Section \ref{sec:er}, we may compare the result in \cite{MR1434128} regarding critical {\ER} graphs with ours on the  critical evoSI model. The breadth-first walk process there is replaced by $X_{I,t}$, and the assertions of Theorem \ref{phase2} and Lemma \ref{lem:sde} below   are akin to the convergence to $W^R(s;u)$ in {\ER} case. In some sense, the critical {\ER} graph behaves similarly to the case $\Delta<0$. Indeed, the size of the giant component for critical {\ER} graph is of order $O(n^{2/3})$, which is reminiscent of the fact that the critical evoSI process cannot enter the third phase if $\Delta<0$. Moreover, the probability of having a component of size $\Theta(n)$  for critical {\ER} graphs decays exponentially in $n$ (see \cite[Theorem 5.4]{MR3617364})---a decay rate that is evidently $o(n^{-1/3})$. 
\end{remark}

\subsubsection{Organization}

The rest of the paper is organized as follows. In Section \ref{sec:pfthm2} we prove Theorem \ref{YZLimit} and Corollary \ref{phase1}. In Section \ref{sec:pflem1} we establish some upper and lower bounds on $X_{I,t}$ that are needed later. Then we will prove Theorems \ref{phase2} and \ref{survivalcond} in Sections \ref{sec:pfthm3} and \ref{sec:pfthm4}, respectively. Theorem \ref{thm:AB} for the avoSI model is proved in Section \ref{sec:pfthm1}. The discussion of the AB-avoSI case is presented in Section \ref{sec:ABavo} and the proof of some technical lemmas is deferred to Appendix \ref{sec:tech}.

\section{Proof of Theorem \ref{YZLimit} and Corollary \ref{phase1}}\label{sec:pfthm2}

In this section, we construct two auxiliary stochastic processes $(Y_{\ell})_{\ell \in \mathbb{N}}$ and $(Z_{\ell})_{\ell \in \mathbb{N}} $. They will be our focus as we prove Theorem \ref{YZLimit} and Corollary \ref{phase1}. 

To streamline notations for the subsequent proofs, we will omit the subscript $1$ (used in $\mathbb{P}_1$) and write $\mathbb{P}$ instead, while retaining the same meaning: all probability statements hereafter correspond to the probability space defined by the initial condition of one uniformly random infected vertex. Correspondingly, the expectation symbol $\mathbb{E}$ used throughout this section (and the rest of the paper) refers to the expectation under this specific probability space. Also, below we will always write $\lambda$ to mean the critical rate $\lambda_c=\rho m_1/(m_2-2m_1)$.

\subsection{The number of jumps in $[0, qn ^ {-1/3}]$}

Recall that we are interested in the behavior of $X _{I,t}$. Its value may change when the two types  of random events (rewiring and infection) defined in the avoSI model occur. We say that we witness a ``jump'' each time a random event occurs. In this section, we bound the number of jumps over the time interval $[0,qn^{-1/3}]$. 
The total jump rate is
\begin{equation*}
    \left(\lambda X_{I,t}+\rho X_{I,t}\right) \frac{X_t-1}{\lambda X_{I,t}}=\left(1+\frac{\rho}{\lambda}\right) (X_t-1),
\end{equation*}
where the factor $(X_t-1)/(\lambda X_{I,t})$ is due to the time-change in the analysis of avoSI model.
Therefore, when $t$ is small, the jump rate is approximately $(1+\rho/\lambda)m_1n$ and the number of jumps in $[0,qn^{-1/3}]$ is around $(1+\rho/\lambda) m_1qn^{2/3} =N _q $ (defined in \eqref{eq:realnq}) by the law of large numbers. Define two events
\begin{equation}\label{jctl}
\begin{split}
    \widehat{    \Omega}_{\ref{jctl}}:=\{\text{The number of jumps in } [0,qn^{-1/3}] \text{ is in }[N_q-n^{0.6} ,N_q+n^{0.6}]\},   \\
        \Omega_{\ref{jctl} }:= \{\text{The number of jumps  in } [0,qn^{-1/3}] \text{ is less than or equal to } N_q+n^{0.6} \}.
\end{split}
\end{equation}
The following lemma shows the concentration of the number of jumps.

\begin{lemma}\label{lemjump} 
    There exist constants $C_1,C_2>0$ such that for large $n$,    
    \begin{equation*}
        \P(\Omega_{\ref{jctl}}^c) +
        \P(\{ \gamma_n \geq qn^{-1/3}\} \cap \widehat{\Omega}_{\ref{jctl}}^c )\leq C_1\exp(-C_2    n^{8/15}).
    \end{equation*}
\end{lemma}

\begin{proof}
    Note that after each jump, the total number of unpaired half-edges $X_t$ either remains unchanged or decreases by $2$. By the degree sequence concentration assumption {\bf (H3)}, it follows that  initially the number of unpaired half-edges $X_0=\sum_{k\geq 0}kp_{k,n} \in [ m_1n-n^{C_{\ref{concern1}}}, m_1n+n^{C_{\ref{concern1}}}]$. 
    
    Let $\delta_i$ be the time lapse between the $(i-1)$-th and the $i$-th jump (with the $0$-th jump set at time $0$, which does not count). Moreover, $\delta_i$ is set as $\infty$ if there are no infected half-edges left after $(i-1)$-th jump. 
For all $1 
\leq i\leq N_q+n^{0.6}+1$,
the number of half-edges at time $t_i$ would fall in the interval
\begin{equation}\label{hatx-1}
    [X_0-2i, X_0]\subset  [   m_1n - C _{\ref{hatx-1}}n^{2 / 3}, m_1n+n^{C_{\ref{concern1}}} ],
\end{equation}
where $C_{\ref{hatx-1}}$ is some positive constant. Based on this observation,  we construct two i.i.d.\@ sequences $\{y_1,y_2, \ldots\}$ and $\{z_1,z_2, \ldots \}$, where $y_i \sim \textnormal{Exp}((1 + \rho/\lambda) (m_1n + n^{C_{\ref{concern1}}}))$ and $z_i \sim \textnormal{Exp} ((1 + \rho/\lambda) (m_1n - C_{\ref{hatx-1}}n^{2/3}))$. 
Then we can make a coupling such that $y_i\leq \delta_i\leq z_i$ for $i\leq N_q+n^{0.6}+1$ (the second inequality also requires the presence of infected edges). 

    We now use this coupling to bound the probabilities that the number of jumps is too high (by the comparison with $y_i$) or too low
(by the comparison with $z_i$).
    Let $\omega_i \sim \textnormal{Exp}   ((1+\rho/\lambda)m_1)$ be i.i.d.\@ for all $1\leq i\leq N_q+n^{0.6} +1 $, which are standardized versions of $y_i$ and $z_i$. By Bernstein's inequality (Lemma \ref{bernstein}), there exist constants $C_ {\ref{y1++} \textnormal{a}}$, $C_{\ref{y1++} \textnormal{b}}> 0$ such 
that  the probability of having a high number of jumps  satisfies
    \begin{equation}\label{y1++}
        \begin{aligned}
            \P (\Omega _{\ref{jctl}} ^c) & \leq\P  \bigg( \sum_{1\leq i\leq N_q+n^{0.6} + 1} y_i \leq qn^{-1/3} \bigg) \\
            &  \leq  \P\bigg( \Big| \sum_{1\leq i\leq N_q+n^{0.6} + 1}  \Big( \omega _ i- \frac{\lambda}{m_1 (\lambda + \rho)} \Big)   \Big| >C_ {\ref{y1++} \textnormal{a}}    n^{0.6}\bigg) \leq 2 \exp(-C_ {\ref{y1++} \textnormal{b}} n^{8/15}).
        \end{aligned}
    \end{equation}
  By the same argument, there exist some constants $C_{\ref{ctlZ} \textnormal{a}}, C_{\ref{ctlZ} \textnormal{b}}>0$ such that 
    \begin{equation}\label{ctlZ}
        \begin{aligned}    
           & \P(\{\gamma_n>qn^{-1/3}\}\cap \widehat{\Omega}_{\ref{jctl} }^c) \leq \P \Big( \sum_{1\leq i\leq N_q - n^{0.6} + 1} z_i> qn^{-1/3} \Big) +\P (\Omega _{\ref{jctl}} ^c) \\
           & \leq  \P\bigg( \Big| \sum_{1\leq i\leq N_q-n^{0.6}+1} \Big(\omega_i- \frac{\lambda}{m_1 (\lambda + \rho)} \Big) \Big| > C_{\ref{ctlZ}\textnormal{a}} n^{0.6}\bigg) + \P (\Omega _{\ref{jctl}}^c) \\
           & \leq  2 \exp(- C _{\ref{ctlZ} \textnormal{b}} n^{8/15}) + 2 \exp(-C_{\ref{y1++} \textnormal{b}} n^{8/15}),
        \end{aligned}
    \end{equation}
    which, together with \eqref{y1++}, completes the proof of the lemma.
\end{proof}

\subsection{Construction of the comparison random walks}

Let $\mathfrak{t}(\ell)$ be the time of the $\ell$-th jump (the zeroth jump is at time $0$ as before),  and let $U_{\ell} := X_{I,\mathfrak{t}(\ell)}$. Recall that under $\P$, only one vertex is initially infected. By the construction, each vertex $i$ has a fixed degree $d_{i,n}$. The initially infected vertex $u$ is chosen uniformly at random from $\{1, \dotsc, n\}$, so the initial number of infected half-edges $U_0 := X_{I,0}=d_{u,n}$. In this section, we will bound $(U_\ell)_{\ell \in \mathbb{N}}$ by two random walks.

Recall that $S_{t,k}$ is the number of susceptible vertices with $k$ half-edges at time $t$, while $I_t$ and $S_t$ are the numbers of infected vertices and susceptible vertices at time $t$, respectively. For $\ell \ge 1$, $(U_{\ell}-U_{\ell-1})$ is distributed as follows, given all the information up to time $\mathfrak{t}(\ell - 1)$.

\begin{enumerate}
    \item  For all $k\geq 3$,  $(U_{\ell}-U_{\ell-1})$ is equal to $ (k-2)$ with probability
    \begin{equation*}
        \frac{\lambda k S_{t,k}}{(\lambda+\rho)(X_t-1)}   \textnormal{.}
    \end{equation*} 
    \item $U_{\ell}-U_{\ell-1} = 0$ with probability 
    \begin{equation*}
        \frac{2\lambda S_{t,2}}{(\lambda+\rho) (X_t-1)}
        +  \frac{\rho}{\rho+\lambda} \cdot \frac{\I_t}{n}    \textnormal{.}
    \end{equation*}
    
    \item $U_{\ell}-U_{\ell-1} = -1$ with probability 
    \begin{equation*}
        \frac{\lambda S_{t,1}}{(\lambda+\rho) (X_t-1)}
         +\frac{\rho}{\rho+\lambda} \cdot \frac{S_t}{n}  \textnormal{.}
    \end{equation*} 
    
    \item $U_{\ell}-U_{\ell-1} =-2$ with probability 
    \begin{equation*}
        \frac{\lambda (\X_{I,t}-1)}{(\lambda+\rho)  (X_t-1)}    \textnormal{.}
    \end{equation*} 
\end{enumerate}

We now explain the calculation leading to the transition probabilities. Recall that $X_{I,t}$ denotes the number of half-edges belonging to infected vertices at time $t$. After the $(\ell - 1)$-th jump, the total jump rate is given by    
\begin{equation*}
    (\lambda +\rho ) \X_{I,t} \frac{\X_t-1}{\lambda \X_{I,t}} = \frac{(\lambda+\rho)(\X_t-1)}{\lambda}.
\end{equation*}
This rate consists of the rates of a few types of events. The rate at which an infected half-edge pairs with a susceptible half-edge attached to a vertex with a total of $k$ half-edges is
\begin{equation*}
    \lambda X_{I,t}\cdot\frac{kS_{t,k}}{X_t-1} \cdot\frac{X_t-1}{\lambda X_{I,t}} = kS_{t,k}.
\end{equation*}
We can see that the number of infected half-edges increases by $(k-2)$ after the pairing (if $k=1$, it actually decreases by 1). Hence, the probability in point (1) is obtained by the memoryless property of the exponential distribution and the independence of the events. The other three probabilities for $(U_{\ell}-U_{\ell-1})$ are computed analogously.  The following Lemma \ref{compbase} will be crucial for establishing upper and lower bounds for $(U_{\ell})$.

\begin{lemma}
    \label{compbase}
    There exists a constant $C_{\ref{u1xit}}>0$ such that the event
    \begin{equation}\label{u1xit}
        \begin{aligned}
            \Omega_{\ref{u1xit}}:= { } &
            \{X_{I,t}\leq C_{\ref{u1xit}}N_q, \abs{X_t-1-m_1n} \leq 3N_q, I_t\leq 2N_q, \forall \, t\leq qn^{-1/3}\}
            \\   &\cap   \{np_{k,n}-q_{k,n}/(k+1)\leq S_{t,k} \leq np_{k,n}+q_{k,n}/(k+1),
            \forall \, 0\leq k \leq  v_n, \forall \,  t\leq qn^{-1/3}\}\\
            &\cap \{S_{t,k}=0, \forall \,  k>v_n,\forall \,  t\leq qn^{-1/3} \} \\
            &\cap \left((\Omega_{\ref{jctl}} \cap \{\gamma_n<qn^{-1/3}\}) \cup (\widehat{\Omega}_{\ref{jctl}} \cap \{\gamma_n\geq qn^{-1/3}\} )
 \right)        \end{aligned}
    \end{equation}
    occurs with probability at least $( 1-n^{-0.4})$ for all large $n$, where
    \begin{equation}\label{eq:qkndef}
        q_{k,n}:=  C_{\ref{u1xit}}qn^{2/3}\exp(-\eta _ {\ref{exp1}}k/2)  \quad \mbox{and}\quad   v_n:=4+\log(nC_{\ref{exp1}})/   \eta _ {\ref{exp1}}.
    \end{equation}
\end{lemma}

The proof of Lemma \ref{compbase} is postponed to Section \ref{subsec:compbase}.
We now  bound $(U_{\ell})_{\ell \in \mathbb{N}}$ between two carefully constructed random walks $(Y_{\ell})_{\ell \in \mathbb{N}}$ and $(Z_{\ell})_{\ell \in \mathbb{N}}$.
    
\textbf{Construction of the Upper Bound $(Y_{\ell})$.} Let $Y_0=d_{u, n}= U_0$, and for $\ell \ge 1$, let the increments $(Y_{\ell}-Y_{\ell-1})$ be independent and identically distributed as follows. We emphasize here that the distribution of $(Y_\ell)_{\ell \in \mathbb{N}}$ is independent of the dynamics of the model. Recall the definition of $N_q$ in \eqref{eq:realnq}.

\begin{itemize}
    \item For all $3\leq k\leq v_n$, $Y_{\ell}-Y_{\ell-1} = k-2$ with probability
    \begin{equation*}
        \frac{knp_{k,n}+q_{k,n} }{(1+\rho/\lambda)(m_1n-3N_q)}.
    \end{equation*}
    
    \item $Y_{\ell}-Y_{\ell-1} = 0$ with probability 
    \begin{equation*}
        \frac{2np_{2,n}+q_2 }{(1+\rho/\lambda)(m_1n-3N_q) }+\frac{2N_q}{(1+\lambda/\rho)n }.
    \end{equation*}
    
    \item  $Y_{\ell}-Y_{\ell-1} = -1$ with the remaining probability.
\end{itemize}

Conditioned on $\Omega_{\ref{u1xit}}$, $(Y_{\ell}-Y_{\ell-1})$ stochastically dominates $(U_{\ell}-U_{\ell-1})$ for all $\ell\geq 1$ (provided that $\mathfrak{t}(\ell) \leq qn^{-1/3}$). Thus, 
we can couple the increments of $(Y_{\ell})$ with those of $(U_{\ell})$ on a common probability space such that $Y_{\ell}-Y_{\ell-1}\geq U_{\ell}-U_{\ell-1}$ on $\Omega_{\ref{u1xit}}$.

\textbf{Construction of the Lower Bound $(Z_{\ell})$.} Analogous to the construction of $(Y_\ell)$, let $Z_0=d_{u,n}$ and for $\ell \geq 1$, let $(Z_{\ell}-Z_{\ell-1})$ be independent with the following distribution:
\begin{itemize}
    \item For $2\leq k\leq \log(nC_{\ref{exp1}})/ \eta_{\ref{exp1}}$,   $Z_{\ell}-Z_{\ell-1} =k-2$ with probability
    \begin{equation*}
        \frac{knp_{k,n}-\min\{nkp_{k,n},2q_{k,n}\}}{(1+\rho/\lambda)(m_1n+3N_q)}.
    \end{equation*}
    
    \item $Z_{\ell}-Z_{\ell-1}= -2$ with probability 
    \begin{equation*}
        \frac{C_{\ref{u1xit}}N_q }{(1+\rho/\lambda)(m_1n-3N_q)}.
    \end{equation*}
    
    \item  $Z_{\ell}-Z_{\ell-1} = -1$ with the remaining probability.
\end{itemize} 

Under $\Omega_{\ref{u1xit}}$, $(U_{\ell}-U_{\ell-1})$ stochastically dominates $(Z_{\ell}-Z_{\ell-1})$ for all $\ell \geq 1$. Since $Z _0=U_0 =Y_ 0$, we can couple the three processes $(Z_\ell)$, $(U_\ell)$, and $(Y_\ell)$ on a common probability space such that on $\Omega_{\ref{u1xit}}$,
\begin{equation}\label{compZUY}
    Z_{\ell}\leq U_{\ell}\leq Y_{\ell}
\end{equation}
holds for all $\ell$ up to time $qn^{-1/3}$. Recall that the goal of this section---Theorem \ref{YZLimit}---involves the processes $(Y_\ell)$ and $(Z_\ell)$. The following lemma characterizes the expected values of their respective increments. 

\begin{lemma} \label{yzmean}
There exist two positive constants $C_{\ref{limy}},C_{\ref{limz}}$ (independent of $q$) such that 
    \begin{align}
        & \lim_{n\to\infty} n^{1/3}\E [Y_{\ell}-Y_{\ell-1}] = C_{\ref{limy}}q, \label{limy}  \\
        & \lim_{n\to\infty} n^{1/3}\E [Z_{\ell}-Z_{\ell-1}] = -C_{\ref{limz}}q. \label{limz}
    \end{align}
    Moreover, we have    
    \begin{equation}\label{y2z2moment}
        \limsup_{n \to\infty} n^{1/3} \left|\E[(Y_{\ell}-Y_{\ell-1})^2] - \sigma^2\right| + \limsup_{n \to\infty} n^{1/3}\left|\E[(Z_{\ell}-Z_{\ell-1})^2] - \sigma^2\right|<\infty,
    \end{equation}
    where $\sigma^2$ is defined in \eqref{defN_q}.
\end{lemma}

The proof of Lemma \ref{yzmean} is available in Section \ref{subsec:yzmean}. Define
\begin{equation*}
    J_Y:=\{Y_{\ell}>0 \text{ for all }0\leq \ell \leq N_q-n^{0.6}\}, \quad J_Z:=\{Z_{\ell}>0   \text{ for all }0\leq \ell \leq N_q+n^{0.6}\}.
\end{equation*}
According to \eqref{compZUY},  on $\Omega_{\ref{u1xit}}$ it holds that   
\begin{equation}\label{indicatorcomp}
    \1[J_Z]  \leq  \1[\gamma_n>qn^{-1/3}] \leq \1[J_Y].
\end{equation}
Since $\Omega_{\ref{u1xit}}$ holds with probability at least $(1-n^{-0.4})$, we deduce that
\begin{equation}\label{zuy1}
    \P(J_Z)- n^{-0.4} \leq  \P(\gamma_n>qn^{-1/3}) \leq \P(J_Y)+ n^{-0.4}.
\end{equation}

\subsection{Analysis for the upper bound $(Y_{\ell})$} \label{sec:uppbd_Y}

In Sections \ref{sec:uppbd_Y} and \ref{sec:lobd_Z}, we study the stochastic processes $(Y_{\ell})$ and $(Z_{\ell})$, respectively. Recall the definition of $\hat{\Omega}_{\ref{jctl}}$. On this event, the number of jumps in the interval $[0,qn^{-1/3}]$ lies between $\lfloor N_q - n^{0.6}\rfloor $ and $\lfloor N_q + n^{0.6}\rfloor $. Here we introduce a process $(\bar{Y}(t))_{0 \leq t \leq 1}$. For integers $0\leq k \leq N_q-n^{0.6}$, define 
\begin{equation}\label{defbarY}
    \bar{Y} \left(\frac{k}{N_q-n^{0.6}}\right) := \frac{Y_{k}}{\sigma \sqrt{N_q-n^{0.6}}} = \frac{Y_{k}-\E(Y_k)}{\sigma \sqrt{N_q-n^{0.6}}}+\frac{\E(Y_k)}{\sigma \sqrt{N_q-n^{0.6}}}.
\end{equation}
Furthermore, we linearly interpolate the value of $\bar{Y}(t)$ between the points in \eqref{defbarY}. For $\frac{\lfloor N_q - n^{0.6} \rfloor}{N_q - n^{0.6}} \leq t \leq 1$, we set $\bar{Y}(t) := \bar{Y} \left( \frac{\lfloor N_q - n^{0.6} \rfloor}{N_q - n^{0.6}} \right)$, where $\lfloor x \rfloor$ denotes the floor function (integer part of $x$). By Lemma \ref{yzmean},
\begin{equation*}
    \E(Y_k)=\E(Y_0)+kC_{\ref{limy}}qn^{-1/3}+o( n^{-1/3} )  \quad \text{ as } n \to \infty.
\end{equation*}
Therefore, by Donsker's Theorem, $(\bar{Y}(t))$ converges in distribution to 
\begin{equation}\label{cqdef}
    B_{t,q} := B_t+C_{\ref{cqdef}}(q)t,
\end{equation}
where $(B_t)$ is a standard Brownian motion, and $C_{\ref{cqdef}}(q):={C_{\ref{limy}}} {\sigma^ {-1}}  \sqrt {(1+\rho/\lambda) m_1}q^{3/2}$. 
Moreover, we have the following lemma, whose proof is postponed to Section \ref{Donskerinv}.

\begin{lemma}\label{lemcondiy}
    Conditioned on the event
    $$\{Y_{\ell}>0 \text{ for all } 0\leq \ell \leq  N_q-n^{0.6}\}=\{\bar Y(t)>0 \text{ for all } 0\leq t\leq 1\},$$ 
    we have the convergence in distribution,
    \begin{equation}\label{chiqdef}
        \bar{Y}(1)=\frac{Y_{ [ N_q-n^{0.6} ] }}{\sigma \sqrt{N_q-n^{0.6}}} \cd \chi_q := 
        B_{1+T_q, q}- B_{T_q,q}        =B_{1+T_q}-B_{T_q}+C_{\ref{cqdef}}(q),
    \end{equation}
    where $T_q := \inf\{t\geq 0: B_{u,q}\geq B_{t,q} $ for all $t \leq  u \leq t +1 \}$, and $T_q < \infty$ almost surely.
\end{lemma}
Notice that when $Y_0=1$, Lemma \ref{lemcondiy} follows from the same argument as in \cite{MR415702}, since (3.1) of \cite{MR415702} holds for i.i.d.\@ sequences, and the arguments of Theorem 3.2 in \cite{MR415702} rely only on Donsker's Theorem. Actually a stronger conclusion hold true:  the conditional distribution of the procress $(\bar{Y}(t))_{0\leq t\leq 1}$ converges to
\begin{equation*}
    B_{t+T_q,q}-B_{T_q,q}=B_{t+T_q}-B_{T_q}+C_{\ref{cqdef}}(q)t.
\end{equation*}
However, since our initial value $Y_0= O(\log n)$ depends on $n$, we need to be careful for the weak convergence---that is precisely what Lemma \ref{lemcondiy} establishes.

The following lemma provides a bound for the tail probability of $T_q$.  

\begin{lemma}\label{tailT}
    There exist two constants $c_{\ref{tailtq}},C_{\ref{tailtq}}>0$ such that for all $r > 1$,
    \begin{equation}\label{tailtq}
        \sup_{q\in (0,1)}\P(T_q>r)\leq C_{\ref{tailtq}} e^{-c_{\ref{tailtq}}r} .
    \end{equation}
\end{lemma}

\begin{proof}
    For any $N>0$, define a new measure $\P_{q,N}$ by
    \begin{equation}\label{changeofmeasure}
        \frac{\mathrm{d}\P_{q,N}}{\mathrm{d}\P} := \exp\{-C_{\ref{cqdef}} (q)B_{N+1}- C_{\ref{cqdef}} (q)^2 (N+1)/2\}.
    \end{equation}
   By Girsanov's Theorem, the process $(B_{t,q})_{0 \leq t \leq N +1}$ is a standard Brownian motion up to time $N+1$ under $\P_{q,N}$, which implies that the distribution of $T_q\wedge N$ (as defined in Lemma \ref{lemcondiy}) under $\P_{q,N}$ is independent of $q$. Therefore, $a := \P_{q,1}(T_q<1)$ is a constant independent of $q$. Note that $\{T_q<1\}$ is measurable with respect to the $\sigma$-field generated by $(B_{t,q}) _ {0\leq t\leq 2}$. By Lemma 2.2 of \cite{MR415702}, we have $a>0$. Applying H\"{o}lder's inequality,
    \begin{equation}\label{p0q}
        0<a=\P_{q,1}(T_q<1)=\E \left[ \frac{\mathrm{d}\P_{q,1}}{\mathrm{d}\P} 1_{\{T_q<1\}} \right] \leq \sqrt{ \P(T_q<1)} \cdot  \sqrt{ \E \left( \frac{\mathrm{d}\P_{q,1}}{\mathrm{d}\P} \right)^2 }.
    \end{equation}
    Since $B_2$ follows a normal distribution with mean $0$ and variance $2$, substituting \eqref{changeofmeasure} into the expectation yields,
    \begin{equation}\label{derivativeintegrable}
        \E \left( \frac{\mathrm{d}\P_{q,1}}{\mathrm{d}\P} \right) ^ 2 = \E [\exp (-2C_{\ref{cqdef}}(q) B_2-2C_{\ref{cqdef}}(q)^2 )] = \exp (2C_{\ref{cqdef}}(q)^2 ).
    \end{equation}
    Combining \eqref{p0q}, \eqref{derivativeintegrable} and the definition of $C_{\ref{cqdef}}(q)$, for all $0 <q< 1$,
    \begin{equation}\label{ptq<1}
        \P(T_q<1) \geq a^2  \exp (  -2C_{\ref{cqdef}}(q)^2   )\geq a^2  \exp (  -2C_{\ref{cqdef}}(1)^2 )=: C_{\ref{ptq<1}}.
    \end{equation}
    As $(B_{t,q})_{0\leq t\leq 2}$ has stationary and independent increments under $\P$, for $r>2$,
    \begin{equation*}
        \P(T_q>r) = \P \bigg(\inf_{t\leq u \leq t+1} B_{u,q} <B_{t,q} \text{ for all } 0 \leq t\leq r \bigg)\leq \P(T_q \geq 1)^{\lfloor   r/2 \rfloor}\leq (1-C_{\ref{ptq<1}})^{\lfloor r/2 \rfloor}. 
    \end{equation*}
    The desired result follows immediately. 
\end{proof}
  
As established in Lemma \ref{lemcondiy}, the distribution of $\chi_q$ is of great interest, and its  convergence as $q \to 0$ is described in the following lemma.  

\begin{lemma}\label{b+converge}
    Given any  continuous function $g$ with at most polynomial rate of growth at $\infty$, we have 
    \begin{equation*}
        \lim_{q\to 0} \E [g(\chi_q )] = \E [g(B^+_1)],  
    \end{equation*}
    where $B^+_1$ is the Brownian meander at time $1$.
\end{lemma}

\begin{proof}
The proof contains two steps.

\emph{Step 1.} We first show that  $\chi_q \cd B_1^+$ as $q\to 0$. Let
    \begin{equation*}
        T_0:=\inf\{t\geq 0: B_u\geq B_t \text{ for all } t < u < t+1 \}  \textnormal{.}
    \end{equation*}
    Then, $B_1^+$ is equal in law to $(B_{1+T_0} - B_{T_0})$. Let $f$ be a bounded and continuous function, and fix a constant $N>0$. We decompose $\E [f(B^+_1)]$ as follows: 
    \begin{align}
        \E [f(B^+_1)] & = \E[f(B_{1 + T_0}-B_{T_0})] \nonumber \\
        & = \E [f(B_{1+T_0}-B_{T_0})\1[T_0\leq N]] + \E [f(B_1^+)\1[T_0> N]]  \textnormal{.}  \label{intb+}
    \end{align}
    Similarly,
    \begin{equation}\label{inb+2}
        \E [f(\chi_q)]=\E [f(\chi _q)\1[T_q\leq N]  ] + \E[f(\chi_q)\1[T_q>N]] \textnormal{.} 
    \end{equation}
    Since $T_0 <\infty$ a.s., and $f$ is bounded, it follows from Lemma \ref{tailT} that
    \begin{equation}\label{tailbt0}
        \lim_{N\to\infty} \E [f(B_1^+)\1[T_0> N]]  = \lim_{N\to\infty}  \sup_{0 <q <1} \E [f(\chi_q)\1[T_q>N] ] = 0 \textnormal{.} 
    \end{equation}
    By the definition of $\P_{q,N}$ (cf$.$ \eqref{changeofmeasure}), we obtain
    \begin{equation}\label{eq:fchiq0}
    \begin{split}
&                \E [f(\chi_q)\1[T_q\leq N]] \\
                =& \E_{q,N} [f(B_{1+T_q,q}-B_{T_q,q})\1[T_q\leq N] \exp (C_{\ref{cqdef}}(q) B_{N+1} + C_{\ref{cqdef}}(q)^2 (N+1)/2) ] \\
                = & \E_{q,N} [f(B_{1+T_q,q}-B_{T_q,q})\1[T_q\leq N] \exp (C_{\ref{cqdef}}(q) B_{N+1,q} - C_{\ref{cqdef}}(q)^2 (N+1)/2) ] \\
      = &\exp({ - C_{\ref{cqdef}}(q)^2 (N+1)/2 } ) \E [  f(B_{T_0  +1}-B_{T_0})   \1[T_0\leq N] \exp[C_{\ref{cqdef}}(q) B_{N+1} ] ]\textnormal{,}         
    \end{split}
    \end{equation}
    where $\E_ {q, N}$ denotes the expectation with respect to the probability measure $\P_{q,N}$. 
    In order to take the $q\to 0$ limit in \eqref{eq:fchiq0}, we   note that
    \begin{equation*}
        \left|f(B_{1+T_0}-B_{T_0})   \1[T_0\leq N]   \exp \bigg[C_{\ref{cqdef}}(q) B_{N+1} \bigg] \right| \leq  \left( \sup_ {x \geq 0 } |f(x)| \right) \exp \bigg [ C_{\ref{cqdef}}(1)  |B_{N+1}| \bigg]  \textnormal{.}
    \end{equation*}
    By the Dominated Convergence Theorem and the fact $C_{\ref{cqdef}}(q) \to 0$ as $q\to 0$, we get
    \begin{equation}\label{bulkbtq}
        \E  \bigg[f( \chi _ q )\1[T_q\leq N] \bigg] \to \E \bigg [f(B_{T _ 0 + 1}-B_{T_0})\1[T_0\leq N] \bigg], \quad \text{as } q \to 0 \textnormal{.}
    \end{equation}
    Combining \eqref{intb+}, \eqref{inb+2}, \eqref{tailbt0}, and \eqref{bulkbtq}, we get
    \begin{equation}\label{B_tq}
        \E [f(\chi_q) ] \to \E [f(B^+_1)],  \quad \text{as } q \to 0  \textnormal{,}
    \end{equation}
    which implies that $\chi_q \cd B_1^+$ as $q \to 0$, as desired.     
        
    \emph{Step 2.}  We next show that for any function $g$ satisfying the assumption of Lemma \ref{b+converge}, $g(\chi_q)$ is uniformly integrable in $q \in ( 0 , 1 )$. By the polynomial growth condition on  $g$, for some $C,r>0$
    \begin{equation*}
        g(x)^2 \leq C(|x|^r+1), \quad \forall x \in\mathbb{R}.
    \end{equation*}
    Without loss of generality, we can assume that $r \geq 1$. Using the inequality $(a+b) ^ r \leq 2^{r-1}(a ^ r +b ^ r)$ for $a, b\geq 0$, we see that 
    \begin{align*}
         \sup_{0 < q<1}  \E[ g(\chi_q)^2 ] \leq & C \sup_{0 < q<1} \E[(|\chi_q|^ r+1)]    \\
        \leq &  C+ C \sup_{0 < q<1}  \E \left[ \left(2 \sup_ {0 \leq s \leq 1+ T_q} |B_s| +C_{\ref{cqdef}} (1)\right)^r \right]   \\
        \leq  &   C' +C' \sup_{0 < q<1}\E  \left[\sup_ {0 \leq s \leq 1+ T_q} |B_s|^r  \right],
    \end{align*}
    where $C'$ is another (large) constant.    
    Therefore, to prove $\sup_{0 < q<1} \E[ g(\chi_q)^2]<\infty$, it suffices to control $$\sup_{0 < q<1}\E  \left[\sup_ {0 \leq s \leq 1+ T_q} |B_s|^r  \right],$$ 
    which can be bounded by
    \begin{align*}
        & \sup_{0 < q<1}   \left[  \sum_{N=1}^{\infty}  \E \left[ \left( \sup_ {0 \leq s \leq N +1} |B_s| ^r\right) \1[T_q \in [N-1,N)] \right]\right] \\ 
        \leq & \sum_{N=1}^{\infty} \left(\E \left[ \sup_ {0 \leq s \leq N +1} |B_s| ^{2r}\right] \right) ^{1 /2} \sup_{0 < q<1}\left(  \P(T_q  >  N-1) \right)  ^ {1/2}  \\ 
        \leq & C''\sum_{N=1}^{\infty} (N+1)^{r/2}\exp(-c_{\ref{tailtq}} (N-1) / 2)<\infty.
    \end{align*}
    Therefore, $\{ g(\chi _ q), q\in (0,1)\}$ is uniformly integrable,  which together with \emph{Step 1} completes the proof of Lemma \ref{b+converge}.   
\end{proof}

Recall that 
\begin{equation*}
    J_Y=\{Y_{\ell}>0 \textnormal{ for all }0\leq \ell \leq N_q-n^{0.6}\}\textnormal{.}
\end{equation*}
We now study the asymptotics of $\P(J_Y)$.  

\begin{prop}\label{prop-est-Y}
    The limit    
    \begin{equation}\label{lim-survY}
        c_{q,Y} = \lim_{n \to \infty} n^{1/3}\P(J_Y)
    \end{equation}
    exists, and it satisfies
    \begin{equation}\label{cqylimit}
        \lim_{q\to 0} q^{1/2} c_{q,Y} = \sigma^{-1} \sqrt{\frac {2m_{1}} {\pi\left(1+ \frac{\rho}{\lambda}\right)} } \textnormal{.}
    \end{equation}
\end{prop}

\begin{proof}
    We first prove the existence of the limit of $q^{1/2}c_{q,Y}$. 
    Let $\phi_n(\theta)$ be the moment generating function of $-n^{-1/3}(Y_{\ell}-Y_{\ell-1})$, i.e., 
    $$
    \phi_n(\theta)=  \E [ \exp(-\theta n^{-1/3}(Y_{\ell}-Y_{\ell-1})) ].
    $$
    Combining Lemma \ref{yzmean} with the fact $\abs{Y_{\ell}-Y_{\ell-1}} = O( \log n)$, for any $\abs{\theta}=O(1)$, we have
    \begin{equation}\label{eq:phintheta}
       \phi_n(\theta) =1-\theta C_{\ref{limy}} \cdot qn^{-2/3}+\frac{\theta^2 \cdot  n^{-2/3}}{2}\sigma^2+o(n^{-2/3})\textnormal{.}
    \end{equation}
    In particular, \eqref{eq:phintheta} implies the existence of two constants $0<\theta_*<\theta^*$, such that
    \begin{equation}\label{eq:phisign}
        \phi_n(\theta_*)<0, \quad \mbox{and} \quad \phi_n(\theta^*)>0, \quad 
    \end{equation}
    for all large $n$. Combining \eqref{eq:phisign} with the convexity of $\phi_m(\theta)$ (in $\theta$), we deduce that  there exists a unique positive solution $\theta _n$ to the following equation
    \begin{equation*}
        \E [ \exp(-\theta n^{-1/3}(Y_{\ell}-Y_{\ell-1})) ]=1 \textnormal{.}
    \end{equation*}
    By \eqref{eq:phintheta}, as $n \to \infty$, we have
    \begin{equation}  \label{theta1}
\theta_n=2C_{\ref{limy}}q/\sigma^2+o(1)\textnormal{.}
    \end{equation}
    By the definition of $\theta_n$, the process $\{\exp(-\theta_n n^{-1/3}Y_{\ell})\}_{\ell\in \mathbb{N}}$ is a martingale. Define    
    \begin{equation*}
        \Gamma := \inf \{ \ell \geq 0 : Y_\ell = 0\} \wedge  \lfloor N_q-n^{0.6}\rfloor \textnormal{.}
    \end{equation*}
    Then by the Optional Stopping Theorem, 
    \begin{align*}    
        \E[ \exp (-\theta_n n^{-1/3}Y_0 ) ]&= \E[ \exp (- \theta_n n^{-1/3}Y_{\Gamma} )]  \\
        & = 1-\P(J_Y) + \E \left[ \exp (-\theta_n n^{-1/3} Y_{ \lfloor N_q-n^{0.6}\rfloor  } ) \middle |J_Y\right] \P(J_Y) \textnormal{.}
    \end{align*}
   Solving the above equation yields    
    \begin{equation}\label{pjyy0}
        \P(J_Y)=  \frac{\E   (1-\exp(-\theta_n n^{-1/3} Y_0))}{1 - \E \left[ \exp (-\theta_n n^{-1/3} Y_{ [N_q-n^{0.6}]} ) \middle |J_Y\right]} \textnormal{.}
    \end{equation}
    Note that  $Y_0=d_{u, n}\leq \log(nC_{\ref{exp1}})/ \eta _ {\ref{exp1}}$ by \eqref{dmaxbd}. Combining {\bf(H3)} and \eqref{theta1}, we see that
    \begin{equation}
    \begin{split}
            \E \left [  1-\exp\left(-\theta_n n^{-1/3} Y_0 \right) \right] 
            =& \theta_n n^{-1/3}\E [ Y_0 ] +O(n^{-2/3} \log^2 n )  \\
        =&   \theta_n n^{-1/3} \sum _{k=0} ^ \infty kp_k^*  + o(n^{-1/3}) \\
        =& 2C_{\ref{limy}}q  m_1 n^{-1/3}/\sigma^2 + o(n^{-1/3}) \textnormal{.}  \label{ethetay0}
    \end{split}
    \end{equation}
    By Lemma \ref{lemcondiy}, as $n \to \infty$,  
    \begin{equation}\label{expconverge}
        \E \left[ \exp (-\theta_n n^{-1/3} Y_{ [N_q-n^{0.6}] }) \middle |J_Y\right] \to\E [ \exp (-2C_{\ref{cqdef}}(q)\chi_q)] \textnormal{.}
    \end{equation}
    Combining \eqref{pjyy0}, \eqref{ethetay0} and \eqref{expconverge}, we conclude that 
    \begin{equation}\label{q4}
        \lim _{n \to\infty} n^{1/3}\P(J_Y) = \frac{2C_{\ref{limy}}qm_1 /\sigma^2}{1-\E [ \exp (-2C_{\ref{cqdef}}(q)\chi_q)  ] }=: c_ {q,Y}  \textnormal{,}   
    \end{equation}
    which implies \eqref{lim-survY}. 
    
    Next, we prove \eqref{cqylimit}. Recall that $B^+_1$ (the Brownian meander at time $1$) has the density function $x\exp(-x^2/2) \1[(0, +\infty)]$. See, e.g., \cite[equation (1.1)]{MR436353}. Thus $\E(B_1^+)=\sqrt{\pi/2}$ and $\E[(B^+_1)^2]<\infty$.
    By Lemma \ref{b+converge} (with the test function $g(x)$ taken to be $x$ or $x^2$),
        \begin{equation}\label{btqb1}
            \lim_{q\to 0} \E [ \chi_q] = \E[B_1^+]= \sqrt{\pi/2}, \quad \mbox{and}\quad
            \lim_{q\to 0}\E[\chi_q^2]=\E [(B_1^+)^2]<\infty
            \textnormal{.}
        \end{equation}
   Combining the definition of $C_{\ref{cqdef}}(q)$ and \eqref{btqb1} and the elementary inequality
   $$
   \abs{\exp(-x) -(1-x)}\leq x^2, \quad \forall \, x\geq 0,
   $$
   we have that,    as $q\to 0$, 
   \begin{align}
        1-\E [\exp (-2C_{\ref{cqdef}}(q)\chi_q )] & =2C_{\ref{cqdef}}(q) \E [ \chi_q] + O (q ^ 3) \nonumber \\
        &=C_{\ref{limy}}\sigma^{-1}q^{3/2} \sqrt{2\pi \left(1+ \frac{\rho}  {\lambda}  \right) m_1}  + O(q^3)\textnormal{,} \label{q2}
   \end{align}
   proving  \eqref{cqylimit} by the definition of $c_{q,Y}$ in \eqref{q4}.
\end{proof}

\subsection{Analysis for the lower bound $(Z_{\ell})$} \label{sec:lobd_Z}

Similar to the previous section on the upper bound $(Y_{\ell})$ and $(\bar Y(t))$, we introduce a stochastic process $ (\bar{Z}(t))_{0\leq t\leq 1}$ as follows. For integers $0\leq k\leq N_q+n^{0.6}$, let
\begin{equation}  \label{defbarz}
    \bar Z\left(\frac{k}{N_q+n^{0.6}}\right) := \frac{Z_{k}}{\sigma \sqrt{N_q+n^{0.6}}} = \frac{Z_{k}-\E Z_k}{\sigma \sqrt{N_q+n^{0.6}}}+\frac{\E Z_k}{\sigma \sqrt{N_q+n^{0.6}}} \textnormal{,}
\end{equation}
and the value of $\bar{Z} (t)$ between the points in \eqref{defbarz} is linearly interpolated. 
\begin{prop}\label{prop-est-Z}
    There exists a constant $c_{q,Z}>0$ such that 
    \begin{equation*}
        c_{q,Z} \leq  \liminf_{n\to+\infty} n^{1/3}\P(J_Z) \leq \limsup_{n\to+\infty} n^{1/3}\P(J_Z)\leq  c_{q,Y}
    \end{equation*}
    and that
    \begin{equation}    \label{cqzlimit}
        \lim_{q\to 0} q^{1/2}    c_{q,Z} =\lim_{q\to 0}q^{1/2}c_{q,Y}\textnormal {.}
    \end{equation}
\end{prop}

\begin{proof}
    By  \eqref{zuy1} and \eqref{lim-survY}, we must have
    \begin{equation}\label{eq:pjzjz}
        \liminf_{n\to+\infty} n^{1/3}\P(J_Z) \leq \limsup _ {n\to+\infty} n^{1/3}\P(J_Z)\leq  c_{q,Y}  \text{.}
    \end{equation}
    It remains to find a proper lower bound for the left hand side of \eqref{eq:pjzjz}.    
    Combining Lemma \ref{yzmean} and the fact $ Z_\ell-Z_{\ell-1}= O(\log n)$, we obtain
    \begin{equation*}
        \E [ \exp(\theta n^{-1/3}(Z_{\ell}-Z_{\ell-1})) ] = 1-\theta C_{\ref{limz}}q n^{-2/3} + \frac{\theta ^2\cdot n^{-2/3}}{2}\sigma^2+o(n^{-2/3}), \quad \text{ as } n \to\infty \text{.}
    \end{equation*}
    As in the proof of Proposition \ref{prop-est-Y}, there exists a unique positive solution ${\theta} = \tilde {\theta} _n$ to the equation
    \begin{equation*}
        \E [ \exp(\theta n^{-1/3}(Z_{\ell}-Z_{\ell-1})) ] =1 \textnormal{,}   
    \end{equation*}
    for all large $n > 0$, which satisfies
    \begin{equation}\label{eq:hattheta_n}
        \tilde{\theta}_n=2C_{\ref{limz}} q/ \sigma^2+o(1), \quad \mbox{as }n\to\infty \text{.}
    \end{equation}
    Thus we can define a new probability $\tilde{\P}$ by 
    \begin{equation}\label{PtildeP}
        \frac{\mathrm{d}\tilde{\P}}{\mathrm{d}\P}  := \exp\left[\tilde{\theta}_n n^{-1/3} (Z_{\ell}-Z_0)\right]  \textnormal{.}
    \end{equation}  
    Clearly $Z_0$ has the same law under $\tilde{\P}$ as $\P$.
Combining Lemma \ref{yzmean} and $|Z_\ell-Z_{\ell-1}|= O(\log n)$, 
    \begin{align*}
            & \tilde{\mathbb{E}}(Z_\ell-Z_{\ell-1}) = \E\left( (Z_\ell -Z_{\ell-1})\exp(\tilde{\theta}_n n^{-1/3}(Z_{\ell}-Z_{\ell-1})) \right) \\
           =&  \E  (Z_{\ell}-Z_{\ell-1})+ \tilde{\theta}_n n^{-1/3} \E (Z_{\ell}-Z_{\ell-1})^2 +O(n^{-2/3} \log^3 n)\\
           = & -C_{\ref{limz}} q  n^{-1/3} + 2C_{\ref{limz}} q  n^{-1/3} +o(n^{-1/3} )=C_{\ref{limz}} q  n^{-1/3} + o(n^{-1/3}) .
     \end{align*}
    Similar to \eqref{cqdef}, under the measure $\tilde{\P}$, $(\bar Z(t))$  converges in law to the process
    \begin{equation}\label{cqdef-tilde}
        \tilde{B}_{t,q}:= B_t+ \frac{C_{\ref{limz}}}{\sigma}\sqrt{(1+\rho/\lambda)m_1}q^{3/2}t =: B_t+C_{\ref{cqdef-tilde}}(q)t,
    \end{equation}
Repeating the proof of Lemma \ref{lemcondiy}, we can show that, under  the law of     $\tilde{\P}(\cdot |J_Z)   $, 
    \begin{equation}\label{eq:tildez1con}
\tilde{Z}(1)    \cd 
 \tilde{\chi}_q:=\tilde{B}_{1+\tilde{T}_q,q}- \tilde{B}_{\tilde{T}_q,q} =  B_{1+\tilde{T}_q} -B_{\tilde{T}_q} +C_{\ref{cqdef-tilde}}(q),
    \end{equation}
     where 
    \begin{equation*}
        \tilde{T}_q:=\inf\{t\geq 0: \tilde{B}_{u,q}\geq \tilde{B}_{t,q},\text{ for all } u\in [t,t+1]\}< \infty, \mbox{ a.s.}
    \end{equation*}
    by Lemma \ref{tailT}. For future reference, we note that Lemma \ref{b+converge} is also valid for $\tilde{\chi}_q$. In particular,
    \begin{equation}\label{eq:tilchiq}
        \tilde{\chi}_q \cd B_1^+ \quad \mbox{ as }q \to 0.
    \end{equation}

    Define a stopping time
    $$
    \tilde{\Gamma}:=\inf \{ \ell \geq 0 : Z_\ell \leq  0\} \wedge  \lfloor N_q+n^{0.6}\rfloor \textnormal{.}
    $$
Due to the constraint  $Z_\ell-Z_{\ell-1}\geq -2$, we see that
$$Z_{\tilde{\Gamma}} = 0 \mbox{ or } -1,\quad \mbox{ on } J_Z^c.$$
Consequently, by the Optional Stopping Theorem (applied to the martingale\footnote{Indeed, 
$\tilde{\E}  \exp\left(-\tilde{\theta}_n (Z_{\ell}-Z_{\ell-1})\right)= \E
 \exp\left(
 \tilde{\theta}_n n^{-1/3} (Z_{\ell}-Z_{\ell-1})
 -\tilde{\theta}_n n^{-1/3} (Z_{\ell}-Z_{\ell-1})\right)
=1$.
}  $\{ \exp[ -\tilde{\theta}_n n^{-1/3}Z_\ell] \}_{ \ell \in \N }$ under $\tilde{\P}$  and stopping time $\tilde{\Gamma}$),
    \begin{align*}
        & 1-\tilde{\P}(J_Z) + \tilde{\E} \left(\exp\left(-\tilde{\theta}_n \frac{Z_{\lfloor N_q+n^{0.6} \rfloor  }}{n^{1/3}} \right)  |J_Z\right)  \tilde{\P} (J_Z) \\
        \leq & \tilde{\E} (\exp\left( -\tilde{\theta}_n n^{-1/3}Z_{\tilde{\Gamma}}\right)) + \tilde{\E} \left(\exp\left(-\tilde{\theta}_n \frac{Z_{\lfloor N_q+n^{0.6}\rfloor  }}{n^{1/3}} \right)  |J_Z\right)  \tilde{\P} (J_Z)\\
        =& \tilde{\E}\left(  \exp\left( -\tilde{\theta}_n n^{-1/3}Z_0\right) \right) = \E \left(  \exp\left( -\tilde{\theta}_n n^{-1/3}Z_0\right) \right),
    \end{align*}
    which in turn is equivalent to
    \begin{equation}\label{bdd-j-z}
        \left( 1-\E\left( \exp\left( -\tilde{\theta}_n n^{-1/3}Z_0\right) \right) \right) \left(1-\tilde{\E}\left(\exp\left(-\tilde{\theta}_n \frac{Z_{  \lfloor N_q+n^{0.6} \rfloor}}{n^{1/3}} \right)  |J_Z\right)  \right)^{-1} \leq  \tilde{\P}(J_Z) .
    \end{equation}
    Using a similar argument in \eqref{ethetay0} with $\theta_n$ replaced by $\tilde{\theta}_n$, 
    \begin{equation}\label{1-zl}
    \begin{split}
        1-\E\left(  \exp\left( -\tilde{\theta}_n n^{-1/3}Z_0\right) \right) & = \tilde{\theta}_n n^{-1/3} \E (Z_0) + o(n^{-1/3})\\
        & =  2C_{\ref{limz}}qm_1n^{-1/3}/\sigma^2+o(n^{-1/3}).
    \end{split}
    \end{equation}
    By \eqref{eq:hattheta_n} and  \eqref{eq:tildez1con},
we also get
\begin{equation}\label{zlcond}
    \lim_{n\to\infty} 
    \tilde{\E}\left(\exp\left(-\tilde{\theta}_n \frac{Z_{  \lfloor N_q+n^{0.6} \rfloor}}{n^{1/3}} \right)  |J_Z\right)  
    =\E\left(\exp\left(- 2C_{\ref{cqdef-tilde}}(q) \tilde{\chi}_q \right)  \right). 
\end{equation}
    Plugging \eqref{1-zl} and \eqref{zlcond} back to \eqref{bdd-j-z} yields that 
    \begin{equation}\label{lo-J-z-tilde}
        \liminf_{n\to+\infty} n^{1/3}\tilde{\P}(J_Z) \geq \left(1-\E\left(\exp\left(- 2C_{\ref{cqdef-tilde}}(q) \tilde{\chi}_q \right)  \right)   \right)^{-1} \frac{2C_{\ref{limz}}qm_1}{\sigma^2} =:C_{\ref{lo-J-z-tilde}}(q).
    \end{equation}
    Now by the change of measure rule \eqref{PtildeP} and the fact that $Z_0= O(\log n)$, we have
    \begin{align*}
        \P(J_Z) & =  \tilde{\E}\left(\exp\left(-\tilde{\theta}_n \frac{Z_{N_q+n^{0.6}}-Z_0}{n^{1/3}}  \right) \1[J_Z]\right) \\
        & = e^{O(n^{-1/3}\log n)} \tilde{\P}(J_Z) \tilde{\E} \left(\exp\left(-\tilde{\theta}_n \frac{Z_{N_q+n^{0.6}}}{n^{1/3}} \right)  |J_Z\right).
    \end{align*}
    Together with \eqref{lo-J-z-tilde}, we conclude that
    \begin{equation}\label{def-of-cqz}
        \liminf_{n\to+\infty} n^{1/3}\P (J_Z) \geq  C_{\ref{lo-J-z-tilde}}(q)  \E\left(\exp\left(- 2C_{\ref{cqdef-tilde}}(q) \tilde{\chi}_q \right)  \right) =: c_{q,Z}.
    \end{equation}
    Repeating the same argument  in \eqref{q2} with $\chi_q$ replaced by $\tilde{\chi}_q$, we get \eqref{cqzlimit}. 
\end{proof} 

\begin{proof}[Proof of Theorem \ref{YZLimit}]
    The result of Theorem \ref{YZLimit} follows immediately from \eqref{zuy1} and Propositions \ref{prop-est-Y} and \ref{prop-est-Z}. 
\end{proof}

\subsection{Proof of Corollary \ref{phase1}}

Define
\begin{equation*}
    V_n(Y):=\frac{1}{\sigma\sqrt{N_q}} \max_{\lfloor N_q-n^{0.6}\rfloor \leq i\leq \lfloor N_q+n^{0.6}\rfloor  } (Y_i-Y_{N_q-n^{0.6}})   
\end{equation*}
and 
\begin{equation*}
    V_n(Z):= \frac{1}{\sigma \sqrt{N_q}}\max_{\lfloor N_q-n^{0.6}\rfloor \leq i\leq \rfloor N_q+n^{0.6}\rfloor  } (Z_{i}-Z_{\lfloor N_q+n^{0.6} \rfloor }).
\end{equation*}
Let $f$ be any bounded positive continuous increasing function. Similarly to  \eqref{indicatorcomp}, the following inequality holds on  on $\Omega_{\ref{u1xit}}$,  
\begin{align}
    f\left(\frac{ Z_{\lfloor N_q+n^{0.6}\rfloor  } }{\sigma \sqrt{N_q}} -V_n(Z)\right) 
    \1[J_Z]  & \leq  f\left( \frac{X_{I,qn^{-1/3}}}{\sigma \sqrt{N_q}} \right)\1[\gamma_n\geq qn^{-1/3}] \nonumber \\
    &  \leq  f\left(\frac{ Y_{\lfloor N_q-n^{0.6} \rfloor } }{\sigma \sqrt{N_q}}+V_n(Y)\right) \1[J_Z]\label{zuy2}.
\end{align}
Integrating  \eqref{zuy2} on $\Omega_{\ref{u1xit}}  $ and set $\Vert f\Vert_\infty:= \sup_{z\in \mathbb{R}}|f(z)|$, we readily get
\begin{equation}
\begin{split}
       & \E\left(f\left(\frac{Z_{\lfloor N_q+n^{0.6} \rfloor }}{\sigma\sqrt{N_q}}-V_n(Z)\right)\1[ J_Z] \right)-\Vert f \Vert_{\infty} \P(\Omega^c_{\ref{u1xit}}  )  \\ 
        \leq &\E\left(f\left(\frac{X_{I, qn^{-1/3}}}{\sigma \sqrt{N_q}}\right) \1[\gamma_n>qn^{-1/3}] \right)\\ 
    \leq &\E\left(f\left(\frac{Y_{\lfloor N_q-n^{0.6} \rfloor }}{\sigma\sqrt{N_q}}+V_n(Y)\right)\1[ J_Y] \right)+\Vert f \Vert_{\infty} \P(\Omega^c_{\ref{u1xit}}). 
\end{split}
\end{equation}
Dividing each term in the above equation by $\P(\gamma_n>qn^{-1/3})$ and using  \eqref{zuy1}, we find 
\begin{align}
    & \E\left(f\left(\frac{Z_{N_q}}{\sigma\sqrt{\lfloor N_q+n^{0.6} \rfloor }}-V_n(Z) \right)| J_Z \right) \frac{\P(J_Z)}{\P(J_Y)+n^{-0.4} }-\Vert f\Vert_{\infty}\frac{\P(  \Omega_{\ref{u1xit}} ^c )}{\P(J_Z)-n^{-0.4} } \nonumber\\
    \leq  & \E\left(f\left(\frac{X_{I, qn^{-1/3}}}{\sigma \sqrt{N_q}}\right)| \gamma_n>qn^{-1/3} \right) \label{YZU}\\
    \leq  & \E\left(f\left(\frac{Y_{\lfloor N_q-n^{0.6}\rfloor }}{\sigma\sqrt{N_q}}+V_n(Y)\right)| J_Y \right)  \frac{\P(J_Y) }{\P(J_Z)-n^{-0.4} }+\Vert f\Vert_{\infty}\frac{\P(\Omega_{\ref{u1xit}} ^c )}{\P(J_Z)-n^{-0.4} }. \nonumber
\end{align}

The following Proposition \ref{prop:decayyzep}   shows that the  terms $V_n(Y)$ and $V_n(Z)$ can be ignored even after we condition on $J_Y$ and $J_Z$.  

\begin{prop}\label{prop:decayyzep}
    For any fixed $\ep>0$, the decay rates of $\P(\abs{V_Y(n)} > \ep)$ and $\P(\abs{V_Z(n)} > \ep)$ are faster than any polynomial of $n$.
\end{prop}

Assume Proposition \ref{prop:decayyzep} for the moment, we now prove Corollary \ref{phase1}.
\begin{proof}[Proof of Corollary \ref{phase1}]

    The last inequality of \eqref{ineq1} (upper bound part) follows immediately from Theorem \ref{YZLimit} and \eqref{YZU}, so it remains to prove the first inequality (lower bound part). We keep the notations of $\tilde{\E}$, $\tilde{\P}$ and $\tilde{\chi}_q$ as in the proof of Proposition \ref{prop-est-Z}. Notice that 
    \begin{align*}
        & \E\left(f\left(\frac{Z_{N_q+n^{0.6}}}{\sigma\sqrt{N_q+n^{0.6}}}\right)| J_Z \right) \P(J_Z) \\
         = &\tilde{\E}\left(\exp\left(-\tilde{\theta}_n \frac{Z_{N_q+n^{0.6}}-Z_0}{n^{1/3}}  \right) f\left(\frac{Z_{N_q+n^{0.6}}}{\sigma\sqrt{N_q+n^{0.6}}} \right)| J_Z \right)\tilde{\P}(J_Z).
    \end{align*}
    By \eqref{zlcond} and \eqref{lo-J-z-tilde}, taking the limit $n\to\infty$ in the above inequality, we arrive at
    \begin{align}
        & \liminf_{n\to+\infty} n^{1/3} \E\left(f\left(\frac{Z_{N_q+n^{0.6}}}{\sigma\sqrt{N_q+n^{0.6}}}\right)| J_Z \right) \P(J_Z) \nonumber \\
        \geq & C_{\ref{lo-J-z-tilde}}(q) \E\left(e^{-2C_{\ref{cqdef-tilde}}(q)\tilde{\chi}_q} f(\tilde{\chi}_q)\right)= \frac{c_{q,Z}}{\E\left(\exp\left(- 2C_{\ref{cqdef-tilde}}(q) \tilde{\chi}_q \right)  \right)} \E\left(e^{-2C_{\ref{cqdef-tilde}}(q)\tilde{\chi}_q} f(\tilde{\chi}_q)\right),\label{lo-func-x}
    \end{align}
    where  we used the definition of $c_{q,Z}$ in \eqref{def-of-cqz} in the last step. Now we define another random variable $\zeta_q$ by 
    \begin{equation*}
        \P(\zeta_q\in \cdot ) := \frac{1}{\E\left(\exp\left(- 2C_{\ref{cqdef-tilde}}(q) \tilde{\chi}_q \right)  \right)} \E\left(e^{-2C_{\ref{cqdef-tilde}}(q)\tilde{\chi}_q} \1[\tilde{\chi}_q \in \cdot]\right).
    \end{equation*}
    Recall that $\tilde{\chi}_q \cd B_1^+$ and $C_{\ref{cqdef-tilde}}(q) \to 0$ as $q\to 0$ (see \eqref{eq:tilchiq} and \eqref{cqdef-tilde}).  
    We thus conclude that $\zeta_q \cd  B^+_1$ as $q\to 0$. Combining \eqref{lim-survY}, \eqref{YZU} and \eqref{lo-func-x}, we derive the first inequality of \eqref{ineq1}, which also proves Corollary \ref{phase1}.
\end{proof}

It remains to finish the proof of Proposition \ref{prop:decayyzep}.
    \begin{proof}[Proof of Proposition \ref{prop:decayyzep}]
Let us first consider $\P(\abs{V_Y(n)} > \ep)$.
By Lemma \ref{yzmean}, we can take $n$ sufficiently large such that 
$$\frac{2n^{0.6}}{\sigma\sqrt{N_q}} \abs{\E(Y_1-Y_0)} <\ep/2,$$
which, together with  Doob's maximal inequality, implies that, for any integer $k\geq 2$,
\begin{equation*}
\begin{split}
       \P(\abs{V_Y(n)} > \ep) 
       &\leq  \P\left(\frac{1}{\sigma\sqrt{N_q}} \max_{i\leq 2n^{0.6} } \abs{Y_i-Y_0 -\E(Y_i-Y_0 )}  +\frac{2n^{0.6}}{\sigma\sqrt{N_q}} \abs{\E(Y_1-Y_0)} > \ep\right)\\  
& \leq        \P\left(\frac{1}{\sigma\sqrt{N_q}} \max_{i\leq 2 n^{0.6} } \abs{Y_i-Y_0 -\E(Y_i-Y_0 )}   > \ep/2\right)\\
& \leq  \frac{2^k(k/(k-1))^k}{\ep^k\sigma^k\sqrt{N_q^k}} 
    \E\left(\abs{Y_{2 n^{0.6}}-Y_0 -\E(Y_{2 n^{0.6}}-Y_0 )}  ^k \right).
\end{split}
\end{equation*}
For even integer $k$,  by  \cite[p$.$ 60, Supplement 16]{MR388499},   there exist some constant $C(k)$  such that
\begin{equation}\label{yk}
        \E\left(\abs{Y_{2 n^{0.6}}-Y_0 -\E(Y_{2 n^{0.6}}-Y_0 )}  ^k \right)
        \leq C(k) (n^{0.6})^{k/2} \E[( ( Y_1-Y_0)-\E(Y_1-Y_0))^k]. 
\end{equation}
Using \eqref{bdpkn} and the distribution of $(Y_1-Y_0)$ specified below Lemma \ref{compbase},
the right-hand side of \eqref{yk} is further bounded by $ \widetilde{C}(k) n^{0.3k}$,  
where $\widetilde{C}(k)$ is another positive constant. 
Combining these estimates,  fix any positive even integer $k$,   
$$
\P(\abs{V_Y(n)} > \ep) =O( n^{-(\frac{1}{2}-0.3)k}),
$$
as desired. The probability $\P(\abs{V_Z(n)} > \ep)$ can be handled similarly and details are omitted. 

\end{proof}

\section{Upper and lower bounds for $X_{I,t}$}\label{sec:pflem1}

The goal of this section is to derive some estimates for $X_{I,t}$ (the number of infected half-edges at time $t$) that will serve as the main ingredients in the proofs of Theorems \ref{phase2} and \ref{survivalcond}. We will achieve this by deriving and analyzing the first-order approximation for the drift term in the evolution equation of $X_{I,t}$.  Recall the definitions of the quantities $X_t, S_{t,k}, X_{S,t}$ in the beginning of Section \ref{sssec:threestage}.

\subsection{Evolution equation for $X_{I,t}$}
We start from the evolution of $X_t$, which  is given by (see \cite[Equation (3.14)]{MR3275704} and \cite[Equation (4.3.10)]{MR4474532})
\begin{equation}\label{newxt}
    \mathrm{d} \X_t=-2(\X_t-1)\mathrm{d}t+\mathrm{d}M_{0,t},
\end{equation}
where $(M_{0,t})$ is a martingale. For $S_{t,k}$, \cite[Equation (3.2.3)]{MR4474532} reads
\begin{equation}\label{newStk}
    \mathrm{d}S_{t,k} = -kS_{t,k} \mathrm{d}t +\1[k\geq 1] \frac{\rho}{\lambda}\frac{S_{t,k-1}}{n}(X_t-1) \mathrm{d}t
    - \frac{\rho}{\lambda}\frac{S_{t,k}}{n}(X_t-1) \mathrm{d}t+\mathrm{d} M_{k,t},
\end{equation}
where $(M_{k,t})$ is a martingale for each $k$. Summing \eqref{newStk} over $k$ from $0$ to $\infty$ and noting that the second and third terms cancel, we get
\begin{equation}\label{newSt}
    \mathrm{d} S_t=-X_{S,t}\, \mathrm{d}t+ \mathrm{d}M_{1,t}
\end{equation}
for some martingale $(M_{1,t})$. Multiplying both sides of \eqref{newStk} by $k$ and summing over $k$, 
\begin{equation}\label{newXst}
    \mathrm{d} X_{S,t} = -\sum_{k=0}^{\infty} k^2S_{t,k}\mathrm{d} t +\frac{\rho}{\lambda}\cdot \frac{S_t}{n}(X_t-1)\mathrm{d}t+\mathrm{d} M_{2,t}
\end{equation}
for some martingale term $(M_{2,t})$.

Set $M_{I,t}=M_{0,t}-M_{2,t}$. Combining \eqref{newxt} and \eqref{newXst} and using $X_{I,t}=X_t-X_{S,t}$, we see that
\begin{equation}\label{-1xit0}
    \mathrm{d} X_{I,t} = \left(-2(X_t-1) +\sum_{k=0}^{\infty} k^2S_{t,k} -\frac{\rho}{\lambda}\cdot \frac{S_t}{n}(X_t-1)\right)\mathrm{d}t+ \mathrm{d} M_{I,t}.
\end{equation}
Moreover, if we multiply both sides of \eqref{newStk} by $k^2$ and $k^3$ respectively, then we get
\begin{equation}\label{newk^2Stk}
    \mathrm{d} \left(\sum_{k=0}^{\infty} k^2S_{t,k}\right)= \left(- \sum_{k=0}^{\infty} k^3S_{t,k}+2\frac{\rho}{\lambda} \cdot \frac{X_{S,t}}{n}(X_t-1)+\frac{\rho}{\lambda} \cdot \frac{S_t}{n}(X_t-1)\right)\mathrm{d} t +\mathrm{d} M_{3,t}
\end{equation}
and 
\begin{align}\label{newk^3stk}
    \mathrm{d} \left(\sum_{k=0}^{\infty} k^3S_{t,k}\right)=\bigg(-\sum_{k=0}^{\infty} k^4S_{t,k} 
    & +3\frac{\rho}{\lambda} \cdot \frac{\sum_{k=0}^{\infty} k^2S_{t,k}}{n}(X_t-1)
    +3\frac{\rho}{\lambda} \cdot \frac{X_{S,t} }{n}(X_t-1)\\
    &+ \frac{\rho}{\lambda} \cdot \frac{S_t}{n}(X_t-1) \bigg) \mathrm{d} t +\mathrm{d} M_{4,t}\nonumber
\end{align}
for some martingales $(M_{3,t})$ and $(M_{4,t})$. (The reason to include  equations \eqref{newk^2Stk} and \eqref{newk^3stk} here is that they will  be needed\footnote{Equations 
\eqref{newk^2Stk} and \eqref{newk^3stk} reveal that   a first-order approximation of $\sum_k k^2 S_{t,k}$ will necessarily involve a zero-th order approximation of $\sum_k k^3 S_{t,k}$.} in the proof of Lemma \ref{cor:driftes} below.)


\subsection{Approximations for the drift part of $X_{I,t}$}
The first step to understand the evolution of $X_{I,t}$ is to control the associated martingale terms. 
\begin{lemma}\label{lbstep1}
    There exists a positive constant $C_{\ref{s12}}$ such that for large $n$, with probability at least $(1-n^{-10})$, we have
    \begin{equation}\label{s11}
        \sum_{j=0}^4 \left(\sup_{0\leq t\leq \gamma_n\wedge 1} \abs{M_{j,t}}\right)+ \sup_{0\leq t\leq \gamma_n\wedge 1} \abs{M_{I,t}}\le n^{0.6}
    \end{equation}
    and 
    \begin{equation}\label{s12}
        \sup_{0\leq t\leq  \gamma_n\wedge 1} \sum_{k=0}^{\infty} (k+1)^4S_{t,k} \leq C_{\ref{s12}}n.
    \end{equation}
\end{lemma}

\begin{proof}
    The proof of Lemma \ref{lbstep1} follows from the same arguments as in \cite[Lemma 4.3]{MR4474532}. 
\end{proof}

Let $\Omega_{\ref{s11}}$ and $ \Omega_{\ref{s12}}$  be the events given by \eqref{s11} and \eqref{s12}, respectively.  Set 
\begin{equation}\label{3}
    \Omega_{\ref{3}}:= \Omega_{\ref{s11}}\cap \Omega_{\ref{s12}}.
\end{equation}
On the event $\Omega_{\ref{3}}$,  the following Lemma \ref{cor:driftes} provides a first-order approximation to the drift part of $X_{I,t}$ in \eqref{-1xit0}. Recall the definition of $\Delta$ in \eqref{Delta}. To simplify the presentation, without loss of generality we assume that $C_{\ref{concern1}}>0.6$, i.e., $C_{\ref{concern1}}\in (0.6, 2/3)$. (Otherwise, we can always replace $C_{\ref{concern1}}$ by 
$\max\{C_{\ref{concern1}}, 0.6+\epsilon\}$ for some small constant $\epsilon>0$.) 

\begin{lemma}\label{cor:driftes}
    On $\Omega_{\ref{3}}$, there exists a positive constant $C_{\ref{driftes}}$ such that for all $0 \leq t\leq \gamma_n \wedge 1$,  
    \begin{equation}\label{driftes}
        \abs{ \left(-2(X_t-1) +\sum_{k=0}^{\infty} k^2S_{t,k} -\frac{\rho}{\lambda}\frac{S_t}{n}(X_t-1)\right)-m_1 n\Delta t}\leq C_{\ref{driftes}} (nt^2+n^{C_{\ref{concern1}}} ).
    \end{equation}
\end{lemma}

\begin{proof}
The proof consists of four steps. In the first three steps we give  linear approximations for $X_t, S_t$ and $\sum_{k=0}^{\infty} k^2 S_{t,k}$ appearing   in \eqref{-1xit0}.   We then combine them and obtain \eqref{driftes}.

\emph{Step 1.} We first consider $X_t$.
    Combining \eqref{newxt} and  the fact $X_s\leq X_0$ for all $s\geq 0$, we see that on $\Omega_{\ref{3}}$,
    \begin{equation}\label{xt-x0}
        \abs{X_t-X_0}\leq \int_0^t 2X_0 \mathrm{d}s+\abs{M_{0,t}}\leq  2X_0t+n^{0.6}.
    \end{equation}
    Using \eqref{newxt} again, we deduce from the above inequality that 
    \begin{equation*}
        \abs{X_t-\left(1-2t\right)X_0 } \leq  2t+2\int_0^t \abs{X_s-X_0}\mathrm{d}s+\abs{M_{0,t}} \leq 2t+2\int_0^t  (2X_0s+n^{0.6}) \mathrm{d} s +n^{0.6},
    \end{equation*}
    which is bounded by $(2X_0t^2+4n^{0.6})$ for $t\leq 1$. Under Assumption {\bf (H3)}, we have $|X_0-m_1n|\leq n^{C_{\ref{concern1}} }$. Therefore, $X_t$ can be approximated by 
    \begin{equation}\label{hatxbd0}
    \begin{split}
        \abs{X_t-(1-2t)m_1n} & \leq 
        |X_t-(1-2t)X_0|+|X_0-m_1 n|\\
& \leq      2(m_1 n+n^{C_{\ref{concern1}}}) t^2+4n^{0.6}+n^{C_{\ref{concern1}}}   \\
        & \leq  C\left(n^{C_{\ref{concern1}}}+m_1 nt^2\right). 
        \end{split}
    \end{equation}
    
    \emph{Step 2.} We now analyze $S_t$ using a similar bootstrapping argument.
    On the event $\Omega_{\ref{3}}$, the drift part of $X_{S,t}$ in \eqref{newXst} is bounded by  
    \begin{equation}\label{bdd-k^2sx0}
        \sum_{k=0}^{\infty} k^2 S_{t,k}+\frac{\rho}{\lambda}X_0 \leq \sup_{0\leq t\leq \gamma_n\wedge 1} \sum_{k=0}^{\infty}(k+1)^4 S_{t,k} + \frac{\rho}{\lambda} (m_1 n+ n^{C_{\ref{concern1}}}) \leq C_{\ref{bdd-k^2sx0}}n
    \end{equation}
    for some positive constant $C_{\ref{bdd-k^2sx0}}$. Since $\abs{M_{2,t}}\leq n^{0.6}$ on $\Omega_{\ref{3}}$, we see that 
    \begin{equation}\label{hatk^1stkbd}
        \abs{\sum_{k=0}^{\infty} k\S_{t,k}-X_{S,0}}=\abs{X_{S,t}-X_{S,0}}\leq   C_{\ref{bdd-k^2sx0}}nt+n^{0.6}.
    \end{equation}
    Thus, on the event $\Omega_{\ref{3}}$, by \eqref{newSt}, we have the bound
    \begin{equation*}
        \abs{S_t-(S_0-X_{S,0}t)}\leq \int_0^t \abs{X_{S,u}-X_{S,0}}\mathrm{d}u+\abs{M_{1,t}}\leq C_{\ref{bdd-k^2sx0}}nt^2 +2n^{0.6}.
    \end{equation*}  
    Using $S_0=n-1$ and $X_{S,0}=X_{0}-X_{I,0}$ (with $X_{I,0}\leq \log(nC_{\ref{exp1}})/   \eta _ {\ref{exp1}}$ by \eqref{dmaxbd}), we see that
    \begin{equation}\label{hatstbd0}
        \abs{S_t-(1-m_1 t)n}\leq 3n^{0.6}+C_{\ref{bdd-k^2sx0}}nt^2+ n^{C_{\ref{concern1}}}t \leq C_{\ref{bdd-k^2sx0}} nt^2+ 4n^{C_{\ref{concern1}}}. 
    \end{equation}

    \emph{Step 3.} We now turn to  $\sum_{k=0}^{\infty} k^2 S_{t,k}$. 
    Combining \eqref{s12}, \eqref{hatxbd0}, \eqref{bdd-k^2sx0} and \eqref{hatk^1stkbd}, we find that the drift term of \eqref{newk^3stk} is bounded by $C n$ for some  constant $C$, while the martingale part is bounded by $n^{0.6}\leq n^{C_{\ref{concern1}}}$.  Therefore,  on $\Omega_{\ref{3}}$, 
    \begin{equation}\label{hatk^3stkbd}
        \abs{\sum_{k=0}^{\infty} k^3 \S_{t,k}- nm_3 }\leq C_{\ref{hatk^3stkbd}}\left(n^{C_{\ref{concern1}}}+ nt\right)
    \end{equation} 
    for some constant $C_{\ref{hatk^3stkbd}}>0$. Here  we also have used Assumption {\bf (H3)} to replace $\sum_{k=0}^{\infty} k^3 \S_{0,k}$ by $nm_3$. Consequently, by \eqref{newk^2Stk}, \eqref{hatxbd0}, \eqref{hatk^1stkbd},  \eqref{hatstbd0} and \eqref{hatk^3stkbd},  there exists some constant $C_{\ref{hatk^2stkbd}}>0$ such that on $\Omega_{\ref{3}}$, 
    \begin{equation}\label{hatk^2stkbd}
        \abs{\sum_{k=0}^{\infty} k^2\S_{t,k}-\left(m_2 n+\left(-m_3+2\frac{\rho}{\lambda}m_1^2+\frac{\rho}{\lambda}m_1 \right)nt \right) }\leq C_{\ref{hatk^2stkbd}}\left(n^{C_{\ref{concern1}}}+ nt^2\right).
    \end{equation} 

    \emph{Step 4.}     Combining \eqref{hatxbd0}, \eqref{hatstbd0} and \eqref{hatk^2stkbd}, we can approximate 
    \begin{equation*}
        -2(X_t-1) +\sum_{k=0}^{\infty} k^2S_{t,k} -\frac{\rho}{\lambda}\frac{S_t}{n}(X_t-1)
    \end{equation*}
    by
    \begin{equation}\label{approx}
        -2(1-2t)m_1n+\left(m_2 +\left(-m_3+2\frac{\rho}{\lambda}m_1^2+\frac{\rho}{\lambda}m_1 \right)t \right)n-\frac{\rho}{\lambda} (1-m_1t)(1-2t)m_1 n.
    \end{equation}
    Moreover, the error induced from this approximation is bounded by (for $t\leq 1$)
    \begin{equation}\label{erbd}
        C(nt^2+n^{C_{\ref{concern1}}})
    \end{equation}
    for some constant $C>0$. By the calculations in \cite[Proof of Lemma 4.4]{MR4474532}, we see that,  the expression in \eqref{approx} is equal to
    \begin{equation}\label{simpapp}
    n(m_1\Delta t-2m_1^2\rho t^2/\lambda).
    \end{equation}
    Lemma \ref{cor:driftes} then follows from \eqref{approx}, \eqref{erbd} and \eqref{simpapp}. 
\end{proof}

As an application of Lemma \ref{cor:driftes}, we have the following corollary.

\begin{corollary}\label{lbstep2}
    For $qn^{-1/3}\leq s\leq t\leq \gamma_n\wedge 1$, on $\Omega_{\ref{3}}$,
    \begin{equation*}
        \left| X_{I, t}-X_{I,s} - \frac{1}{2}m_1\Delta (t^2-s^2)n -\left( M_{I,t}-M_{I,s}\right)  \right|\leq C_{\ref{driftes}} \left( n(t^3-s^3)+n^{C_{\ref{concern1}}}(t-s)\right).
    \end{equation*}
\end{corollary}

\begin{proof}
    Combining \eqref{-1xit0} and Lemma \ref{cor:driftes},  on $\Omega_{\ref{3}}$, 
    \begin{align*}
        &\left| X_{I, t}-X_{I,s} -\int_s^t m_1n \Delta r \mathrm{d}r -\left( M_{I,t}-M_{I,s}\right)  \right|\\
        \leq & \int_s^t \left|\left(-2(X_r-1) + \sum_{k=0}^{\infty} k^2S_{r,k} -\frac{\rho}{\lambda}\frac{S_r}{n}(X_r-1)\right)-m_1 n\Delta r \right| \mathrm{d}r \\
        \leq & \int_s^t C_{\ref{driftes}} \left( nr^2+n^{C_{\ref{concern1}}}\right) \mathrm{d}r \leq C_{\ref{driftes}} \left( n(t^3-s^3)+n^{C_{\ref{concern1}}}(t-s)\right) ,
    \end{align*}
    as desired.
\end{proof}

 \section{Proof of Theorem \ref{phase2}}\label{sec:pfthm3}
 
In this section we consider $t \in [qn^{-1/3},Qn^{-1/3}]$ and prove Theorem \ref{phase2}.  Let $t=sn^{-1/3}$ so that the range of $s$ is $[q,Q]$. 
Define a rescaled version of $X_{I,t}$ for $q\leq s\leq  n^{1/3}\gamma_n$ by
\begin{equation*}
\hat{X}_{I,s}:= n^{-1/3} (X_{I,sn^{-1/3}}  -X_{I,qn^{-1/3}}). 
\end{equation*}
For the convenience of proofs in this and the next Section \ref{sec:pfthm4}, we 
let $M_{I,t}$ evolve by a suitably rescaled Brownian motion after time $\gamma_n$, i.e.,
$$
M_{I,t}=M_{I,\gamma_n}+   n^{1/3} C_{\ref{deff1xq}} B_{t-\gamma_n},
$$
and let
$\hat{X}_{I,s}$ evolve after time $n^{1/3}\gamma_n$ by  a Brownian motion with a quadratic drift,
$$
\hat{X}_{I,s}=\hat{X}_{I,n^{1/3}\gamma_n}+
\frac{m_1 \Delta}{2}(s^2-  (n^{1/3}\gamma_n)^2   )
+C_{\ref{deff1xq}} B_{s-n^{1/3}\gamma_n},
$$
where $\{B_{u}, u\geq 0\}$ is a standard Brownian motion independent of the evoSI process.

\begin{lemma}\label{lem:sde}
    For any fixed $0<q<Q$, conditioned on $\{\gamma_n>qn^{-1/3}\}$, $\{\hat{X}_{I,s},q\leq s\leq Q\}$ converges in law to
    \begin{equation*}
        A(s)= \frac{m_1\Delta }{2}(s^2-q^2)+C_{\ref{deff1xq}}B_{s-q}, \, \, \, \; q\leq s\leq Q,
    \end{equation*}
    where $\{B_u, u\geq 0\}$ is a standard Brownian motion. Moreover, $\{\hat{X}_{I,s} ,q\leq s\leq Q\}$ is asymptotically independent of $X_{I,qn^{-1/3}}$.
\end{lemma}

\begin{proof}
By Corollary \ref{lbstep2}, the drift part of $\hat{X}_{I,s}$ is well approximated by that of $A(s)$, so it remains to control its  martingale part 
$\tilde{M}_{I,s}:=n^{-1/3}(M_{I,sn^{-1/3}}-M_{I,qn^{-1/3}}).$

The quadratic variation of $\tilde{M}_{I,s}$  for $s\leq n^{1/3}\gamma_n$ stems from three parts:
    \begin{itemize}
        \item A susceptible vertex with $k\geq 1$  half-edges gets infected and increases $\hat{X}_{I,s}$ by $(k-2)n^{-1/3}$. This gives a contribution of $n^{-2/3}(k-2)^2$ to $\langle \tilde{M}_{I,s}, \tilde{M}_{I,s} \rangle $.
        \item Two infected half-edges pair with each other and decrease $\hat{X}_{I,t}$ by $2n^{-2/3}$. This provides a contribution of $4n^{-2/3}$.
        \item An infected half-edge is rewired to a susceptible vertex. This decreases $\hat{X}_{I,s}$ by $n^{-1/3}$ and increases  $\langle \tilde{M}_{I,s}, \tilde{M}_{I,s} \rangle $ by $n^{-2/3}$.
    \end{itemize}
    Let $U_n(s)$ be the compensator process  for $\langle \tilde{M}_{I,s},\tilde{M}_{I,s}\rangle$. Then for $s\leq n^{1/3}\gamma_n$,
    \begin{equation}\label{compm}
      U_n(s)=n^{-2/3}\int_{qn^{-1/3}}^{sn^{-1/3}} \left( \sum_{k=0}^{\infty} k(k-2)^2S_{r,k}  +4(X_{I,r}-1)+\frac{X_r}{\lambda} \cdot \frac{\rho S_r}{n} \right) \mathrm{d}r. 
    \end{equation}
    For $s\geq n^{1/3}\gamma_n$, we simply have 
    $$U_n(s)=U_n(n^{1/3}\gamma_n)+ C_{\ref{deff1xq}} (s-n^{1/3}\gamma_n).$$
    We now control the three terms in the integrand of \eqref{compm}.
    \begin{itemize}
        \item  Combining \eqref{hatk^1stkbd}, \eqref{hatk^3stkbd} and \eqref{hatk^2stkbd}, there exists $C_{\ref{kk-2}}>0$ such that on $\Omega_{\ref{3}}$,
    \begin{equation}\label{kk-2}
        \abs{\sum_{k=0}^{\infty} k(k-2)^2S_{t,k}-n(m_3-4m_2+4m_1)}
        \leq 
         C_{\ref{kk-2}}   n^{2/3}
        , \forall \, t\leq Qn^{-1/3}.
    \end{equation}
        \item  Combining Corollary \ref{lbstep2} and the fact that $X_{I,0}\leq \log(nC_{\ref{exp1}})/ \eta _ {\ref{exp1}}$, we get that 
    \begin{equation}\label{kk-3}
        X_{I,t} \leq X_{r,0} + C_{\ref{driftes}} \left( nt^3+ n^{C_{\ref{concern1}}}t\right) \leq n^{C_{\ref{concern1}} }+ C_{\ref{driftes}} \left( nt^3+ n^{C_{\ref{concern1}}}t\right) 
        \leq 2n^{C_{\ref{concern1}} }, \,
        .\forall \, t\leq Qn^{-1/3}
    \end{equation}
        \item Similarly, using  \eqref{hatxbd0} and \eqref{hatstbd0}, we find
        \begin{equation}\label{kk-6}
        \abs{ X_t \frac{S_t}{n}- m_1 n  }\leq  
C_{\ref{kk-6}}        n^{2/3}, \,
        \forall \, t\leq Qn^{-1/3}         \textnormal{.}
        \end{equation}
    \end{itemize}
    Combining    \eqref{kk-2}-\eqref{kk-6} and the definition of $U_n(s)$, we see that on $\Omega_{\ref{3}}$, for all $q\leq s\leq Q$,
    \begin{align*}
        &\abs{U_n(s)-n^{-2/3}\int_{qn^{-1/3}}^{sn^{-1/3}} n( m_3-4m_2+4m_1+\rho m_1/\lambda)) \mathrm{d}r} \\
        \leq & Cn^{-2/3} (Q-q)n^{-1/3}  \cdot n^{2/3},
    \end{align*}
    which converges to 0 as $n\to\infty$. Since $\lim_{n\to \infty}\P(\Omega_{\ref{3}}|\gamma_n>qn^{-1/3})= 1$ by Theorem \ref{YZLimit} and Lemma \ref{lbstep1}, we conclude that 
    \begin{equation}
        \sup_{q\leq s\leq Q}\abs{U_n(s)-(s-q) (m_3-4m_2+4m_1+\rho m_1/\lambda)}\to 0, \quad \textnormal{in probability}.
    \end{equation}
    Lemma \ref{lem:sde} now follows from \cite[Theorem 4.1 in Chapter 7]{MR838085} and the fact that  $\hat{X}_{I, q}\equiv 0$  (which implies the asymptotic independence between $\hat{X}_{I,s}$ and $X_{I, qn^{-1/3}}$).
\end{proof}

\begin{proof}[Proof of Theorem \ref{phase2}]
    We only prove in details for the first inequality (the lower bound part) in Theorem \ref{phase2}, since the upper bound follows from a similar argument\footnote{The upper bound  needs the fact that for any $x<0$, with probability 1, the following picture occurs: let $\tau$ be the first time that $A(s)$  hits $x$, if $\tau<\infty$, then there exists a  sequence of times $\tau_k \searrow \tau  $, such that $A(\tau_k)<0$ for each $k\geq 1$ (this can be seen from the law of iterated logarithm -- Brownian motion locally has super square-root fluctuations). This fact guarantees that $n^{1/3}\gamma_n$ is close to the first hitting time of $A(s)$.}. For any $\ep>0$, define
    \begin{equation}\label{ev:ep}
        \Omega_{\ref{ev:ep}}(\ep,q,Q):=\{\widehat{X}_{I.s}+n^{-1/3}X_{I,qn^{-1/3}}  >\ep, \forall qn^{-1/3}\leq t\leq Qn^{-1/3}\}.
    \end{equation}
On $ \Omega_{\ref{ev:ep}}(\ep)$,  $X_{I,t}$ stays positive, and satisfies the trivial bound  $$
X_{I,Qn^{-1/3}}\geq n^{1/3}  \hat{X}_{I,Q}.
$$
Thus we have the relation 
\begin{equation}\label{eq:rela}
    \{ X_{I,Qn^{-1/3}}\geq m_1 \Delta Q^2 n^{1/3}/4 \}  \subset  \Omega_{\ref{ev:ep}}(\ep,q,Q) \cap \{\hat{X}_{I,Q}\geq m_1 \Delta Q^2/4 \}. 
 \end{equation}
Define two events 
    \begin{equation}\label{omeqQ}
        \Omega_{\ref{omeqQ}}(\ep,q,Q):=\{\sqrt{\sigma^2 (1+\rho/\lambda)m_1 q}\zeta_q+A(s)\geq \ep, \forall q\leq s\leq Q\},
    \end{equation}
    and 
    \begin{equation}\label{omeq}
        \Omega_{\ref{omeq}}(\ep,q):=\{\sqrt{\sigma^2 (1+\rho/\lambda)m_1 q}\zeta_q+A(s)\geq \ep, \forall s\geq q\},
    \end{equation}
    By the law of iterated logarithm of for Brownian motion, we see that
    \begin{equation}\label{omeqlim}
        \lim_{Q\to\infty} \P(\Omega_{\ref{omeqQ}}(\ep,q,Q) ) =\P(\Omega_{\ref{omeq}}(\ep,q) ).
    \end{equation}
    Combining Corollary \ref{phase1},  Theorem \ref{YZLimit}, Lemma \ref{lem:sde} and equation \eqref{eq:rela},
    \begin{equation}\label{xiQ2}
        \begin{split}
            &\liminf_{n\to\infty}\P(X_{I,Qn^{-1/3}}\geq m_1 \Delta Q^2 n^{1/3}/4|\gamma_n>qn^{-1/3})\\
            \geq { } & (c_{q,Z}/ c_{q,Y}) \P( \Omega_{\ref{omeqQ}}(\ep,q,Q)\cap
            \{ A(Q)>m_1 \Delta Q^2/4\}).
        \end{split}
    \end{equation}
    By the definition of $A(Q)$ in Lemma \ref{lem:sde},
    \begin{equation}\label{AQbd}
        \lim_{Q\to\infty}  \P(A(Q)>m_1\Delta Q^2/4) \geq 
        \lim_{Q\to\infty} \P(B_{Q-q}\geq -(Q-q)) \geq 1-\lim_{Q\to\infty}(Q- q)^{-1}=1. 
    \end{equation}
    where we have used the Chebyshev's inequality is the second inequality.
    Combining \eqref{omeqlim}, \eqref{xiQ2} and \eqref{AQbd}, we deduce that  
    \begin{equation}\label{xiQ>q}
        \liminf_{Q\to \infty}  \liminf_{n\to\infty}\P(X_{I,Qn^{-1/3}}\geq m_1 \Delta Q^2 n^{1/3}/4|\gamma_n>qn^{-1/3}) \geq (c_{q,Z}/c_{q,Y})
        \P(\Omega_{\ref{omeq}}(\ep,q)).
    \end{equation}
    

    Since the left hand side is independent of $\ep$, we can take a limit $\ep\to 0$ in the right-hand side of  \eqref{xiQ>q} and obtain that 
    \begin{equation*}
        \liminf_{Q\to \infty}  \liminf_{n\to\infty}\P(X_{I,Qn^{-1/3}}\geq m_1 \Delta Q^2 n^{1/3}/4|\gamma_n>qn^{-1/3}) 
        \geq \frac{c_{q,Z}}{c_{q,Y}} \P(\Omega_{\ref{omeq}}(0,q))=
        \frac{c_{q,Z}}{c_{q,Y}} \E F_2(\zeta_q, q),
    \end{equation*}
    which we have used the definition of $F_1(x,q)$  in \eqref{deff1xq} in the last equality. This  completes the proof of Theorem \ref{phase2}.
\end{proof}
 
\section{Proof of Theorem \ref{survivalcond}}\label{sec:pfthm4}

In this section we prove Theorem \ref{survivalcond}. By Theorem \ref{YZLimit}, we can pick two small positive constants, $q_0<1$ and $c$, such that for all large $n$,
\begin{equation*}
    \P(\gamma_n>q_0 n^{-1/3})>cn^{-1/3}.
\end{equation*}
By Theorem  \ref{phase2},
 there exists a constant $C_{\ref{XIQQ}}>0$  such that for all large $Q$,
\begin{equation}\label{XIQQ}
 \liminf_{n\to\infty}   n^{1/3}\P(X_{I,Qn^{-1/3}}\geq m_1\Delta Q^2n^{1/3}/4)\geq C_{\ref{XIQQ}}.
\end{equation}
Define a new probability measure $\overline{\P}$ such that for any event $\Omega$,
\begin{equation*}
    \overline{\P}(\Omega):=\P\left(\Omega \mid \{X_{I,Qn^{-1/3}}>m_1\Delta Q^2  n^{1/3}/4 \} \right)=\frac{\P\left(\Omega \cap \{X_{I,Qn^{-1/3}}>m_1\Delta Q^2  n^{1/3}/4 \} \right)}{\P\left(X_{I,Qn^{-1/3}}>m_1 \Delta Q^2 n^{1/3}/4 \right)}
\end{equation*}
and let $\overline{\E}$ be the corresponding expectation operator. All the martingales appearing in previous sections, such as $(M_{I,t})$, are still martingales  with respect to $\overline{\P}$ if we only consider $t\geq Qn^{-1/3}$. Fix any positive constant $t_*$ such that 
\begin{equation}\label{cond1}
    t_*\leq \frac{m_1\Delta }{4C_{\ref{driftes}}} .
\end{equation}
To prove Theorem \ref{survivalcond}, we divide the time interval $[Qn^{-1/3}, t_*]$ into $( N_{\ref{maxQ}}+1 ) $ pieces, 
\begin{equation*}
    [Qn^{-1/3},2Qn^{-1/3}], [2Qn^{-1/3},4Qn^{-1/3}], \ldots, [2^{N_{\ref{maxQ}}-1}Qn^{-1/3},2^{N_{\ref{maxQ}}}Qn^{-1/3}],[2^{N_{\ref{maxQ}}} Qn^{-1/3},t_*],
\end{equation*}
where 
\begin{equation}\label{maxQ}
    N_{\ref{maxQ}}=  N_{\ref{maxQ}}(Q):=\sup\{\ell\in \mathbb{Z}: 2^{\ell}Qn^{-1/3}<t_* \}. 
\end{equation}
Define 
\begin{equation}\label{lu2}
    \Omega_{\ref{lu2}}:=  \bigcap_{0\leq \ell \leq N_{\ref{maxQ}}} \left\{\sup_{2^{\ell}Qn^{-1/3}\leq u\leq t_*\wedge 2^{\ell+1}Qn^{-1/3}}\abs{M_{I,u}-M_{I,2^{\ell}Qn^{-1/3} } }\leq 2^{\ell} Qn^{1/3}\right\}.
\end{equation} 
    
\begin{lemma}\label{plu2}
    It holds that 
    \begin{equation*}
        \lim_{Q\to\infty}\liminf_{n\to\infty}\overline{\P}(\Omega_{\ref{lu2}})=1. 
    \end{equation*}
\end{lemma}

\begin{proof}
    Combining Lemma \ref{lbstep1}, \eqref{XIQQ} and the definition of $\overline{\P}$, for sufficiently large $n$, 
    \begin{equation}\label{omega3bd}
        \overline{\P}(\Omega_{\ref{3}}^c)\leq \frac{\P\left(\Omega_{\ref{3}}^c \right)}{\P\left(X_{I,Qn^{-1/3}}>m_1 \Delta Q^2 n^{1/3}/4 \right)}\leq  2n^{-10+1/3}/C_{\ref{XIQQ}}. 
    \end{equation}
    Denote by $\psi_n(t)$ the integrand in  \eqref{compm}, i.e., for $t\leq \gamma_n$,
    \begin{equation*}
        \psi_n(t)= \sum_{k=0}^{\infty} k(k-2)^2S_{t,k} +4(X_{I,t}-1)+\frac{X_t}{\lambda}\frac{\rho S_t}{n}.
    \end{equation*}
    For $t\geq \gamma_n$,  $\psi_n(t)$ is equal to
be   $C_{\ref{deff1xq}} n$.    As  $(M_{I,t})$ is still a martingale under  $\overline{\E}$,  we have
    \begin{equation}\label{etabd0}
        \overline{\E}\left((M_{I,t_*\wedge 2^{\ell+1}Qn^{-1/3}}-M_{I,2^{\ell}Qn^{-1/3}})^2 \right) = \overline{\E}\left(\int_{2^{\ell} Qn^{-1/3}}^{t_*\wedge 2^{\ell+1}Qn^{-1/2}} \psi_n(t)\mathrm{d}t\right).
    \end{equation}
      From \eqref{s12}, there exists a constant $C_{\ref{etabd1}}>0$ such that 
    \begin{equation}\label{etabd1}
        \psi_n(t)\leq C_{\ref{etabd1}}n \textnormal{ on } \Omega_{\ref{3}}.
    \end{equation}
   Noticing $S_{t,k}=0$ for $k>X_0$ and $X_0\leq 2m_1 n$, we see that  for some constant $C_{\ref{etabd2}}>0$,
    \begin{equation}\label{etabd2}
        \psi_n(t)\leq 
        \max\{  n X_0^3+4X_0+\frac{\rho}{\lambda} X_0, C_{\ref{deff1xq}} n\}
        \leq C_{\ref{etabd2}}n^4   \quad \mbox{ on }\Omega_{\ref{3}}^c,
    \end{equation}
    Therefore, combining \eqref{omega3bd}-\eqref{etabd2},  for large enough $n$, we have 
    \begin{equation}\label{etabd3}
        \begin{split}
            &\overline{\E}\left((M_{I,t_*\wedge 2^{\ell+1}Qn^{-1/3}}-M_{I,2^{\ell}Qn^{-1/3}})^2 \right) \\
            = &\overline{\E}\left(\int_{2^{\ell} Qn^{-1/3}}^{t_*\wedge 2^{\ell+1}Qn^{-1/3}} \psi_n(t)\1[\Omega_{\ref{3}}]\mathrm{d}t\right)+ \overline{\E}\left(\int_{2^{\ell} Qn^{-1/3}}^{t_*\wedge 2^{\ell+1}Qn^{-1/3}} \psi_n(t)\1[\Omega^c_{\ref{3}}]\mathrm{d}t\right) \\
            \leq & 2^{\ell} Qn^{-1/3} C_{\ref{etabd1}}n + 2^{\ell+1} Qn^{-1/3} C_{\ref{etabd2}}n^4 n^{-10+1/3}/C_{\ref{XIQQ}} \leq C_{\ref{etabd3}} 2^{\ell}Qn^{2/3}
        \end{split}
    \end{equation}
    for some constant $C_{\ref{etabd3}}>0$. Using the $L^2$ maximal inequality, for any $0\leq \ell\leq N_{\ref{maxQ}}$,
    \begin{align*}
        & \overline{\P}\left(\sup_{2^{\ell}Qn^{-1/3}\leq u\leq t_*\wedge 2^{\ell+1}Qn^{-1/3}}\abs{M_{I,u}-M_{I,2^{\ell}Qn^{-1/3} } }>2^{\ell} Qn^{1/3}\right)\\  
          \leq & \frac{4}{(2^{\ell}Qn^{1/3})^2}  \overline{\E}\left((M_{I,t_*\wedge 2^{\ell+1}Qn^{-1/3}}-M_{I,2^{\ell}Qn^{-1/3}})^2 \right)\\
          \leq  & \frac{4C_{\ref{etabd3}}2^{\ell}Qn^{2/3}}{2^{2\ell}Q^2n^{2/3}} = 4C_{\ref{etabd3}}  2^{-\ell}Q^{-1}.
    \end{align*}
Finally,  by the union bound, 
    \begin{equation*}
        \overline{\P}(\Omega_{\ref{lu2}}) \geq 1-\sum_{\ell=0}^{N_{\ref{maxQ}}}  4C_{\ref{etabd3}}  2^{-\ell}Q^{-1} \geq 1-8C_{\ref{etabd3}} Q^{-1},
    \end{equation*}
    which proves Lemma \ref{plu2} by sending $n\to\infty$ and then $Q\to \infty$.
\end{proof} 

\begin{lemma}\label{lem:ind}
    Fix a sufficiently large $Q>0$ and condition on    $\{X_{I,Qn^{-1/3}}\geq m_1\Delta Q^2n^{1/3}/4\}$.  On the event $\Omega_{\ref{lu2}}$,  we have
    \begin{equation}\label{ind2}
        X_{I,2^{\ell}Qn^{-1/3}}  \geq  m_1 \Delta n(2^{\ell}Qn^{-1/3} )^2/16
    \end{equation}
    for all $0\leq \ell \leq N_{\ref{maxQ}}$. Moreover, 
    \begin{equation}\label{xit1bd}
        X_{I,t_*} \geq m_1 \Delta n t_*^2/128>0.
    \end{equation}
    In particular, we have $\gamma_n>t_*$ on $\Omega_{\ref{lu2}}$.
\end{lemma}

\begin{proof}
    We prove \eqref{ind2} by  induction in $\ell$.  We start from the base case where $\ell=0$. Then \eqref{ind2} holds trivially since we already condition on the event $\{X_{I,Qn^{-1/3}}\geq m_1\Delta Q^2n^{1/3}/4\}$. Assume that \eqref{ind2} holds for some $\ell\geq 0$. We will show that \eqref{ind2} holds for $\ell+1$. This gives the induction step.
    
    Combining Corollary \ref{lbstep2} and the definition of $\Omega_{\ref{lu2}}$, for $s:=2^{\ell}Qn^{-1/3}\leq t\leq  2^{\ell+1}Qn^{-1/3}$, 
    \begin{equation}\label{eq:xit-s}
    \begin{split}
        X_{I,t}& \geq X_{I,s} +\frac{1}{2}m_1\Delta(t^2-s^2) n + (M_{I,t}-M_{I,s}) -C_{\ref{driftes}} \left( n(t^3-s^3)+n^{C_{\ref{concern1}}}(t-s)\right)\\
        & \geq   X_{I,2^{\ell}Qn^{-1/3}} +\frac{1}{2}m_1\Delta(t^2-s^2) n -C_{\ref{driftes}} n(t^3-s^3)-C_{\ref{driftes}}n^{C_{\ref{concern1}}-1/3}2^{\ell+1}Q -2^{\ell} Q n^{1/3}. 
        \end{split}
    \end{equation} 
    By the condition \eqref{cond1}, we see that
    $$
    t+s \leq 2t\leq 2t_* \leq \frac{m_1 \Delta}{2 C_{\ref{driftes}}},$$
    and therefore
    \begin{align*}
        & \frac{1}{2}m_1\Delta(t^2-s^2) n -C_{\ref{driftes}} n(t^3-s^3)\\
        = & n(t-s)\left[t\left(\frac{m_1\Delta }{2}- C_{\ref{driftes}}t\right) +s\left( \frac{m_1\Delta}{2} - C_{\ref{driftes}} (t+s)\right)\right] \\
        \geq & nt(t-s)\left(\frac{m_1\Delta }{2}- C_{\ref{driftes}}t\right) \geq \frac{m_1 \Delta nt (t-s)}{4}.
    \end{align*}
    Therefore,  by \eqref{eq:xit-s} and the assumption $C_{\ref{concern1}} <2/3$, we conclude that 
    \begin{equation}\label{minxit}
        X_{I,t} \geq  X_{I,2^{\ell}Qn^{-1/3}}+ \frac{m_1 \Delta nt}{4}(t-2^{\ell}Qn^{-1/3}) -2^{\ell} Qn^{1/3} (1+2C_{\ref{driftes}}).
    \end{equation}
    By the induction hypothesis, i.e., \eqref{ind2} for $\ell$, the first term $X_{I,2^{\ell}Qn^{-1/3}}$ is bounded from below by $m_1 \Delta n(2^{\ell}Qn^{-1/3} )^2/16$. For the third term, if $Q$ is large enough, then for all $\ell\geq 0$,
    \begin{equation}\label{thirdfourth}
        2^{\ell} Qn^{1/3} (1+2C_{\ref{driftes}}) = n(2^{\ell}Qn^{-1/3} )^2 \frac{1+2C_{\ref{driftes}} }{2^\ell Q} \leq  \frac{1}{100}m_1 \Delta n(2^{\ell}Qn^{-1/3} )^2.
    \end{equation}
    Therefore, combining \eqref{minxit}, \eqref{thirdfourth} and the induction hypothesis,
    \begin{equation}\label{-2xitbd}
        X_{I,t}\geq \frac{ m_1 \Delta n(2^{\ell}Qn^{-1/3} )^2}{16}+ \frac{m_1\Delta nt(t-2^{\ell}Qn^{-1/3}) }{4} -\frac{ m_1 \Delta n(2^{\ell}Qn^{-1/3} )^2}{100}>0
    \end{equation}
     for all $2^{\ell}Qn^{-1/3}\leq t\leq  2^{\ell+1}Qn^{-1/3}$. This implies that $\gamma_n\geq  2^{\ell+1} Qn^{-1/3}$. Using \eqref{-2xitbd} again with $t=2^{\ell+1}Qn^{-1/3}$, we deduce that
    \begin{equation*}
        X_{I,2^{\ell+1}Qn^{-1/3}}\geq \frac{m_1\Delta n (2^{\ell}Qn^{-1/3} )^2  }{2} \geq \frac{m_1 \Delta n(2^{\ell+1}Qn^{-1/3})^2 }{16},
    \end{equation*}
    which proves \eqref{ind2}  for the case of $(\ell+1)$ and completes the proof of  \eqref{ind2}. Now taking $\ell=N_{\ref{maxQ}}$ in \eqref{ind2} and \eqref{-2xitbd},  we conclude that 
    \begin{align*}
        X_{I,t_*}&\geq   m_1 \Delta n(2^{N_{\ref{maxQ}}}Qn^{-1/3} )^2/16-  m_1 \Delta n(2^{N_{\ref{maxQ}}}Qn^{-1/3} )^2/100 \\
        &\geq   nm_1 \Delta(2^{N_{\ref{maxQ}}}Qn^{-1/3} )^2/32\geq m_1\Delta t_*^2/128.
    \end{align*}
    This proves \eqref{xit1bd} and completes the proof of Lemma \ref{lem:ind}.
\end{proof}

\begin{proof}[Proof of Theorem \ref{survivalcond}]
    Taking $C_{\ref{const3}}:= t_*$ where $t_*$ is any fixed positive constant satisfying \eqref{cond1}.  By Lemma \ref{lem:ind}, for all large $Q>0$, we have $\overline{\P}(\gamma_n>C_{\ref{const3}}) \geq \overline{\P}(\Omega_{\ref{lu2}}) $. Combining this with Lemma \ref{plu2}, we see that 
    \begin{align*}
        & \lim_{Q\to\infty}\liminf_{n\to\infty} \P \left(\gamma_n>C_{\ref{const3}} \Big|X_{I,Qn^{-1/3}}>m_1Q^2 n^{1/3}/4  \right) \\
        = & \lim_{Q\to\infty} \liminf_{n\to\infty} \overline{\P}(\gamma_n>C_{\ref{const3}})  \geq \lim_{Q\to\infty} \liminf_{n\to\infty} \overline{\P}(\Omega_{\ref{lu2}})=1,
    \end{align*}
    which implies \eqref{const3} and completes the proof of Theorem \ref{survivalcond}. 
\end{proof}

\section{Proof of Theorem \ref{thm:AB}}\label{sec:pfthm1}

We need one more ingredient to prove Theorem \ref{thm:AB}. Recall the random variables $\zeta_q$ and $\chi_q$ in Corollary \ref{phase1} and the definition of $F_1(x,q)$ in \eqref{deff1xq}. 
Define two constants 
\begin{equation}\label{cfinal}
     C_{\ref{cfinal}}:=\frac{m_1\Delta}{2C_{\ref{deff1xq}}}, \quad  C'_{\ref{cfinal}}:=\frac{(4C_{\ref{cfinal}})^{1/3} \sigma}{C_{\ref{deff1xq}}} \sqrt{\left(1+\frac{\rho}{\lambda}\right)m_1}
\end{equation}
and a function
\begin{equation*}
  F_2(x,q):=
(2C^2_{\ref{cfinal}})^{1/3} \int_0^{\infty} \exp\left( xt (2C_{\ref{cfinal}}^2)^{1/3}
-\frac{2}{3}C^2_{\ref{cfinal}}t^3-2C_{\ref{cfinal}}^2qt^2-2C_{\ref{cfinal}}^2 q^2t\right)\mathrm{d}t. 
\end{equation*}

\begin{lemma}\label{lem:ef1}
Consider the case $\Delta>0$.    We have the convergence
\begin{equation}\label{f1lim}
    \begin{split}
        \lim_{q\to 0} \frac{\E(F_1(\chi_q,q)) }{\sqrt{q}}&= \lim_{q\to 0} \frac{\E(F_1(\zeta_q,q)) }{\sqrt{q}}\\
       =&-C'_{\ref{cfinal}}\left(\frac{\mbox{Ai}'(0)}{\mbox{Ai}(0)}+ \sum_{k=1}^{\infty}(F_2(z_k,0)+z_k^{-1})\right) \sqrt{\frac{\pi}{2}}
       =:C_{\ref{f1lim}}.
    \end{split}
\end{equation}
Moreover, the constant $C_{\ref{f1lim}}$ is strictly positive. 
\end{lemma}

\begin{proof}
We will  apply the formulas  for hitting probability of a Brownian motion with a parabolic curve
in \cite{MR938153, MR2738342}. 
 By \cite[Proposition 3.10]{MR938153} and \cite[Theorem 2.2]{MR2738342},
\begin{align}
    F_1(x,q)=1-&\exp\left(-\frac{2C_{\ref{cfinal}}\sigma q^{3/2}x}{C_{\ref{deff1xq}}} \sqrt{\left(1+\frac{\rho}{\lambda}\right)m_1} \right) \nonumber \\
    &\times \left(\frac{\mbox{Ai}(C'_{\ref{cfinal}}\sqrt{q}x )}{\mbox{Ai}(0)}+ \sum_{k=1}^{\infty} (F_2(z_k,q)+z_k^{-1}) \frac{\mbox{Ai}(z_k+C'_{\ref{cfinal}}\sqrt{q}x  )}{\mbox{Ai}'(z_k)} \right) \label{eq:series}
\end{align}
where Ai is the Airy function and $z_k$'s are the real negative zeros of $\mbox{Ai}$. 
Recall the asymptotics\footnote{The Airy function Ai is bounded on the  real line, while its derivative  Ai' satisfies  $\abs{\mbox{Ai}'(x)}\leq C( |x|^{1/4}+1)$ for $x\leq 0$  and stays bounded for $x\geq 0$.} of $\mbox{Ai}$ and $\mbox{Ai}'$
in \cite[Formulas (A.4) and (A.5)]{MR938153} as well as the 
the sharp bounds\footnote{We have $z_k=\Theta(k^{2/3})$ and $\mbox{Ai}'(z_k)=\Theta(k^{1/6})$.} for $z_k$ and Ai'($z_k$) in 
 \cite[Formulas (A.30) and (A.31)]{MR938153}. We find that the series\footnote{We have $\abs{F_2(x,q)+x^{-1}}+\abs{\partial_{\sqrt{q}} F_2(x,q) }\leq C (|x|+1)^{-2}$ for $x<-1$. Thus, the series in \eqref{eq:series} and its derivative (in $\sqrt{q}$) are both absolutely convergent.} in \eqref{eq:series}  is continuously differentiable in $\sqrt{q}$ and can be differentiated term-by-term.
Thus the function $F_3(x,q):=F_1(x,q)/\sqrt{q}$ satisfies, uniformly for $x$ in any compact set of $[0,\infty)$, 
\begin{equation}\label{eq:f3limit}
    \lim_{q\to 0}F_3(x,q)= -C'_{\ref{cfinal}}x\left(\frac{\mbox{Ai}'(0)}{\mbox{Ai}(0)}+ \sum_{k=1}^{\infty}(F_2(z_k,0)+z_k^{-1})\right).
\end{equation}
Moreover, for some $C_{\ref{cfi2}}>0$,
\begin{equation}\label{cfi2}
    F_3(x,q)\leq C_{\ref{cfi2}}x \quad   \text{ for all }x\geq 0 \text{ and }q\leq 1.
\end{equation}
Consequently, by Lemma \ref{b+converge}, the random variables $F_3(\chi_q,q)$ and $F_3(\zeta_q,q)$ are uniformly integrable in $0< q<1$. Equation  \eqref{f1lim} now follows from \eqref{eq:f3limit}, the weak convergence of $\xi_q$ and $\zeta_q$ (as stated in Corollary \ref{phase1}) and the fact $\E(B_1^+)=\sqrt{\pi/2}$,  

To prove the positivity of $C_{\ref{f1lim}}$, by the probabilistic definition of $F_1$ in  \eqref{deff1xq}, 
$F_1$ is an increasing function of $x$. In addition, for some $c>0$, $ F_1(1,q)$ can be bounded below by
\begin{equation*}
         \P \left( \min_{0\leq t\leq 1} B_t \geq -\sigma\sqrt{(1+\rho/\lambda)m_1q}/2 \right)  \times  
        \P \left( \frac{m_1 \Delta}{2}((q+s)^2-q^2) +B_s>0,  \forall s\geq 1\right)\geq c\sqrt{q},
\end{equation*}
where we have used the reflection principle for $\min_{0\leq t} B_t$. Hence the limit in  \eqref{f1lim} is strictly positive and we finish the proof of Lemma \ref{lem:ef1}.
 \end{proof}
 
 We prove Theorem \ref{thm:AB} next.

\begin{proof}[Proof of Theorem \ref{thm:AB}]
    We first prove the case  $\Delta>0$ in two steps.
    
   \emph{Step 1.} We analyze  the asymptotics of $\gamma_n$.  Fix any $0<q<Q$, we have 
    \begin{equation}\label{gammalb}
        \begin{split}
            n^{1/3}\P(\gamma_n\geq C_{\ref{const3}})\geq &n^{1/3}\P(\gamma_n>qn^{-1/3})\times  \P(\X_{I,Qn^{-1/3}}>m_1\Delta Q^2n^{1/3}/4 |\gamma_n>qn^{-1/3})\\
            &\times \P(\gamma_n\geq C_{\ref{const3}}| X_{I,Qn^{-1/3}} >m_1\Delta Q^2n^{1/3}/4 ).
        \end{split}
    \end{equation}
    By letting $n\to\infty$ and then $Q\to\infty$ in \eqref{gammalb}, and by Theorems \ref{YZLimit}-\ref{survivalcond}, we deduce that
    \begin{equation}\label{nQ}
        \liminf_{n\to\infty} n^{1/3}\P(\gamma_n\geq C_{\ref{const3}})\geq c_{q,Z}\frac{c_{q,Z}}{c_{q,Y}}  \E(F_1(\zeta_q,q)).
    \end{equation}
    Letting $q\to 0$ in \eqref{nQ}, by Theorem \ref{YZLimit} and equation \eqref{f1lim}, we have 
    \begin{equation}\label{qto00}
        \liminf_{n\to\infty} n^{1/3}\P(\gamma_n\geq C_{\ref{const3}})\geq \left( \frac{\pi \sigma^2}{2m_1} \left(1+\frac{\rho}{\lambda}\right) \right)^{-1/2}C_{\ref{f1lim}}=:C_{\ref{eq:del>0case}}>0.
    \end{equation}
    Note that we used the fact $C_{\ref{f1lim}}>0$ in the last equality.
    
    For the other direction, taking any $\delta>0$, when $n$ is large enough, we have
    \begin{equation*}
        n^{1/3}\P(\gamma_n\geq \delta)\leq n^{1/3}\P(\gamma_n>qn^{-1/3}) \times \P( \gamma_n>Qn^{-1/3} |  \gamma_n>qn^{-1/3}).
    \end{equation*}
    Sending $n\to\infty$ and then $Q\to\infty$, we find 
    \begin{equation}\label{supgamma}
        \limsup_{n\to\infty} n^{1/3}\P(\gamma_n\geq \delta ) \leq  c_{q,Y}\frac{c_{q,Y}}{c_{q,Z}}  \E(F_1(\zeta_q,q)).
    \end{equation}
    Similarly to \eqref{qto00}, by letting $q\to 0$, we see that
    \begin{equation}\label{2qto00}
        \limsup_{n\to\infty} n^{1/3}\P(\gamma_n\geq \delta )\leq C_{\ref{eq:del>0case}}.
    \end{equation}
 Now it follows from \eqref{qto00} and \eqref{2qto00} that  for any $\delta<C_{\ref{const3}}$,
    \begin{equation}\label{asympsurv}
        \lim_{n\to\infty}n^{1/3}\P(\gamma_n\geq \delta)=C_{\ref{eq:del>0case}}. 
    \end{equation}

    \emph{Step 2}. We transfer the asymptotics of $\P(\gamma_n>\delta)$ to the probability of a major outbreak.    Recall from  equation \eqref{hatstbd0} that, on the event $\Omega_{\ref{3}}$, which has probability at least $(1-n^{-10})$, we have
    \begin{equation}\label{bounds-for-i}
        \abs{\frac{I_t}{n}-m_1 t}=\abs{\frac{S_t}{n}-(1-m_1 t)}\leq C_{\ref{bdd-k^2sx0}}t^2+4n^{C_{\ref{concern1}}-1},\quad \forall 0\leq t\leq \gamma_n\wedge 1.
    \end{equation}
    Therefore,  on $\Omega_{\ref{3}}$,
    the final epidemic size $\hat{\Lambda}_n$ satisfies
    \begin{equation}\label{idiff}
        \frac{\widehat{\Lambda}_n}{n}  \geq \frac{I_t}{n} \geq m_1 t - C_{\ref{bdd-k^2sx0}}t^2 -4n^{C_{\ref{concern1}}-1},\quad \forall 0\leq t\leq \gamma_n\wedge 1.
    \end{equation}
Fix a $\delta_0<\min\{C_{\ref{const3}},1\}$ such that $\epsilon_0:= m_1 \delta_0 - 2C_{\ref{bdd-k^2sx0}}\delta_0^2 >0$. 
By \eqref{asympsurv} and \eqref{idiff},   for all $\ep<\ep_0$,
        \begin{equation}\label{lowtn}
        \liminf_{n\to\infty}n^{1/3}\P(\widehat{\Lambda}_n>\ep n)\geq C_{\ref{eq:del>0case}}. 
    \end{equation}
    On the other hand, for any $\epsilon>0$, choosing $\delta< \min\{C_{\ref{const3}},1\}$ sufficiently small such that  $2C_{\ref{bdd-k^2sx0}}\delta^2+m_1\delta <\epsilon$, we claim that
    \begin{equation}\label{eq:finalclaim}
        \Omega_{\ref{3}} \cap \{\gamma_n <\delta\}  \subset \{\widehat{\Lambda}_n\leq \ep n\}.
    \end{equation}
 This can be seen from    
   \eqref{bounds-for-i} (applied with $t=\gamma_n$), which gives 
    \begin{equation*}
        \frac{\widehat{\Lambda}_n}{n}=\frac{I_{\gamma_n}}{n}\leq m_1 \gamma_n+ C_{\ref{bdd-k^2sx0}}\gamma_n^2+4n^{C_{\ref{concern1}}-1} \leq m_1 \delta+ 2C_{\ref{bdd-k^2sx0}}\delta^2<\ep.
    \end{equation*}
 From \eqref{asympsurv} and \eqref{eq:finalclaim}, we conclude that
    \begin{equation}\label{upptn}
        \limsup_{n\to\infty} n^{1/3}\P(\widehat{\Lambda}_n>\ep n)\leq \limsup_{n\to\infty} n^{1/3}\P(\gamma_n \geq \delta)=C_{\ref{eq:del>0case}}.
    \end{equation}
    Following from \eqref{lowtn} and \eqref{upptn}, we complete the proof of the case $\Delta>0$. The case of $\Delta<0$ also follows from \eqref{supgamma} and \eqref{upptn}, since $F_1(x,q)\equiv 0$ when $\Delta<0$. 
\end{proof}

\section{Analysis of the AB-avoSI process}\label{sec:ABavo}
Recall from Section \ref{sssec:evoavo} that the AB-avoSI process is established as a lower bound for the evoSI process. In this section, we outline the key modifications required to prove Theorem \ref{thm:avo} for AB-avoSI process. To avoid notational ambiguity, we use a bar to denote various quantities associated with the AB-avoSI process  corresponding  to their avoSI counterpart. Our analysis begins with the evolution equation for $\bar{X}_{I,t}$ (cf$.$ \cite[Equation (4.3.12)]{MR4474532}),
\begin{equation}\label{new-1xit0}
	\begin{split}
	\mathrm{d}\bar X_{I,t}=&\left(-2(\bar X_t-1)
	+\sum_{k=0}^{\infty} k^2\bar S_{t,k}
	-\frac{\rho}{\lambda}\frac{\bar S_t}{n}(\bar X_t-1)\right)\mathrm{d}t- E_t \mathrm{d}t+
	\mathrm{d}\bar M_{I,t},
	\end{split}
\end{equation}
where $(\bar M_{I,t})$ is a martingale. The additional error term\footnote{Here $\bar I(i,t)=1$ if the half-edge $i$  is infected  at time $t$; $\bar S(j,t)=1$ if  half-edge $j$ is susceptible at time $t$. $D(j,t)$ denotes the degree (number of half-edges) of the vertex to which half-edge $j$ is attached at time $t$. In the equation below \eqref{kgeq3}, $S(j,k,t)=1$ if $\bar{S}(j,t)=1$ and $D(j,t)=k$. For further details, see \cite[Section 4.2]{MR4474532}.} in this equation is given by 
\begin{equation}\label{defet2}
	E_t = \frac{1}{\bar X_{I,t}} \left( \sum_{i,j=1}^{\bar X_0}  (D(j,t)-1) \cdot \1[G_{i,j}(t)\cap \{\bar S(j,t)=1\}] \right) ,
\end{equation}
where $G_{i,j}(t) = \left \{\bar I(i,t)=1,\   A(i,t)\leq B(j,t) \right \}$.

As with the avoSI process, we split the analysis into three stages according to the relative size of the drift term and the martingale term, where the error term $E_t$ is shown to be negligible. 

\emph{Stage 1 $(0\leq t\leq qn^{-1/3}$).} This first stage relies once more on a random walk comparison argument. Let $\bar U_{\ell}$ be the number of infected half-edges in AB-avoSI after the $\ell$-th jump. The transition rules for $(\bar{U}_{\ell})$ are similar to that of $(U_{\ell})$ but need slight modifications, e.g., for $k\geq 3$, the increment $\bar U_{\ell+1}-\bar U_{\ell}=k-2$ occurs with probability
\begin{equation}\label{kgeq3}
    \frac{\lambda}{(\lambda+\rho)\bar X_{I,t} (\bar X_t-1)}N_{t,k},
\end{equation} 
where 
\begin{equation*}
    N_{t,k}:=\sum_{i,j=1}^{\bar X_0} \1 [ \bar I(i,t)=1,A(i,t)> B(j,t),S(j,k,t)=1) ].
\end{equation*}
One can show that Lemmas \ref{lemjump} and \ref{compbase} are still valid\footnote{The number of jumps  $N_q$ remains the same while the value of constant $C_{\ref{u1xit}}$  can be different. Additionally we need to replace $S_{t,k}$ in the second line of \eqref{u1xit} with $N_{t,k}/\bar X_{I,t}$.} here. Hence, the random walks $(Y_{\ell})$ and $(Z_{\ell})$  defined previously (with the same variance $\sigma^2$) remain to be the upper and lower bounds of $(\bar U_{\ell})$. Consequently, Theorem \ref{YZLimit} and Corollary \ref{phase1} also hold for the AB-avoSI process. The new constants $c_{q,Y}$ and $c_{q,Z}$ can possibly take different values in this setting. However,  the limit in \eqref{eq:limcqyz} remains unaffected (which is enough for our purpose).

\emph{Stage 2 ($qn^{-1/3}\leq t \leq Qn^{-1/3}$)} and \emph{Stage 3 ($t\geq Qn^{-1/3}$).} The analysis of these two stages is based on estimating each term in \eqref{new-1xit0}. Lemmas \ref{lbstep1}  and \ref{cor:driftes} are still true. For Corollary \ref{lbstep2}, the upper bound for $(\bar X_{I,t}- \bar X_{I,s})$ still holds, but we need to modify the lower bound by
$$
	\bar X_{I,t} -\bar X_{I,s}\geq m_1\Delta (t^2-s^2)n/2 -C (nt^2(t-s)+n^{C_{\ref{concern1}}} t)+(\bar M_{I,t}-\bar M_{I,s})
	-\int_s^t E_u \mathrm{d}u.
$$
To handle the error term $E_t$, we generalize the approach\footnote{In \cite{MR4474532}, $b$ is taken as a fixed constant independent of $t$. Here we modify it to a piecewise function $b(t)$ for small $t$ to adapt the argument to the diffusion regime.} in \cite{MR4474532} by defining a function $b(t)$ by
\begin{equation*}
b(t) :=
    \begin{cases}
        1-n^{-0.06} & \text{if }  0\leq t\leq n^{-0.2}, \\
        1-b_0 & \text{if } t\geq n^{-0.2},
    \end{cases}
\end{equation*}
where $b_0$ is  some small positive constant.  We then define \begin{equation*}
	\begin{split}
    	L(t)&:=\sum_{j=1}^{\bar  X_0}D(j,t)\1[\bar S(j,t)=1, B(j,t)>0],\\
		X(I,b,t)&:=\sum_{i=1}^{\bar X_0}\1[A(i,t)\leq b(t)t, \bar I(i,t)=1] \leq \bar X_{I,t}, \\
		L(b,t)&:=\sum_{j=1}^{\bar X_0}D(j,t)
  \1[\bar S(j,t)=1, B(j,t)>b(t)t] \leq L(t).
	\end{split}
\end{equation*}
Using these quantities, the error term $E_t$ is bounded above by
\begin{equation*}
	E_t\leq \frac{X(I,b,t)}{\bar X_{I,t}}L(t)+L(b,t).
\end{equation*}
We then use the following lemma to control $E_t$, whose proof is similar to \cite[Lemmas 4.5 and 4.6]{MR4474532} and is therefore omitted. 

\begin{lemma}\label{ctlet2}
	There exists a constant $C_{\ref{lbtctl}}>0$ such that 	for any $b_0\in (0,1)$ and any $\ep>0$, 
	with probability at least $(1-n^{-0.4})$, for all $qn^{-1/3}\leq t\leq \gamma_n\wedge 1$,
	\begin{equation}\label{lbtctl}
    L(t)\leq  C_{\ref{lbtctl}} (nt+n^{-0.4}),\quad \text{and} \quad  L(b,t)\leq C_{\ref{lbtctl}}(1-b(t))^{1/2}tn.
	\end{equation}
    Moreover, there exist two positive constants $C_{\ref{Utdef}}$ and $c_{\ref{Utdef}}$	such that with probability at least $(1-n^{-0.4})$,  for all $qn^{-1/3}\leq t\leq \gamma_n\wedge 1$, $X(I,b,t)$ is bounded by
    \begin{equation}\label{Utdef}
     C_{\ref{Utdef}}n 
    		\left( t^3+ t^2\left( \exp\left\{ -\frac{c_{\ref{Utdef}}(1-b(t))t}{t^2+n^{-0.4}} \right\} +(1-b(t)) \right) + t n^{-0.4} + n^{0.02}\sqrt{(1-b(t))t/n}\right).
    \end{equation}
\end{lemma}

With the bounds in Lemma \ref{ctlet2} established, we can apply the identical proofs for Theorems \ref{phase2} and \ref{survivalcond}---along with a contradiction argument as in \cite[Step 2, Section 4.5]{MR4474532}---to show that these two results also hold for the AB-avoSI process.

Once the analysis of Stages 1–3 is completed as outlined above, we deduce Theorem \ref{thm:avo} by following the same proof as in Section \ref{sec:pfthm1}. The main result, Theorem \ref{surviprob}, then follows from Theorems \ref{thm:AB} and \ref{thm:avo} as desired.

%

\section*{Acknowledgement}
We thank Prof$.$ Rick Durrett for suggesting this problem and reading an earlier version of this paper. We also thank Prof$.$ Hui Xiao for helpful discussions on Lemma \ref{lemcondiy}. Haojie Hou and Dong Yao would like to express their gratitude to Prof$.$ Zipeng Wang and Prof$.$ Xinxin Chen for their great organization  and generous hospitality of a probability workshop at the College of Mathematics and Statistics, Chongqing University. Wenze Chen is supported by the Fundamental Research Funds for the Central Universities (No$.$ CUSF-DH-T-2025035). Haojie Hou is supported by the China Postdoctoral Science Foundation (No$.$ 2024M764112). Dong Yao is supported by National Key R\&D Program of China (No$.$ 2023YFA1010101) and NSFC grant (No$.$ 12571161). 

\appendix \section{Proof of some technical lemmas}\label{sec:tech}
\subsection{Two concentration inequalities in probability}\label{sec:prelim}

We begin by recalling Bernstein's inequality for subexponential random variables. A random variable $X$ is said to be \emph{subexponential} if and only if $\P(\abs{X}>r)\leq C_1\exp(-C_2r)$ for some constants $C_1,C_2>0$ and all $r>0$. The  \emph{subexponential norm} of $X$ is then defined to be
\begin{equation*}
    \norm{X}_{\operatorname{subexp}}:=\inf\{r>0:\E(\exp(\abs{X}/r))\leq 2\}.
\end{equation*}
Bernstein's inequality for i.i.d. random variables (cf$.$ \cite[Theorem 2.8.1]{MR3837109}) is stated as follows. 

\begin{lemma}\label{bernstein} 
    Let $\{X_i, 1 \leq i \leq n\}$ be i.i.d$.$ mean-zero subexponential random variables. Then there exists an absolute constant $c>0$ such that for all $m>0$,
    \begin{equation*}
        \P\left(\abs{\sum_{i=1}^n X_i}\geq m\right)\leq 2\exp\left(-c\min\left\{\frac{m^2}{n\norm{X_1}_{\operatorname{subexp}}^2}, \frac{m}{\norm{X_1}_{\operatorname{subexp}}} \right\}  \right).
    \end{equation*}
    As a consequence, if $m=n^a$ for some $a\in (1/2,1)$, then
    \begin{equation*}
        \P\left(\abs{\sum_{i=1}^n X_i}>n^a\right)\leq 2\exp\left(-c'n^{2a-1}\right),
    \end{equation*}
    for some constant $c'>0$ that depends only on $\norm{X_1}_{\operatorname{subexp}}$.
\end{lemma}


Another tool that we have used in this paper is the following concentration inequality for Bernoulli random variable, which can be found in \cite[Theorem 2.5]{MR2248695}.

\begin{lemma}\label{chernoff}
    Let $\{Y_i, 1 \leq i \leq n\}$ be independent Bernoulli random variables, and let $\mu := \sum_{i=1}^n\E(Y_i)$. Then for any $u>0$,  
    \begin{equation*}
        \P\left(\abs{\sum_{i=1}^n Y_i- \mu}\geq u\right)\leq \exp\left(-\frac{u^2}{2\mu}\right)+\exp\left(-\frac{u^2}{2(\mu+u/3)}\right)\leq 2\exp\left(-\min\left\{ \frac{u^2}{3\mu},\frac{u}{3}\right\}\right).
    \end{equation*}
\end{lemma}

\subsection{Proof of Lemma \ref{compbase}}\label{subsec:compbase} 

\begin{proof}[Proof of Lemma \ref{compbase}]
    To simplify notation, we denote by $\Omega_{\ref{eq:rejc}}$ the final event in the definition of $\Omega_{\ref{u1xit}}$, i.e.,
    \begin{equation}\label{eq:rejc}
        \Omega_{\ref{eq:rejc}}:= \left( \Omega_{\ref{jctl}} \cap \{\gamma_n<qn^{-1/3}\} \right) \cup \left( \widehat{\Omega}_{\ref{jctl}} \cap \{\gamma_n\geq qn^{-1/3}\}  \right).
    \end{equation}    
    We also define 
    \begin{equation}\label{4half}
        \Omega_{\ref{4half}} :=\left \{\text{Every vertex receives no more than 4 rewired half-edges by time } qn^{-1/3} \right \}.
    \end{equation}
    Note that $\P(\Omega_{\ref{4half}}^c \cap \Omega_{\ref{eq:rejc}})$ is bounded by 
    \begin{equation*}
        \mathbb{E}\left[ \left( \text{number of vertices receiving at least 5 rewired half-edges by } qn^{-1/3} \right) \1[\Omega_{\ref{jctl}}] \right].
    \end{equation*}
     This quantity is in turn bounded by
    \begin{equation*}
        \binom{\lfloor N_q+n^{0.6} \rfloor }{5} n^{-5}n= O(n^{-2/3}).
    \end{equation*}
    Here we have used a comparison with  $\lfloor N_q+n^{0.6}\rfloor $  independent Bernoulli trials, each with success probability $1/n$.  
    Define the event 
    \begin{equation}\label{half2}
        \Omega_{\ref{half2}} :=  \Omega_{\ref{eq:rejc}} \cap \Omega_{\ref{4half}}.
    \end{equation}
    By Lemma \ref {lemjump}, there exists a constant $C_{\ref{4j}}>0$ such that 
    \begin{equation}\label{4j}
        \P(  \Omega_{\ref{half2}})  \geq 1-\P(\Omega_{\ref{4half}}^c\cap \Omega_{\ref{eq:rejc}})-\P(\Omega_{\ref{jctl}}^c) -\P(\{\gamma_n\geq qn^{-1/3}\}\cap \widehat{\Omega}_{\ref{jctl}}^c) \geq 1-C_{\ref{4j}} n^{-2/3}.
    \end{equation}
    Recall that at each jump time, either $X_t$ remains unchanged or is replaced by $(X_t-2)$. On $  \Omega_{\ref{half2}}$, for all $t\leq qn^{-/3}$, the number of half-edges of vertex $i$ at time $t$ is at most $ (d_{i}+4)$, and 
    \begin{equation}\label{lbxt1}
        \X_t \geq X_0- 2 (N_q+n^{0.6}) \geq  m_1n-n^{C_{\ref{concern1}}}-2 (N_q+n^{0.6}) \geq m_1 n-3N_q \geq m_1n/2.
    \end{equation}
    Recall that $S_{t,k}$ denotes the number of susceptible vertices with $k$ half-edges at time $t$. On $\Omega_{\ref{half2}}$, using the maximal degree bound from \eqref{dmaxbd}, we have
    \begin{equation}\label{zeronskt}
        S_{t,k}=0, \quad \forall k\geq v_n \text{ and } t\leq qn^{-1/3}. 
    \end{equation}

  Let $C_{\ref{sumwik}}$  be some (large) constant to be determined later, and define the event
    \begin{equation}\label{sumwik}
        \Omega_{\ref{sumwik}}:=\left\{ 
        \abs{S_{t,k}-np_{k,n}}
        \leq \frac{C_{\ref{sumwik}}}{k+1}qn^{2/3} \exp(-   \eta _ {\ref{exp1}}k/2), \quad  \forall 0\leq k\leq v_n \right\}.
    \end{equation}
    
\begin{lemma}\label{lem:wikct}
    Fixing a large enough $C_{\ref{sumwik}}>0$, we have
    \begin{equation}\label{n2}
    \P\left(  \Omega_{\ref{half2}} \cap     \Omega_{\ref{sumwik}}^c  \right)\leq 2\exp(-\Theta(n^{2/3-0.66})).
    \end{equation}
\end{lemma}

    We temporarily assume Lemma \ref{lem:wikct} holds and fix the constant $C_{\ref{sumwik}}$ as specified therein. We claim that
    \begin{equation}\label{claim}
        \Omega_{\ref{sumwik}}\cap \Omega_{\ref{half2}}\subset \Omega_{\ref{u1xit}},
    \end{equation}
    which holds true provided the constant  $C_{\ref{u1xit}}\geq C_{\ref{sumwik}}$ (as in the definition of $\Omega_{\ref{u1xit}}$) is large enough.  Lemma \ref{compbase} now follows by combining Lemma \ref{lem:wikct} and the claim \eqref{claim}.
    \end{proof}

\begin{proof}[Proof of Claim \eqref{claim}]
    By combining \eqref{lbxt1} with the definition of $\Omega_{\ref{half2}}$, for all $t\leq qn^{-1/3}$, we have
    \begin{equation}\label{lowerxt}
        X_t-1 \geq m_1 n-n^{C_{\ref{concern1}}}-2(N_q+n^{0.6})-1\geq m_1 n-3N_q,
    \end{equation}
    where the second inequality holds because $C_{\ref{concern1}} < 2/3$. Similarly, we obtain the following upper bound 
    \begin{equation}\label{upperxt}
        X_t-1 \leq X_0 \leq m_1 n +n^{C_{\ref{concern1}}} \leq m_1n +3N_q.
    \end{equation}
    Since $I_t$  can increase by at most 1 when a jump occurs, it  has the following bound on $\Omega_{\ref{half2}}$: 
    \begin{equation}\label{upperit0}
        I_t\leq 1+ N_q+n^{0.6}\leq 2N_q. 
    \end{equation}
    From the definitions of $\Omega_{\ref{sumwik}}$ and $q_{k,n}$ (given in \eqref{eq:qkndef}), together with the condition $C_{\ref{sumwik}} \leq C_{\ref{u1xit}}$, it follows that
    \begin{equation}\label{ratiobdd}
        | S_{t,k}- np_{k,n} |\leq \frac{C_{\ref{u1xit}}}{k+1}qn^{2/3} \exp(- \eta_{\ref{exp1}}k/2),\quad \forall k\leq v_n \text{ and } t\leq qn^{-1/3}.
    \end{equation}
     Finally, for $X_{S,t}$ (the number of susceptible half-edges at time $t$),  on $\Omega_{\ref{sumwik}}\cap \Omega_{\ref{half2}}$,
    \begin{equation}\label{upperxit}
        \begin{split}
            \X_{I,t} &=\X_t-\X_{S,t} \leq  \X_0-\X_{S,t} = \sum_{k=0}^{\infty} (knp_{k,n}-kS_{t,k}) \\
            &\leq \sum_{k\geq v_n}nkp_{k,n}+ \sum_{0\leq k\leq v_n}\abs{knp_{k,n}-kS_{t,k}}\\
            &\leq C_{\ref{exp1}}\sum_{k\geq v_n}nk\exp(- \eta_ {\ref{exp1}}k) + \sum_{0\leq k\leq v_n} C_{\ref{sumwik}}qn^{2/3} \exp(- \eta _ {\ref{exp1}}k/2) \\  
            &\leq C_{\ref{u1xit}} N_q,
        \end{split}
    \end{equation}
    where the third line uses Remark \ref{rem:bdpk}, and the final inequality holds for all large $C_{\ref{u1xit}}$.
     The claim \eqref{claim} follows from \eqref{zeronskt} and \eqref{lowerxt}--\eqref{upperxit}.

\end{proof}

\begin{proof}[Proof of Lemma \ref{lem:wikct}]
    On the event $\Omega_{\ref{half2}}$, for $t\leq qn^{-1/3}$,
    \begin{equation}\label{bdnskt1}
        S_{t,k}\leq \sum_{j=0}^4 S_{0, k+j} =n\sum_{j=0}^4 p_{k+j,n}. 
    \end{equation}
    The changes in $S_{t,k}$ arise from three distinct factors, as detailed below:
    \begin{itemize}
        \item A susceptible vertex with $k$ half-edges receives a rewired half-edge, which decreases $S_{t,k}$ by 1 and occurs with probability at most $\left( \sum_{j=0}^4 p_{k+j,n} \right)$ in each jump. Indeed, since each rewiring is uniform over all vertices, the probability that an infected half-edge is rewired to a susceptible vertex with $k$ half-edges is $S_{t,k}/n$, which implies the desired upper bound by \eqref{bdnskt1}.
        \item A susceptible vertex with $(k-1)$ half-edges receives a rewired half-edge, which increases $S_{t,k}$ by 1 and occurs with probability at most $\left( \sum_{j=0}^4 p_{k-1+j,n} \right)$ in each jump. The analysis is identical to that in the previous case, with $S_{t,k}$ replaced by $S_{t,k-1}$.
        \item A susceptible vertex with $k$ half-edges loses a half-edge due to a pairing, which decreases $S_{t,k}$ by 1 and occurs with probability at most 
        $$
        \frac{k S_{t,k}}{X_t-1}\leq \frac{kS_{t,k}}{m_1n/2} \leq \frac{2k}{m_1}\sum_{j=0}^4 p_{k+j,n}
        $$ 
        in each jump. Here, the lower bound $m_1n/2$ for the denominator $(X_t-1)$ is derived from \eqref{lbxt1}, and we also use the fact that each pairing is uniform over all half-edges.
    \end{itemize}
    Combining these four factors, we see that for some large constant $C_{\ref{-55}}$, 
    the probability that $S_{t,k}$ changes by 1 in each jump is uniformly bounded by 
    \begin{equation}\label{-55}
        C_{\ref{-55}}(k+1)\sum_{j=-1}^4 p_{k+j,n}.
    \end{equation}
    For a fixed integer $k$, let $(W_{i,k})_{1\leq i\leq N_q+n^{0.6}}$ be i.i.d$.$ Bernoulli random variables with mean 
    \begin{equation}\label{p}
        b_{k}:=\min\left\{ C_{\ref{-55}}(k+1)\sum_{j=-1}^{4}p_{k+j,n}+n^{-0.66},1\right\}.
    \end{equation}
    We can couple $S_{t,k}$ and the sequence $(W_{i,k})$ on a common probability space
    \begin{equation}\label{n1}
        \abs{S_{t,k}-S_{0,k}}  =\abs{S_{t,k}-np_{k,n} } \leq \sum_{i=1}^{N_q+n^{0.6}} W_{i,k},\quad \forall t\leq qn^{-1/3}, \quad \mbox{ on } \Omega_{\ref{half2}}.
    \end{equation} 
    Since $(N_q + n^{0.6})b_k \geq Cn^{2/3 - 0.66}$ for large $n$ and some $C>0$, we apply Lemma \ref{chernoff} to obtain
    \begin{equation}\label{eq:wikbd}
    \P\left(\sum_{i=1}^{N_q+n^{0.6}}W_{i,k}\geq 2(N_q+n^{0.6})b_{k} \right)\leq 2\exp(-\Theta (n^{2/3-0.66})).
    \end{equation}
    Using the bound \eqref{bdpkn} for $p_{k,n}$, there exists a constant $C_{\ref{sumwik}}>0$ such that for all $0\leq k\leq v_n$,
    \begin{equation}\label{bdbk3}
        \begin{split}
            2(k+1)(N_q+n^{0.6})b_{k} \leq & Cqn^{2/3}((k+1)^2\exp(-   \eta _ {\ref{exp1}}k)+(k+1)n^{-0.66})\\
            \leq &C (qn^{2/3}\exp(-   \eta _ {\ref{exp1}}k/2)+n^{0.01}) \\
            \leq & C_{\ref{sumwik}}qn^{2/3} \exp(-   \eta _ {\ref{exp1}}k/2).
        \end{split}
    \end{equation}
    Lemma \ref{lem:wikct}  now follows from \eqref{n1}--\eqref{bdbk3} and the union bound.
\end{proof}

\subsection{Proof of Lemma \ref{yzmean}}\label{subsec:yzmean}  
    
\begin{proof}[Proof of Lemma \ref{yzmean}]
    We will prove the mean and variance asymptotics in two steps.

\emph{Step 1.} 
    For the convergence \eqref{limy} and \eqref{limz} of the means, we only give details for the proof of \eqref{limz}  since  \eqref{limy} is similar. Define 
    \begin{equation}\label{hatq}
        \widehat{q}_{k,n}:=  \min\{nkp_{k,n},2q_{k,n} \}.
    \end{equation}
    The definition of the random walk $(Z_{\ell})$ implies that
    \begin{equation}\label{zlzl-1}
        \begin{split}
            & \E(Z_{\ell}-Z_{\ell-1})\\
            =& \sum_{k=2}^{\log(nC_{\ref{exp1}})/   \eta _ {\ref{exp1}}} (k-2)\frac{knp_{k,n}-\widehat{q}_{k,n} }{(1+\rho/\lambda)(m_1n+3N_q)} -\frac{2C_{\ref{u1xit}}N_q }{(1+\rho/\lambda)(m_1 n-3N_q)}\\
            &-\left(1-\sum_{k=2}^{\log(nC_{\ref{exp1}})/   \eta _ {\ref{exp1}}}\frac{knp_{k,n}-\widehat{q}_{k,n} }{(1+\rho/\lambda)(m_1n+3N_q)} - \frac{C_{\ref{u1xit}}N_q }{(1+\rho/\lambda)(m_1n-3N_q) }  \right)\\
            =&\sum_{k=2}^{\log(nC_{\ref{exp1}})/ \eta _ {\ref{exp1}}} \frac{k(k-1)p_{k,n}n}{(1+\rho/\lambda)(m_1 n+3N_q) }-1 - \frac{C_{\ref{u1xit}}N_q }{(1+\rho/\lambda)(m_1n-3N_q) }\\
            & - \sum_{k=2}^{\log(nC_{\ref{exp1}})/ \eta_{\ref{exp1}}} \frac{(k-1)\widehat{q}_{k,n}}{(1+\rho/\lambda)(m_1n+3N_q)}.
        \end{split}
    \end{equation}
    Since we are at the critical infection rate, it holds that 
    $$
    \frac{\rho}{\lambda}=\frac{m_2-2m_1}{m_1},
    $$
    which implies that
    \begin{equation}\label{eq:kp*}
     \sum_{k=0}^{\infty}k(k-1)p_k^*=m_2-m_1=(1+\rho/\lambda)m_1.
    \end{equation}
    By \eqref{bdpk}, we have the bound
    \begin{equation}\label{eq:tailpk*}
        \sum_{k\geq \log(nC_{\ref{exp1}})/   \eta _ {\ref{exp1}}} k(k-1)p_k^*\leq  C_{\ref{exp1}}
        \sum_{k\geq \log(nC_{\ref{exp1}})/   \eta _ {\ref{exp1}}} 
        k(k-1)\exp(-   \eta _ {\ref{exp1}}k)= O\left(\frac{\log^2 n}{n} \right).
    \end{equation}
    Therefore, using  Assumption {\bf (H3)} and \eqref{eq:kp*}, we get that 
    \begin{equation}\label{limz0}
        \begin{split}
            &\sum_{k=2}^{\log(nC_{\ref{exp1}})/   \eta _ {\ref{exp1}}}
            \frac{k(k-1)p_{k,n}n}{(1+\rho/\lambda)(m_1 n+3N_q) }-1\\
            =&
            \sum_{k=2}^{\log(nC_{\ref{exp1}})/   \eta _ {\ref{exp1}}}
            \frac{k(k-1)p^*_kn}{(1+\rho/\lambda)(m_1 n+3N_q) }+O(n^{C_{\ref{concern1}}-1})-1\\
            =&\frac{m_1n}{m_1 n+3N_q}-1+o(n^{-1/3})
            =-\frac{3N_q}{m_1 n}+o(n^{-1/3}),
        \end{split}
    \end{equation}
    where we have also used the following expansions in the last step,
    \begin{equation}
        \begin{split}
            &(m_1n-3N_q)^{-1}=(m_1n)^{-1}(1+3(1+\rho/\lambda)qn^{-1/3}+O(n^{-2/3})),\\
            &(m_1n+3N_q)^{-1}=(m_1n)^{-1}(1-3(1+\rho/\lambda)qn^{-1/3}+O(n^{-2/3})) \label{expansion2}.
        \end{split}
    \end{equation}
    We now consider the last term in \eqref{zlzl-1} multiplied by $n^{1/3}$, i.e., 
    \begin{equation*}
        \sum_{k=2}^{\log(nC_{\ref{exp1}})/ \eta _ {\ref{exp1}}} n^{1/3}\frac{(k-1)\widehat{q}_{k,n}}{(1+\rho/\lambda)(m_1n+3N_q)}.
    \end{equation*}
    By the definitions of $q_{k,n}$ and $\widehat{q}_{k,n}$ in \eqref{eq:qkndef} and \eqref{hatq}, there exists some constant $C_{\ref{sumkqk}}>0$ such that for any fixed $k\geq 2$,
    \begin{equation}\label{sumkqk}
        \lim_{n\to\infty}n^{1/3} \frac{(k-1)\widehat{q}_{k,n}}{(1+\rho/\lambda)(m_1n+3N_q)}=\1[ p^*_k>0 ]C_{\ref{sumkqk}}q(k-1)\exp(-\eta_{\ref{exp1}}k/2).
    \end{equation}
    We can also give an upper bound for the tail of the sum,
    \begin{equation}\label{largeN}
        \begin{split}
            & \lim_{N\to\infty}\limsup_{n\to\infty}  \sum_{k=N}^{\log(nC_{\ref{exp1}})/   \eta _ {\ref{exp1}}} n^{1/3} \frac{(k-1)\widehat{q}_{k,n}}{(1+\rho/\lambda)(m_1n+3N_q)}\\
            \leq & \lim_{N\to\infty} \limsup_{n\to\infty} 2 C_{\ref{u1xit}}q \sum_{k\geq N} \frac{n (k-1)\exp(-   \eta _ {\ref{exp1}}k/2)}{m_1 n}\,\,=0. 
        \end{split}
    \end{equation}
    Combining \eqref{sumkqk} and \eqref{largeN}, we conclude that 
    \begin{equation}\label{limz2}
        \begin{split}
            &  \lim_{n\to\infty}n^{1/3} \sum_{k=2}^{\log(nC_{\ref{exp1}})/   \eta_{\ref{exp1}}} \frac{(k-1)\widehat{q}_{k,n}}{(1+\rho/\lambda)(m_1n+3N_q)}\\
             =&C_{\ref{sumkqk}}\sum_{k\geq 2} \1[p^*_k>0]  q(k-1)\exp(-   \eta_{\ref{exp1}}k/2). 
        \end{split}
    \end{equation}
    Equation \eqref{limz} then follows from \eqref{zlzl-1}, \eqref{limz0} and \eqref{limz2}. 

\emph{Step 2.} We now consider the variance part \eqref{y2z2moment}. We only prove the case of $(Z_\ell)$ here. Notice that by Assumption {\bf (H3)},
    \begin{align*}
         \E((Z_{\ell}-Z_{\ell-1})^2) =& \sum_{k=2}^{\log(nC_{\ref{exp1}})/   \eta _ {\ref{exp1}}}
         (k-2)^2\frac{knp_{k,n}-\widehat{q}_{k,n} }{(1+\rho/\lambda)(m_1n+3N_q)} +\frac{4C_{\ref{u1xit}}N_q }{(1+\rho/\lambda)(m_1 n-3N_q)}\\
         &+\left(1-\sum_{k=2}^{\log(nC_{\ref{exp1}})/   \eta _ {\ref{exp1}}}\frac{knp_{k,n}-\widehat{q}_{k,n} }{(1+\rho/\lambda)(m_1n+3N_q)} - \frac{C_{\ref{u1xit}}N_q }{(1+\rho/\lambda)(m_1n-3N_q) }  \right)\\
          = &\sum_{k=1}^{\log(nC_{\ref{exp1}})/   \eta _ {\ref{exp1}}}
         [(k-2)^2-1]\frac{knp_{k}^*  }{(1+\rho/\lambda)(m_1n+3N_q)} +1+O(n^{-1/3}),
     \end{align*}
   where we have applied the following estimate in the last equality,
   \begin{equation*}
\sum_{k=1}^{\infty}
             k^2\frac{\widehat{q}_{k,n} }{(1+\rho/\lambda)(m_1n+3N_q)} 
\leq              \frac{2C_{\ref{u1xit}}q n^{-1/3}}{(1+\rho/\lambda)m_1} \sum_{k=1}^{\infty}
             k^2\exp(-   \eta _ {\ref{exp1}}k/2) = O(n^{-1/3}).
   \end{equation*}
    By \eqref{bdpk} again, we get, analogously  to \eqref{eq:tailpk*},
    $$
     \sum_{k\geq \log(nC_{\ref{exp1}})/   \eta _ {\ref{exp1}}} k(k-2)^2 p_k^*\leq  C_{\ref{exp1}}
            \sum_{k\geq \log(nC_{\ref{exp1}})/   \eta _ {\ref{exp1}}} 
            k(k-2)^2\exp(-   \eta _ {\ref{exp1}}k) = O \left( \frac{\log^3 n}{n}\right).
    $$
     Recall the definition of $\sigma^2$ (cf. \eqref {defN_q}) for the process, we see that $\E(Z_{\ell}-Z_{\ell-1})^2 $ equals
    \begin{align*}
        & \sum_{k=1}^{\infty}
             (k-2)^2 \frac{knp_{k}^*  }{(1+\rho/\lambda)m_1n} +
            1-\sum_{k=1}^{\infty}
            \frac{knp_{k}^*}{(1+\rho/\lambda)m_1n } +O(n^{-1/3})\\
             = &\left( \sigma^2 -\frac{\rho}{\rho+\lambda}  \right) + 1 - \frac{m_1}{(1+\rho/\lambda) m_1}  + O(n^{-1/3}) \\
             =&\sigma^2+  O(n^{-1/3}).
    \end{align*}
This completes the proof of    \eqref{y2z2moment} and also Lemma \ref{yzmean}.

    \end{proof} 

\subsection{Proof of Lemma \ref{lemcondiy}}\label{Donskerinv}

To prove Lemma \ref{lemcondiy}, we first treat the centered case (Lemma \ref{lem-cond-centered} below) and then use a change of measure argument. 


\begin{lemma}\label{lem-cond-centered}
    Let $\{X_{i,n}, 1\leq i\leq n \}_{n\geq 1}$ be i.i.d.\@ mean-zero $\mathbb{Z}$-valued random variables such that
    \begin{equation*}
        \mathbb{P}(X_{1,n}\geq -2)=1, \quad \limsup_{n\to\infty} \sqrt{n}(\mathbb{E}(X_{1,n}^2) -\eta^2) <\infty, \quad \text{and} \quad \sup_{n} \mathbb{E}(|X_{1,n}|^3)<\infty.
    \end{equation*}
    For any $y \ge 0$, define the random walk starting at some integer $y$ by $V_{0,n}(y) := y$ and $V_{i,n}(y) := y + \sum_{j=1}^i X_{j,n}$ for $1\leq i\leq n$. Then for any sequence $a_n = O(\log n)$,
    \begin{equation*}
        \lim_{n\to\infty} \sup_{K\in (0, \infty]}\sup_{0<y<a_n}\left| \mathbb{P}\left( V_{n,n}(y)\leq \eta \sqrt{n} K \mid \min_{ j\leq n} V_{j,n}(y) >0 \right) - \int_0^K ze^{-z^2/2} \mathrm{d}z\right| =0.
    \end{equation*}
\end{lemma}

\begin{proof}
    The proof presented herein is based on \cite[Theorem 2.9]{MR4863050}; its full details can be found in  Section 7.2 of  the arXiv version\footnote{Available at 
    \url{https://arxiv.org/abs/2110.05123}.
    } of \cite{MR4863050}. See also \cite[Section 4.1]{MR4990504}.

    As a first step,  note that the coupling argument in \cite[Lemma 3.6]{MR4863050} is still valid in our case (indeed, our case is covered by  \cite[Theorem A]{MR2302850}). 
    We then conclude\footnote{\cite[Lemma 3.6]{MR4863050} actually implies
    $   \P\left(\sup_{0\leq t\leq 1}\left| V_{[nt],n}(0) -(\eta+O(n^{-1/2})) B_{nt}\right|>n^{1/2 -\varepsilon} \right) \leq c_{\varepsilon}/n^{1/2 - 3\varepsilon}$. The $O(n^{-1/2})$ term can be removed thanks to the fact that $n^{-1/2}\sup_{0\leq t\leq 1}\abs{B_{nt}} $ has the same law of $\sup_{0\leq t\leq 1} \abs{B_t}$, which has an exponentially decaying tail.} from  \cite[Lemma 3.6]{MR4863050} (with $\delta=1$ and  $\sigma$ replaced by $\sqrt{\E(X_{1,n}^2)} = \eta + O(n^{-1/2})$) that for any $\varepsilon \in (0, 1/6)$, there exists some constant $c_\varepsilon>0$ (independent of $n$) such that 
    \begin{equation}\label{coupling}
    \P\left(\sup_{0\leq t\leq 1}\left| V_{[nt],n}(0) -\eta B_{nt}\right|>n^{1/2 -\varepsilon} \right) \leq \frac{c_{\varepsilon}}{n^{1/2 - 3\varepsilon}},\quad \forall n\geq 1.
    \end{equation}


   Consequently, \cite[Theorem 2.9]{MR4863050} also holds true for our case.
    Let  $R_n(y)$ be defined by $y- \E\left( V_{\tau_{0,n},n}(y)\right)$ where $\tau_{0,n}:=\inf\{k: V_{k,n}(y)\leq 0\}$, then $R_n(y)\in [y,y+1]$ under the assumption that $\P(X_{1,n}\geq -2)=1$. Then by \cite[Theorem 2.9]{MR4863050},
     uniformly for all $K\in (0, \infty], n\geq 1$ and $y\in (0, a_n)$, we have
    \begin{equation}
    \begin{split}
         &\mathbb{P}\left( V_{n,n}(y) \leq \eta \sqrt{n} K ,\  \min_{j\leq n} V_{j,n} >0  \right) \\ 
          \leq & c_{\varepsilon} \frac{1+y}{n^{1/2 +\varepsilon}} + \left(\mathbb{P}\left( \min_{j\leq n} V_{j,n} (y)>0  \right) + c_{\varepsilon} \frac{1+y}{n^{1/2 +\varepsilon}}\right) \cdot \int_0^K ze^{-z^2/2} \mathrm{d}z \\
          \leq & 2 c_{\varepsilon} \frac{1+y}{n^{1/2 +\varepsilon}}  + \mathbb{P}\left(   \min_{j\leq n} V_{j,n}(y) >0  \right) \cdot \int_0^K ze^{-z^2/2} \mathrm{d}z,
    \end{split}
    \end{equation}
    and moreover, 
    $$\P\left(  \min_{j\leq n} V_{j,n}(y) >0  \right) \geq \frac{2y}{\sigma \sqrt{2\pi n}}- c_{\varepsilon} \frac{1+y}{n^{1/2 +\varepsilon}} \geq \frac{y+1}{\sigma \sqrt{2\pi n}}.$$ 
    We conclude that, uniformly for all $K\in (0,\infty], n\geq 1, y\in (0,a_n)$,
     \begin{align*}
         &\mathbb{P}\left( V_{n,n}(y)\leq \eta \sqrt{n} K \mid \min_{j\leq n} V_{j,n}(y) >0  \right) \\ 
         \leq & 2 c_{\varepsilon} \frac{1+y}{n^{1/2 +\varepsilon}\mathbb{P}\left(   \min_{j\leq n} V_{j,n}(y) >0  \right)}  + \int_0^K ze^{-z^2/2} \mathrm{d}z \leq  2 c_{\varepsilon} \frac{\sigma \sqrt{2\pi}}{n^{\varepsilon}}  + \int_0^K ze^{-z^2/2} \mathrm{d}z.
    \end{align*}
    Repeating the same argument for the lower bound implies Lemma \ref{lem-cond-centered}.
\end{proof}

Now we are ready to prove Lemma \ref{lemcondiy}.
     
\begin{proof}[Proof of Lemma \ref{lemcondiy}]
    Set $N:= \lfloor N_q-n^{0.6}\rfloor $ for the ease of notation. 
    Recall  we have argued below Lemma \ref{lemcondiy} that 
    \[
     \lim_{N\to\infty} \left| \mathbb{P}\left( Y_N \leq \sigma \sqrt{N} K \mid \min_{j\leq N} Y_j >0 , Y_0=1 \right) - \P(\chi_q\leq K) \right| =0,
    \]
    where $\chi_q$ is the random variable defined in \eqref{chiqdef}.
    Thus, it suffices to prove that for any $K>0$ and $a_N= O(\log N)$, there exists some common constant  $G_K$ (which then has to be equal to $\P(\chi_q\leq K)$) such that 
     \begin{equation}\label{Goal}
         \lim_{N\to\infty} \sup_{0<y<a_N}\left| \mathbb{P}\left( Y_N \leq \sigma \sqrt{N} K \mid \min_{j\leq N} Y_j >0 , Y_0=y\right) - G_K \right| =0.
    \end{equation}
    Let $\Upsilon_N$ be the unique positive number satisfying 
    \[
    \E\left( (Y_\ell -Y_{\ell-1}) \exp\left(-\Upsilon_N N^{-1/2} (Y_\ell-Y_{\ell-1}) \right) \right) =0. 
    \]
    Combining Lemma \ref{yzmean}, Taylor's expansion, and the fact that $Y_{\ell}-Y_{\ell-1}= O(\log N)$, we see that 
    \begin{equation*}
    0=C_{\ref{limy}}q n^{-1/3} -\sigma^2 \Upsilon_N N^{-1/2} + o(n^{-1/3}). 
    \end{equation*}
    Solving the above equation yields that 
    \[
    \Upsilon_N = \frac{C_{\ref{limy}}q}{\sigma^2} \sqrt{(1+\rho/\lambda)m_1 q} +o(1)= \frac{C_{\ref{cqdef}}(q)}{\sigma}  + o(1) .
    \]
    Define a new probability measure $\P^*$ and corresponding expectation $\E^*$ by 
    \[
    \frac{\mathrm{d} \P^*}{\mathrm{d}\P}=\frac{\exp\left\{-\Upsilon_N N^{-1/2}(Y_\ell- Y_{0})\right\} }{\E\left(\exp\left\{-\Upsilon_N N^{-1/2}(Y_\ell- Y_{0})\right\} \right)},
    \]
    then $\E^*(Y_\ell-Y_{\ell-1})=0$ according to the definition of $\Upsilon_N$.
    It then follows from Lemma \ref{lem-cond-centered} that for any $K\in (0,\infty)$,
    \begin{align*}
          \lim_{N\to\infty} \sup_{0<y<a_N}\bigg| & \mathbb{E}^*\left(\exp\left(\Upsilon_N N^{-1/2}Y_N\right) \1[Y_N\leq \sigma \sqrt{N} K] \mid \min_{j\leq N} Y_j >0, Y_0=y  \right)\\
          &\qquad - \int_0^K z  e^{C_{\ref{cqdef}}(q) z-z^2/2} \mathrm{d}z\bigg| =0.
    \end{align*}
   According to the definition of $\mathbb{P}^*$, the above limit is equivalent to
   \begin{equation}\label{equ-cond}
        \begin{split}
          \lim_{N\to\infty} \sup_{0<y<a_N}\bigg| & \frac{\mathbb{P}\left( Y_N\leq \sigma \sqrt{N} K \mid \min_{j\leq N} Y_j >0, Y_0=y  \right) }{\mathbb{E}\left( e^{\Upsilon_N N^{-1/2}Y_N} \mid \min_{j\leq N} Y_j >0, Y_0=y  \right)}\\
          &\qquad - \int_0^K z  e^{C_{\ref{cqdef}}(q) z-z^2/2} \mathrm{d}z\bigg| =0.
        \end{split}
   \end{equation}
     Taking $K=1$ in \eqref{equ-cond}, we see that for some $N_*>0$, when $N>N_*$,
        \begin{align*}
             \frac{1}{2}\int_0^1 z  e^{C_{\ref{cqdef}}(q) z-z^2/2} \mathrm{d}z 
             & \leq \inf_{0<y<a_N} \frac{\mathbb{P}\left( Y_N\leq \sigma \sqrt{N}  \mid \min_{j\leq N} Y_j >0, Y_0=y  \right) }{\mathbb{E}\left( e^{\Upsilon_N N^{-1/2}Y_N} \mid \min_{j\leq N} Y_j >0, Y_0=y  \right)} \\
             & \leq \frac{1 }{ \sup_{0<y<a_N}\mathbb{E}\left( e^{\Upsilon_N N^{-1/2}Y_N} \mid \min_{j\leq N} Y_j >0, Y_0=y  \right)}.
     \end{align*}
     Therefore, by Markov's inequality, it holds that 
     \begin{equation}\label{Tightness-argument}
         \begin{split}
              \lim_{K\to\infty}\sup_{N>N_*}  \sup_{0<y<a_N} \mathbb{P}\left( Y_N> \sigma \sqrt{N} K \mid \min_{j\leq N} Y_j >0, Y_0=y  \right) =0.
         \end{split}
     \end{equation}
   Now for any fixed $K_1, K_2\in (0,\infty)$, by \eqref{equ-cond}, 
   \begin{equation}\label{eq:k1k20}
       \lim_{N\to\infty} \sup_{0<y<a_N}\left|  \frac{\mathbb{P}\left( Y_N\leq \sigma \sqrt{N} K_1 \mid \min_{j\leq N} Y_j >0, Y_0=y  \right) }{\mathbb{P}\left( Y_N\leq \sigma \sqrt{N} K_2 \mid \min_{j\leq N} Y_j >0, Y_0=y  \right) } - \frac{\int_0^{K_1} z  e^{C_{\ref{cqdef}}(q) z-z^2/2} \mathrm{d}z}{\int_0^{K_2} z  e^{C_{\ref{cqdef}}(q) z-z^2/2} \mathrm{d}z}\right| =0.
   \end{equation}
    For any $\varepsilon>0$,  let $K_2=K_2(\varepsilon)$ be the constant such that the probability in \eqref{Tightness-argument} is bounded above by $\varepsilon$, i.e., 
    $$
    \sup_{N>N_*}  \sup_{0<y<a_N} \mathbb{P}\left( Y_N> \sigma \sqrt{N} K \mid \min_{j\leq N} Y_j >0, Y_0=y  \right)\leq \varepsilon,
    $$
    and that $$\int_{K_2}^{\infty} z  e^{C_{\ref{cqdef}}(q) z-z^2/2} \mathrm{d}z<\varepsilon. $$  
    For large enough $N$, by \eqref{eq:k1k20},
    \begin{align*}
        &\mathbb{P}\left( Y_N\leq \sigma \sqrt{N} K_1 \mid \min_{j\leq N} Y_j >0, Y_0=y  \right)  \\  
        \leq & \frac{\mathbb{P}\left( Y_N\leq \sigma \sqrt{N} K_1 \mid \min_{j\leq N} Y_j >0, Y_0=y  \right) }{\mathbb{P}\left( Y_N\leq \sigma \sqrt{N} K_2 \mid \min_{j\leq N} Y_j >0, Y_0=y  \right) } \\
         \leq & \varepsilon+ \frac{\int_0^{K_1} z  e^{C_{\ref{cqdef}}(q) z-z^2/2} \mathrm{d}z}{\int_0^{K_2} z  e^{C_{\ref{cqdef}}(q) z-z^2/2} \mathrm{d}z} \leq \varepsilon+\frac{\int_0^{K_1} z  e^{C_{\ref{cqdef}}(q) z-z^2/2} \mathrm{d}z}{\int_0^{\infty} z  e^{C_{\ref{cqdef}}(q) z-z^2/2} \mathrm{d}z -\varepsilon},
    \end{align*}
    which leads to (by sending $N\to\infty$ first and then $\varepsilon \to 0$),
    $$
        \limsup_{N\to\infty}
        \sup_{0<y<a_N} \mathbb{P}\left( Y_N\leq \sigma \sqrt{N} K_1 \mid \min_{j\leq N} Y_j >0, Y_0=y  \right) 
        \leq \frac{\int_0^{K_1} z  e^{C_{\ref{cqdef}}(q) z-z^2/2} \mathrm{d}z}{\int_0^{\infty} z  e^{C_{\ref{cqdef}}(q) z-z^2/2} \mathrm{d}z}.
    $$
    We can also obtain a similar lower bound and conclude that 
     \begin{align*}
        \lim_{N\to\infty} \sup_{0<y<a_N}\bigg| & \mathbb{P}\left( Y_N\leq \sigma \sqrt{N} K_1 \mid \min_{j\leq N} Y_j >0, Y_0=y  \right) - \frac{\int_0^{K_1} z  e^{C_{\ref{cqdef}}(q) z-z^2/2} \mathrm{d}z}{\int_0^{\infty} z  e^{C_{\ref{cqdef}}(q) z-z^2/2} \mathrm{d}z}\bigg| =0,
   \end{align*}
   which implies \eqref{Goal} and this completes the proof of the lemma.
\end{proof}

\footnotesize
\bibliographystyle{plain}
\bibliography{ref}
\end{document}